\documentclass[11pt,a4paper]{article}
\usepackage[full]{textcomp}
\usepackage{newtxtext} 
\usepackage[utf8]{inputenc}

\usepackage{graphicx,amsmath,amssymb} 
\usepackage{mathtools}
\usepackage{enumitem}
\setlist{topsep=1pt,parsep=1pt,itemsep=1pt} 
\usepackage{tikz}
\usetikzlibrary{matrix,arrows}
\usetikzlibrary{decorations.pathmorphing}
\usetikzlibrary{decorations.markings}
\usepgflibrary{fpu}
\usepackage{xcolor}
\usepackage{mathrsfs} 
\usepackage{bbm}

\usepackage{cleveref}
\usepackage{amsthm}  
\usepackage{autonum} 

\usepackage{thmtools}

\input{irrev.sty}
\newcommand{\nn}{\mathbf{n}}
\newcommand{\ex}{\mathbf{e}_x}
\newcommand{\ey}{\mathbf{e}_y}

\newcommand{\A}{\cA}
\newcommand{\B}{\cB}
\newcommand{\AB}{\A\B}
\newcommand{\ABc}{(\A\cup \B)^c}
\newcommand{\Deltabar}{\bar\Delta}
\newcommand{\taujump}{\tau^{{\rm jump}}}
\newcommand{\rsi}{\varrho^2\sigma^2}
\newcommand{\rsii}{\frac{\varrho^2\sigma^2}{2}}

\newcommand{\Law}{\mathsf{Law}}

\newcommand{\NB}[1]{#1}

\newcommand{\avrg}[1]{\langle #1 \rangle}
\newcommand{\bigavrg}[1]{\bigl\langle #1 \bigr\rangle}

\title{An Eyring--Kramers law\\
for slowly oscillating bistable diffusions}
\author{Nils Berglund}
\date{}

\begin{document}

\maketitle

\begin{abstract}
We consider two-dimensional stochastic differential equations, describing the 
motion of a slowly and periodically forced overdamped particle in a  double-well 
potential, subjected to weak additive noise. We give sharp asymptotics of 
Eyring--Kramers type for the expected transition time from one potential well to 
the other one. Our results cover a range of forcing frequencies that are large 
with respect to the maximal transition rate between potential wells of the 
unforced system. The main difficulty of the analysis is that the forced system 
is non-reversible, so that standard methods from potential theory used to 
obtain Eyring--Kramers laws for reversible diffusions do not apply. Instead, we 
use results by Landim, Mariani and Seo that extend the potential-theoretic 
approach to non-reversible systems.  
\end{abstract}

\leftline{\small{\it Date.\/} Revised, May 31, 2021.}
\leftline{\small 2020 {\it Mathematical Subject Classification.\/} 
60H10, 		
34F05   	
(primary), 
60J45,		
81Q20		
(secondary)}
\noindent{\small{\it Keywords and phrases.\/}
Stochastic exit problem, 
diffusion exit,  
first-exit time, 
limit cycle, 
large deviations, 
potential theory, 
semiclassical analysis,
cycling, 
Gumbel distribution.}


\section{Introduction}
\label{sec:intro} 

This work is concerned with time-periodic perturbations of the stochastic 
differential equation (SDE)
\begin{equation}
\label{eq:SDE} 
 \6x_t = -V_0'(x_t)\6t + \sigma \6W_t\;, 
\end{equation} 
describing the overdamped motion of a Brownian particle in a double-well 
potential $V_0: \R\to\R$, which is bounded below and grows at least 
quadratically at infinity. 
\NB{Such periodically forced SDEs appear in many applications. For instance, in 
simple climate models, they have been used to provide an explanation for the 
regular appearance of glaciations in the last million years \cite{Nicolis,BSV}, 
via the mechanism called stochastic resonance. Though this explanation 
remains controversial for Ice Ages, models combining periodic forcing and 
stochastic perturbations are very common in climatology and ecology, see for 
instance~\cite{Velez01,Eisenman_Wettlaufer_08,ATW_21} for various examples. 
Periodically forced SDEs also appear in several other applications, 
including neuroscience and quantum electronics~\cite{WM}. We refer 
to~\cite{WJ,GHM,HanggiReview02} for further reviews on this topic.}

Let us start by recalling some well-known properties of the unperturbed 
system~\eqref{eq:SDE}. Its unique invariant measure has density 
$Z^{-1}\e^{-2V_0(x)/\sigma^2}$ with respect to Lebesgue measure, where $Z$ is 
the normalisation. Furthermore, the dynamics is reversible with respect to this 
measure. 
Denote the local minima of $V_0$ by $x^*_\pm$, and its local maximum by 
$x^*_0$, with $x^*_- < x^*_0 < x^*_+$ (see~\figref{fig:potential} below). Let 
$\tau_+ = \inf\setsuch{t>0}{x_t = x^*_+}$ be the first-hitting time of $x^*_+$. 
Then one has the explicit expression 
\begin{equation}
 \bigexpecin{x}{\tau_+} = 
 \frac2{\sigma^2} \int_x^{x^*_+} \int_{-\infty}^{x_2}  
\e^{2[V_0(x_2)-V_0(x_1)]/\sigma^2}\6x_1\6x_2
\end{equation} 
for the expectation of $\tau_+$ when starting at any $x<x^*_+$. This result is 
obtained by solving an ordinary differential equation (ODE) satisfied by the 
function $x\mapsto \bigexpecin{x}{\tau_+}$, owing to Dynkin's formula (which 
we recall in Appendix~\ref{app:pot} below). In 
particular, the Laplace method shows that when starting in $x^*_-$, this 
expectation satisfies the so-called Eyring--Kramers law~\cite{Eyring,Kramers}
\begin{equation}
\label{eq:EK} 
 \bigexpecin{x^*_-}{\tau_+} = 
 \frac{2\pi}{\sqrt{\abs{V_0''(x^*_0)}V_0''(x^*_-)}}
 \e^{2[V_0(x^*_0)-V_0(x^*_-)]/\sigma^2}
 \bigbrak{1 + \Order{\sigma^2}}\;.
\end{equation} 
Furthermore, in~\cite{Day1}, Day has shown that the law of $\tau_+$ is 
asymptotically exponential, in the sense that 
\begin{equation}
\label{eq:exp_law} 
 \lim_{\sigma\to 0}
 \bigprob{\tau_+ > s \, \expecin{x^*_-}{\tau_+}} = \e^{-s} 
\end{equation} 
holds for all $s>0$. 

While the expected transition time from $x^*_-$ to $x^*_+$ is exponentially 
long, the actual successful transition, also known as the reactive or 
transition path, takes much less time. In~\cite{CerouGuyaderLelievreMalrieu12}, 
C\'erou, Guyader, Leli\`evre and Malrieu have shown that for any fixed $a < x_0 
< x^*_0 < b$ in $(x^*_- ,  x^*_+)$, one has the convergence in law 
\begin{equation}
 \label{eq:reactive}
  \lim_{\sigma\to0} \Law \bigl( \abs{V_0''(x^*_0)} \tau_b - 2 \log(\sigma^{-1}) 
  \bigm| \tau_b < \tau_a \bigr)
 = \Law \Bigl( Z + T(x_0,b) \Bigr) \;,
\end{equation} 
where $T(x_0,b)$ is an explicit deterministic quantity independent of $\sigma$, 
and $Z$ is a standard Gumbel variable, that is, $\prob{Z \leqs t} = 
\exp\set{-\e^{-t}}$ holds for all $t\in\R$. Therefore, the duration of a 
transition is of order $\log(\sigma^{-1})$. See 
also Bakhtin's works~\cite{Bakhtin_2013a,Bakhtin_2014a} for insights on the 
relation of this result to extreme-value theory. 

Several of these results have been generalised to multidimensional diffusions 
of the form 
\begin{equation}
 \6x_t = -\nabla V_0(x_t)\6t + \sigma \6W_t\;, 
\end{equation} 
where now $V_0 : \R^d\to\R$. These are still reversible with respect to the 
invariant measure $Z^{-1}\e^{-2V_0/\sigma^2}$. \NB{Physical derivations of 
multidimensional generalisations of the Eyring--Kramers law~\eqref{eq:EK} have 
been obtained by Landauer and Swanson~\cite{Landauer_Swanson_61} and 
Langer~\cite{Langer_69}. Regarding mathematically rigorous results, a} weaker 
form of the Eyring--Kramers law (that is, without a sharp control of the 
prefactor of the exponential in~\eqref{eq:EK}), known as Arrhenius law, follows 
from the theory of large deviations developed for diffusions by Freidlin and 
Wentzell~\cite{FreidlinWentzell_book}. In~\cite{BEGK,BGK}, Bovier, Eckhoff, 
Gayrard and Klein used potential theory to prove a generalisation 
of~\eqref{eq:EK} to the multidimensional gradient case, as well as the 
asymptotically exponential character~\eqref{eq:exp_law} of the law of transition 
times. Similar results have been obtained by Helffer, Klein and Nier 
in~\cite{HelfferKleinNier04} using methods from semiclassical analysis. See 
also~\cite{LePeutrec_2010,LePeutrec_2011,LPNV_2013,Nectoux_LePeutrec_19} for 
generalisations to diffusions on manifolds with or without boundary. The 
potential-theoretic approach has also been successfully applied to obtain 
Eyring--Kramers laws for stochastic 
PDEs~\cite{BG12a,Barret15,Berglund_DiGesu_Weber_16}. See 
also~\cite{Berglund_Kramers,B_Sarajevo} and references therein, as well 
as~\cite{Bovier_denHollander_book} for a comprehensive account of the 
potential-theoretic approach. 

The situation is much less understood for non-gradient diffusions, whose 
invariant measure is not explicitly known in general, and which are not 
reversible. While the theory of large deviations in~\cite{FreidlinWentzell_book} 
allows to derive Arrhenius laws for these systems as well, determining precise 
asymptotics on transition times of Eyring--Kramers type is much harder than in 
the reversible case. Some partial results in this direction have nevertheless 
been obtained. In~\cite{Bouchet_Reygner_2016}, Bouchet and Reygner proposed an 
Eyring--Kramers law for non-reversible diffusions in a bistable situation, based 
on formal asymptotic computations. In~\cite{Landim_Mariani_Seo19}, Landim, 
Mariani and Seo obtained a generalisation of the potential-theoretic approach 
of~\cite{BEGK,BGK} to non-reversible systems. This allowed them in particular to 
justify the formal result of Bouchet and Reygner for a particular class of  
systems whose invariant measure is known explicitly. See also the 
work~\cite{LePeutrec_Michel_19} by Le Peutrec and Michel for semiclassical 
results on non-reversible diffusions with known invariant measure. In a 
different direction, a reactive path theory for multidimensional, non-reversible 
diffusions was developed by Lu and Nolen in~\cite{Lu_Nolen_2015}, based on 
ideas by E and Vanden-Eijnden~\cite{E_VandeEijnden_2006}. 

In this work, we are concerned with extensions of~\eqref{eq:EK} to systems of 
another type, namely to periodically perturbed versions of~\eqref{eq:SDE} of 
the form
\begin{align}
 \6x_t &= -\partial_x V_0(x_t,y_t)\6t + \sigma \6W^x_t\;, 
\\
 \6y_t &= \eps\6t + \sigma\sqrt{\eps}\varrho \6W^y_t\;,
\label{eq:SDE_fast-slow0} 
\end{align}
where $\set{W^x_t}_t$ and $\set{W^y_t}_t$ are independent standard Wiener 
processes. The parameter $\varrho$ has to be strictly positive for technical 
reasons (we need the diffusion to be elliptic), but our results do not depend on 
$\varrho$ to leading order. This system is a particular case of systems studied 
by the author and Barbara Gentz in~\cite{BG14}. The main result in 
that work gives a rather sharp description of the density of $\tau_0$, the 
first-passage time at the saddle $x^*_0(y)$ of $x\mapsto V_0(x,y)$ (or, more 
precisely, at the deterministic periodic solution tracking the saddle). A 
slightly less precise, but more transparent way of formulating this result
is that 
\begin{equation}
\label{eq:cycling} 
 \lim_{\sigma\to0} \Law \Bigl( \theta(y_{\tau_0}) - \log(\sigma^{-1}) -
\frac{\lambda_+}{\eps} Y^\sigma \Bigr)
 = \Law \biggl( \frac{Z}{2} - \frac{\log2}{2}\biggr) \;,
\end{equation} 
where 
\begin{itemize}
\item 	$\theta(y)$ is a convenient and explicit parametrisation of the 
periodic orbit tracking $x^*_0(y)$;
\item 	$\lambda_+$ is the Lyapunov exponent of this orbit;
\item 	$Z$ follows again a standard Gumbel law;
\item 	and $Y^\sigma$ is asymptotically geometric, meaning that it has 
positive integer values and satisfies 
\begin{equation}
 \lim_{n\to\infty} \pcond{Y^\sigma=n+1}{Y^\sigma>n} = p(\sigma)\;,
\end{equation} 
for a constant $p(\sigma)$ that is exponentially small in $\sigma^2$. 
\end{itemize}
(In fact, we have slighly simplified the precise result, which is given 
in~\cite[Theorem~4.2]{berglund_ihp14_mprf}.) The most striking feature 
of~\eqref{eq:cycling} is that the law of $\theta(y_{\tau_0})$ is shifted by an 
amount $\log(\sigma^{-1})$ as $\sigma$ decreases, and thus does not admit a 
limit as $\sigma\to0$. This is the phenomenon of cycling discovered by 
Day~\cite{Day7,Day3,Day6,Day4}. In fact, this shift by $\log(\sigma^{-1})$ is 
also present in~\eqref{eq:reactive}. As for $Y^\sigma$, its interpretation is 
as follows: under a non-degeneracy assumption, the system has a \lq\lq window 
of opportunity\rq\rq\ during each period to make a transition, which is defined 
by the minimisers of its large-deviation rate function. The integer variable 
$Y^\sigma$ simply gives the period during which the actual transition takes 
place. 

The expectation $\expec{\tau_0}$ can be deduced from~\eqref{eq:cycling}, and is 
close to the inverse of the parameter $p(\sigma)$ (see~\cite{BG9}). Since 
transitions from the saddle to the local minima $x^*_\pm(y)$ take a time of 
order $\log(\sigma^{-1})$ (see~\cite[Theorem~6.2]{berglund_ihp14_mprf}), the 
expectation of the first-hitting time $\tau_+$ of $x^*_+(y)$ has the same sharp
asymptotics as $\expec{\tau_0}$. In~\cite{BG14}, we did not attempt to obtain 
sharp asymptotics for $p(\sigma)$, but only showed that it is close, in the 
sense of logarithmic equivalence, to $\e^{-I/\sigma^2}$ where $I$ is the 
Freidlin--Wentzell quasipotential, which can be expressed as the solution of a 
variational principle.

The aim of the present work is to obtain sharp asymptotics of Eyring--Kramers 
type for $\expec{\tau_+}$, which is equivalent to getting precise asymptotics 
for $p(\sigma)$. \NB{To formulate our results, we introduce the notations 
\begin{equation}
 h_\pm(y) = V_0(x^*_0(y),y)-V_0(x^*_\pm(y),y)
\end{equation} 
for the depths of the two potential wells. Let $h_\pm^{\min}$ and $h_\pm^{\max}$ 
denote the minimal and maximal values reached by $h_\pm(y)$ during one period, 
and let $h^{\min} = \min\set{h_-^{\min},h_+^{\min}}$,  $h^{\max} = 
\max\set{h_-^{\max},h_+^{\max}}$. Then there are several parameter regimes, 
depending on the value of the forcing period compared to different time scales 
related to the noise; we employ here the terminology from the discussion 
in~\cite[Section~4.1.2]{Berglund_Gentz_book}, but see 
also~\cite{GHM,HerrmannImkeller05,ChenGemmerSilberVolkening_19} for related 
discussions with a sometimes slightly different terminology.  
\begin{enumerate}
 \item 	In the \emph{super-adiabatic regime} $\eps \ll 
\e^{-2h^{\max}/\sigma^2}$, the forcing period is much longer than the 
instantaneous Eyring--Kramers time for any $y$, so that transitions between the 
potential wells are expected to be instantaneous compared to the system's 
period.   
\item 	In the \emph{intermediate regime} $\e^{-2h^{\max}/\sigma^2} \lesssim 
\eps \lesssim \e^{-2h^{\min}/\sigma^2}$, transitions times between wells can be 
short or long compared to the forcing period, depending on the current value 
of $y$. This is the parameter regime in which stochastic resonance can 
occur~\cite{BSV,Nicolis,WM,GHM,HanggiReview02,HerrmannImkeller05}.
\item 	In the \emph{fast forcing regime} $\e^{-2h^{\min}/\sigma^2} \ll \eps$, 
the forcing period is much shorter than the 
instantaneous Eyring--Kramers time, so that an averaging effect takes place: 
the expected transition time will be almost constant as a function of the 
starting time.
\end{enumerate}
The fast-forcing regime can be further subdivided into a strong-noise regime 
$\e^{-2h^{\min}/\sigma^2} \ll \eps \ll \sigma^2$ and a weak-noise regime $\eps 
\gg \sigma^2$. Our results apply to part of the strong-noise regime. The 
weak-noise regime has been studied 
formally in~\cite{ChenGemmerSilberVolkening_19}, using a path-integral approach 
similar to the large-deviation approach used in~\cite{Freidlin2}. Indeed, if 
$\sigma^2 \ll \eps$, the spreading of sample paths around minimisers of the 
large-deviation principle (also called \emph{instantons}) becomes small 
compared to the forcing period, so that saddle-point techniques can be applied.
Note that in~\cite{ChenGemmerSilberVolkening_19}, the super-adiabatic and 
intermediate regime together are called adiabatic regime, while our fast 
forcing regime rather corresponds to the intermediate regime in that work  
(though no very precise boundaries between regimes are provided there).}

\NB{Our main result, Theorem~\ref{thm:main_ff}, applies to the 
fast-forcing regime $\e^{-2h^{\min}/\sigma^2} \ll \eps$. It states that for any 
starting point on $x^*_-(y)$,
\begin{equation}
\label{eq:main_intro} 
 \bigexpec{\tau_+}
 = \frac{2\pi\bigbrak{1 + R_1(\eps,\sigma)}}
 {\displaystyle
 \int_0^1\sqrt{\abs{\partial_{xx}V_0(x^*_0(y),y)}\partial_{xx}V_0(x^*_-(y),y)}
 \e^{-2h_-(y)/\sigma^2} \6y}\;,
\end{equation} 
where $R_1(\eps,\sigma)$ is some (complicated) error term. This error term is 
small for sufficiently small noise intensity $\sigma$, and for values of 
$\eps$ satisfying 
\begin{equation}
  \e^{-2h^{\min}/\sigma^2} \ll \eps \ll 
\bigpar{\e^{-2h^{\min}/\sigma^2}}^{1/4}\;.
\end{equation} 
The lower bound in this condition is natural, because it expresses the 
requirement that the probability of making a transition during one period is 
small. The upper bound, on the other hand, is not natural at all, and we would 
rather expect that the condition $\eps \ll \sigma^2$ suffices, i.e.\ 
that~\eqref{eq:main_intro} holds in the full strong noise fast-forcing regime. 
However, our technique of proof currently does not allow us to obtain such a 
weaker condition.}

Note that~\eqref{eq:main_intro} is indeed a generalisation of the static 
Eyring--Kramers law~\eqref{eq:EK}. \NB{However, it is important to remark 
that we obtain a different expression for the mean transition time as soon as 
the barrier height $h_-(y)$ is non-constant in $y$. Assume for instance that 
the map $y\mapsto h_-(y)$ reaches its maximum at an isolated point 
$y^*\in[0,1)$, and that $h_-''(y^*)$ is strictly negative. Then Laplace 
asymptotics yield
\begin{equation}
 \bigexpec{\tau_+}
 = \frac{2\pi \e^{2h_-(y^*)/\sigma^2}}
 {\sqrt{\abs{\partial_{xx}V_0(x^*_0(y^*),y^*)}\partial_{xx}V_0(x^*_-(y^*),y^*)}}
 \frac{\sqrt{h_-''(y^*)}}{\sigma\sqrt{\pi}} \bigbrak{1 + R_1(\eps,\sigma)}\;,
\end{equation} 
which differs from the static Eyring--Kramers law~\eqref{eq:EK} by a factor 
$\sqrt{h_-''(y^*)}/(\sigma\sqrt{\pi})$. The extra factor is due to the fact 
that the transition rate is dominated by time windows of width of order 
$\sigma$ around $y^*$ and its translates, as opposed to the constant transition 
rate one has in the static case.}

\NB{In addition to Theorem~\ref{thm:main_ff}, we establish a more general 
result, Theorem~\ref{thm:main_general}, which also covers the superadiabatic 
and intermediate regimes. This result is less precise in those regimes, since 
it does not characterise the mean transition time when starting from a 
particular point at the bottom of one of the wells, but only the mean time 
averaged along the well bottom. Averaging takes place with respect to a 
probability measure, the so-called equilibrium measure, for which we do not have 
a sharp control in this regime. However, as discussed in 
Section~\ref{ssec:discussion}, combining an approximate knowledge of the 
equilibrium measure with Theorem~\ref{thm:main_general} yields results 
compatible with those discussed for instance in~\cite{GHM,HerrmannImkeller05}. 
In particular, in the superadiabatic regime, the predicted mean transition time 
follows the static Eyring--Kramers law with frozen $y$.}

The remainder of this paper is organised as follows. In 
Section~\ref{sec:results}, we define precisely the considered equations, and 
state all main results. These are proved in Sections~\ref{sec:pot} 
to~\ref{sec:proof_main}, see Section~\ref{ssec:proof} for a more precise 
outline of the structure of the proofs. Finally, the appendix contains some of 
the more technical proofs. 

\subsection*{Notations}

The system studied in this work depends on two small parameters $\eps$ and 
$\sigma$. We write $X \lesssim Y$ to indicate that $X \leqs c Y$ for a constant 
$c$ independent of $\eps$ and $\sigma$, as long as $\eps$ and $\sigma$ are small 
enough. The notation $X\asymp Y$ indicates that one has both $X \lesssim Y$ and 
$Y\lesssim X$, while Landau's notation $X = \Order{Y}$ means that $\abs{X} 
\lesssim Y$. The canonical basis of $\R^2$ is denoted $(\ex,\ey)$. If $a, 
b\in\R$, $a\wedge b$ denotes the minimum of $a$ and $b$, and $a\vee b$ denotes 
the maximum of $a$ and $b$. Finally, we write $\indicator{\cD}(x)$ for the 
indicator function of a set or event $\cD$. 

\subsection*{Acknowledgments}

This work is supported by the ANR project PERISTOCH, ANR–19–CE40–0023. 
To a large extent, it was written during the Spring 2020 lockdown due to the 
Covid-19 epidemics, which probably goes a long way toward explaining why it 
contains a lot of detailed computations. The author wishes to thank all members 
of the probability community who helped keeping the community's spirits intact 
by organising and giving many interesting online seminars. \NB{Thanks also to 
the three anonymous reviewers, whose constructive comments allowed to improve 
the first version of this manuscript.}


\section{ Results}
\label{sec:results} 


\subsection{Set-up}
\label{ssec:setup} 

We will consider a version of~\eqref{eq:SDE_fast-slow0} in which time has been 
scaled by a factor $\eps$, given by 
\begin{align}
 \6x_t &= \frac{1}{\eps} b(x_t,y_t)\6t + \frac{\sigma}{\sqrt{\eps}} \6W^x_t\;, 
\\
 \6y_t &= \6t + \sigma\varrho \6W^y_t\;,
\label{eq:SDE_fast-slow} 
\end{align}
where $\set{W^x_t}_{t\geqs0}$ and $\set{W^y_t}_{t\geqs0}$ are independent 
Wiener processes on a filtered probability space 
$(\Omega,\sF,\fP,\set{\sF_t}_{t\geqs0})$, $\eps$, $\sigma$ and $\varrho$  are 
strictly positive parameters, and the drift term $b$ satisfies the following 
assumptions:
\begin{itemize}
\item 	$b:\R^2\to\R$ is of class $\sC^4$ and is periodic, of period $1$, in 
its second argument.
\item 	For any $y\in[0,1]$, the map $x\mapsto b(x,y)$ vanishes at exactly $3$ 
points $x^*_-(y) < x^*_0(y) < x^*_+(y)$, and the derivative $\partial_x b(x,y)$ 
is nonzero for these $3$ values of $x$.
\item	There are constants $M,L>0$ such that $x b(x,y) \leqs -Mx^2$ whenever 
$\abs{x}\geqs L$.  
\end{itemize}
The above conditions guarantee existence of a pathwise unique strong solution 
$(x_t,y_t)_{t\geqs0}$ for any initial condition $(x_0,y_0)$. We denote by 
$\probin{x,y}{\cdot}$ the law of the process starting in $(x,y)$, 
and by $\expecin{x,y}{\cdot}$ expectations with respect to 
$\probin{x,y}{\cdot}$.

\begin{figure}
\begin{center}
\begin{tikzpicture}[>=stealth',point/.style={circle,inner 
sep=0.035cm,fill=white,draw},x=2.5cm,y=5cm, 
declare function={V(\x) = (1/4)*(\x)^4 - (1/12)*(\x)^3 - (1/2)*(\x)^2;}]

\draw[->,semithick] (-2,0) -> (2.5,0);
\draw[->,semithick] (0.5,-0.5) -> (0.5,0.7);

\pgfmathsetmacro{\aa}{0.25}
\pgfmathsetmacro{\xm}{(\aa - sqrt(\aa * \aa + 4))/2}
\pgfmathsetmacro{\xp}{(\aa + sqrt(\aa * \aa + 4))/2}

\draw[purple,<->,semithick,shorten >=1.5pt,shorten <=1.5pt] (\xm,0) -- 
(\xm,{V(\xm)});
\draw[purple,<->,semithick,shorten >=1.5pt,shorten <=1.5pt] (\xp,0) -- 
(\xp,{V(\xp)});

\draw[purple,<->,semithick,shorten >=1.5pt,shorten <=1.5pt] (\xm,{V(\xm)}) -- 
(\xm,{V(\xp)});

\draw[purple,dashed] (\xm,{V(\xp)}) -- (\xp,{V(\xp)});

\draw[blue,thick,-,smooth,domain=-1.55:1.8,samples=75,/pgf/fpu,/pgf/fpu/output
format=fixed] plot (\x, {V(\x)});

\node[point] at (\xm,0) {};
\node[point] at (0,0) {};
\node[point] at (\xp,0) {};

\node[] at (2.35,-0.05) {$x$};
\node[blue] at (0.8,0.58) {$V_0(x,y)$};
\node[] at (\xm,0.05) {$x^\star_-(y)$};
\node[] at (0,0.05) {$x^\star_0(y)$};
\node[] at (\xp,0.05) {$x^\star_+(y)$};

\node[purple] at ({\xm+0.25},{V(\xm)/2}) {$h_-(y)$};
\node[purple] at ({\xp-0.24},{V(\xp)/2}) {$h_+(y)$};
\node[purple] at ({\xm+0.22},{(V(\xm)+V(\xp))/2}) {$\Delta(y)$};

\end{tikzpicture}
\vspace{-5mm}
\end{center}
\caption[]{For each $y$, the map $x\mapsto V_0(x,y)$ is a double-well 
potential. 
}
\label{fig:potential}
\end{figure}
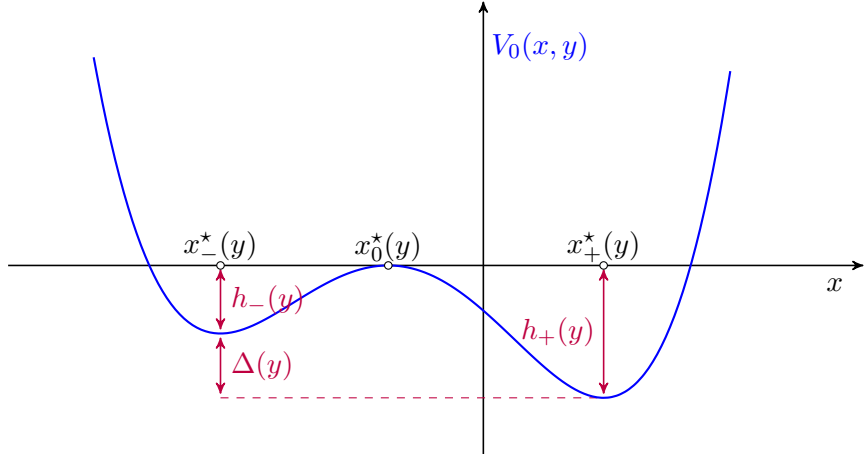

We define the potential 
\begin{equation}
 V_0(x,y) = -\int_{x^*_0(y)}^x b(\bar x,y)\6\bar x\;.
\end{equation}
The assumptions on $b$ imply that for any $y$, $x\mapsto V_0(x,y)$ has local 
minima at $x^*_\pm(y)$, a local maximum at $x^*_0(y)$, and grows at least 
quadratically for large $\abs{x}$. We say that $V_0$ is a double-well 
potential (\figref{fig:potential}). We denote the well depths by 
\begin{equation}
 h_\pm(y) = V_0(x^*_0(y),y)) - V_0(x^*_\pm(y),y)
 = - V_0(x^*_\pm(y),y)\;,  
\end{equation}
and measure the curvatures at stationary points by 
\begin{align}
 \omega_\pm(y) &= \sqrt{\partial_{xx} V_0(x^*_\pm(y),y)}
 = \sqrt{-\partial_x b(x^*_\pm(y),y)}\;, \\
 \omega_0(y) &= \sqrt{\abs{\partial_{xx} V_0(x^*_0(y),y)}}
 = \sqrt{\partial_x b(x^*_0(y),y)}\;.
\end{align} 
The assumptions on $b$ imply that all these quantities are finite and bounded 
away from zero, uniformly in $y$. We further write $\Delta(y) = h_+(y) - 
h_-(y)$ for the difference of the two potential well depths. 


\subsection{Static system}
\label{ssec:static} 

We recall some well-known properties of the static system 
\begin{equation}
\label{eq:static} 
 \6x_t = b(x_t,y)\6t + \sigma \6W^x_t
\end{equation} 
in which $y$ is kept constant. Its direct and adjoint infinitesimal generators 
are the differential operators 
\begin{align}
\label{eq:Lx} 
 \sL_x f &= \frac{\sigma^2}{2} \partial_{xx}f + b\partial_x f\;, 
 &
 \sL^\dagger_x \mu &= \frac{\sigma^2}{2} \partial_{xx}\mu - \partial_x[b\mu]\;, 
 \\
 &= \frac{\sigma^2}{2} \e^{2V_0/\sigma^2} \partial_x 
\bigpar{\e^{-2V_0/\sigma^2}\partial_xf}\;, 
 & 
 &= \frac{\sigma^2}{2} \partial_x 
\bigpar{\e^{-2V_0/\sigma^2}\partial_x(\e^{2V_0/\sigma^2}\mu)}\;.
\end{align} 
In particular, the kernel of $\sL_x$ is spanned by constant functions, while 
the kernel of $\sL^\dagger_x$ is spanned by the density $\pi_0(x \vert y)$ of 
the invariant measure of~\eqref{eq:static}, which is given by
\begin{equation}
 \pi_0(x \vert y) 
 = \frac{1}{Z_0(y)} \e^{-2V_0(x,y)/\sigma^2}\;, 
 \qquad 
 Z_0(y) = \int_{-\infty}^\infty \e^{-2V_0(x,y)/\sigma^2} \6x
\end{equation} 
(\figref{fig:eigenfunctions}). 
\begin{figure}
\begin{center}
\begin{tikzpicture}[>=stealth',point/.style={circle,inner 
sep=0.035cm,fill=white,draw},x=2.5cm,y=7cm, 
declare function={V(\x) = (1/4)*(\x)^4 - (1/12)*(\x)^3 - (1/2)*(\x)^2;}]

\draw[->,semithick] (-2,0) -> (2.5,0);
\draw[->,semithick] (0.5,-0.1) -> (0.5,0.65);

\pgfmathsetmacro{\aa}{0.25}
\pgfmathsetmacro{\sig}{0.25}
\pgfmathsetmacro{\xm}{(\aa - sqrt(\aa * \aa + 4))/2}
\pgfmathsetmacro{\xp}{(\aa + sqrt(\aa * \aa + 4))/2}
\pgfmathsetmacro{\del}{2}
\pgfmathsetmacro{\pp}{0.07*cosh(\del)*(1 + tanh(\del))}

\draw[blue,thick,-,smooth,domain=-2:2.4,samples=75,/pgf/fpu,/pgf/fpu/output
format=fixed] plot (\x, {\sig*exp(-(V(\x)+0.3)/(\sig*\sig))});

\draw[purple,thick,-,smooth,domain=-2:2.4,samples=75,/pgf/fpu,/pgf/fpu/output
format=fixed] plot (\x, {-0.07*cosh(\del)*(tanh(2*\del*\x) - tanh(\del))});

\node[point] at (\xm,0) {};
\node[point] at (0,0) {};
\node[point] at (\xp,0) {};

\node[] at (2.35,-0.05) {$x$};
\node[] at (\xm,-0.05) {$x^\star_-(y)$};
\node[] at (0,-0.05) {$x^\star_0(y)$};
\node[] at (\xp,-0.05) {$x^\star_+(y)$};

\node[blue] at (1.5,0.58) {$\pi_0(x\vert y)$};
\node[purple] at (-1.5,0.55) {$\phi_1(x\vert y)$};

\node[point] at (0.5,{\pp}) {};
\node[] at (0.2,{\pp}) {$\e^{\Deltabar(y)/\sigma^2}$};

\end{tikzpicture}
\vspace{-5mm}
\end{center}
\caption[]{Sketch of the static eigenfunctions $\pi_0(x\vert y)$ and 
$\phi_1(x\vert y)$ for the potential of \figref{fig:potential}.
}
\label{fig:eigenfunctions}
\end{figure}
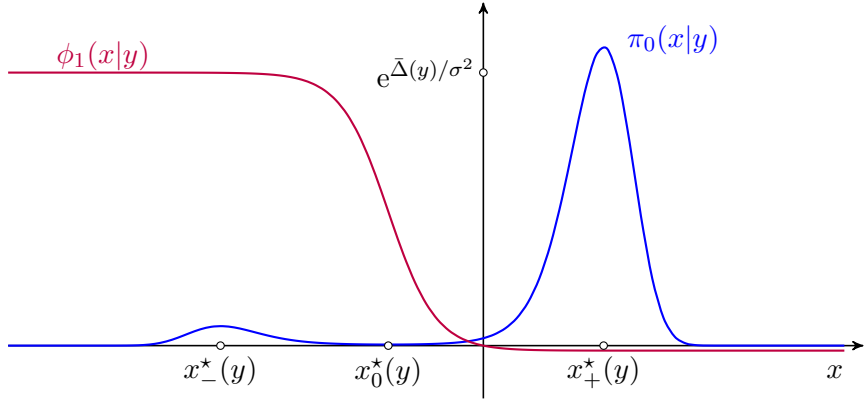
We denote the eigenvalues of $\sL_x$ and $\sL^\dagger_x$ by 
\begin{equation}
 0 = -\lambda_0(y) > -\lambda_1(y) > -\lambda_2(y) \geqs \dots\;,
\end{equation}
and the corresponding $L^2$-normalised eigenfunctions by $\phi_n(\cdot\vert y)$ 
and $\pi_n(\cdot\vert y)$. These are related by 
\begin{equation}
\pi_n(x\vert y) = \pi_0(x\vert y) \phi_n(x\vert y)\;. 
\end{equation} 
There is a spectral gap of order $1$, separating $\lambda_1(y)$ from 
$\lambda_2(y)$ and all subsequent eigenvalues, which is why an important role 
will be played by $\lambda_1(y)$ and the associated eigenfunctions. The 
eigenvalue satisfies 
\begin{equation}
\label{eq:lambda1} 
 \lambda_1(y) = \bigbrak{r_+(y) + r_-(y)} \bigbrak{1+\Order{\sigma^2}}\;,
\end{equation} 
where
\begin{equation}
 r_\pm(y) = \frac{\omega_\pm(y)\omega_0(y)}{2\pi}
 \e^{-2h_{\pm}(y)/\sigma^2}\;.
\end{equation} 
The corresponding eigenfunction can be approximated in terms of the committor 
\begin{equation}
\label{eq:h01} 
 h_0(x\vert y) = 
 \bigprobin{x}{\tau_{x^*_-(y)} < \tau_{x^*_+(y)}}\;,
 \qquad 
 \tau_{\bar x} = \inf\setsuch{t>0}{x_t = \bar x}\;,
\end{equation} 
which satisfies $\sL_x h_0 = 0$ with boundary conditions 
$h_0(x^*_-(y)\vert y) = 1$ and $h_0(x^*_+(y)\vert y) = 0$.
\NB{Indeed, this follows from Dynkin's formuma, see~\eqref{eq:Dynkin} in 
Appendix~\ref{app:pot}.}
Solving this equation, one obtains that for all $x\in[x^*_-(y),x^*_+(y)]$, 
\begin{equation}
\label{eq:h02} 
 h_0(x\vert y) 
 = \frac{1}{N(y)} \int_x^{x^*_+(y)} \e^{2V_0(\bar x,y)/\sigma^2} \6\bar x\;, 
 \qquad 
 N(y) = \int_{x^*_-(y)}^{x^*_+(y)} \e^{2V_0(\bar x,y)/\sigma^2} \6\bar x\;.
\end{equation} 
The first eigenfunction of $\sL_x$ is related to $h_0(x\vert y)$ by 
\begin{equation}
\label{eq:phi1} 
 \phi_1(x\vert y) 
 = \Bigbrak{\e^{\Deltabar(y)/\sigma^2} h_0(x\vert y) 
 - \e^{-\Deltabar(y)/\sigma^2} \bigpar{1-h_0(x\vert y)} }
 \bigbrak{1 + \bigOrder{\lambda_1(y)\lsig}}\;,
\end{equation}
where $\Deltabar(y)$ is defined by 
\begin{equation}
\label{eq:Deltabar} 
 \e^{2\Deltabar(y)/\sigma^2} = \frac{r_-(y)}{r_+(y)}
 \qquad\Rightarrow\qquad 
 \Deltabar(y) = \Delta(y) + 
\frac{\sigma^2}{2}\log\biggpar{\frac{\omega_-(y)}{\omega_+(y)}}\;.
\end{equation} 
The function $x\mapsto\phi_1(x\vert y)$ is almost constant except near 
$x^*_0(y)$, with a value close to $\e^{\Deltabar(y)/\sigma^2}$ for $x < 
x^*_0(y)$ and close to $-\e^{-\Deltabar(y)/\sigma^2}$ for $x > x^*_0(y)$  
(\figref{fig:eigenfunctions}). We give a precise statement of~\eqref{eq:phi1}, 
including bounds on derivatives of $\phi_1$, in 
Section~\ref{ssec:invariant_static}.


\subsection{Two-state jump process}
\label{ssec:jump} 

The spectral-gap property implies that for small $\sigma$, the dynamics of the 
static system~\eqref{eq:static} is well-approximated by a two-state Markovian 
jump process with rates $r_\pm(y)$. 
It is thus natural to expect that the dynamics of the fast-slow  
system~\eqref{eq:SDE_fast-slow} is well-approximated by a time-dependent 
two-state process, in which $y$ plays the role of time (\figref{fig:jump}). 
Its law $(p_-(y),p_+(y))$ satisfies the system  
\begin{align}
 \eps p_-'(y) &= \phantom{-} r_+(y)p_+(y) - r_-(y)p_-(y) \\
 \eps p_+'(y) &= - r_+(y)p_+(y) + r_-(y)p_-(y)\;. 
\label{eq:twostate} 
\end{align}
Let 
\begin{equation}
\label{eq:defA} 
 A(y) = \frac{r_-(y)-r_+(y)}{r_-(y)+r_+(y)}
 = \tanh\biggpar{\frac{\Deltabar(y)}{\sigma^2}}\;, 
\end{equation} 
and let $\delta(y)$ be the $1$-periodic solution of 
\begin{equation}
\label{eq:ODE_delta} 
 \eps \delta'(y) = -\lambda_1(y) \bigbrak{\delta(y) - A(y)}\;.
\end{equation} 
Then it is straightforward to check that the solution of  
System~\eqref{eq:twostate} with initial condition $(p_+(y_0),p_-(y_0))$ 
satisfying $p_+(y_0) + p_-(y_0) = 1$ is given by 
\begin{equation}
\label{eq:ppm} 
 p_\pm(y) = \frac12 \bigbrak{1\pm\delta(y)} 
 \pm \frac12 \Bigbrak{p_+(y_0) - p_-(y_0) - \delta(y_0)} 
 \e^{-\Lambda(y,y_0)/\eps}\;,
\end{equation} 
where 
\begin{equation}
 \Lambda(y,y_0) = \int_{y_0}^y \lambda_1(\bar y)\6\bar y\;.
\end{equation} 
Note that $\delta(y)$ admits the explicit integral representation 
\begin{equation}
\label{eq:deltay} 
 \delta(y) = \frac{1}{\eps(\e^{\Lambda(1,0)/\eps}-1)}
 \int_y^{y+1} \lambda_1(\bar y) A(\bar y) \e^{\Lambda(\bar y,y)/\eps} 
 \6\bar y\;.
\end{equation} 
Two regimes are of particular interest:
\begin{itemize}
\item 	In the \emph{fast forcing regime} $\eps \gg \max_{y\in[0,1]} 
\lambda_1(y)$, the dynamics is averaged, and $\delta(y)$ satisfies 
\begin{equation}
 \delta(y) = \frac{1}{\Lambda(1,0)}
 \int_0^1 \lambda_1(\bar y) A(\bar y)\6\bar y 
 \biggbrak{1+\biggOrder{\frac{\max_{y\in[0,1]} \lambda_1(y)}\eps}}\;.
\end{equation} 
In this case, $\delta(y)$ and $p_\pm(y)$ are asymptotically almost constant. 

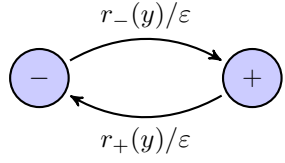
\begin{figure}
\begin{center}
\begin{tikzpicture}[->,>=stealth',shorten >=2pt,shorten 
<=2pt,auto,node
distance=4.0cm, thick,main node/.style={circle,scale=0.7,minimum size=1.1cm,
fill=blue!20,draw,font=\sffamily\Large}]

  \node[main node] (1) {$-$};
  \node[main node] (2) [right of=1] {$+$};

  \path[every node/.style={font=\sffamily\small}]
    (1) edge [bend left,above] node {$r_-(y)/\eps$} (2)
    (2) edge [bend left,below] node {$r_+(y)/\eps$} (1)
    ;
\end{tikzpicture}
\vspace{-5mm}
\end{center}
\caption[]{Time-dependent two-state markovian jump process.
}
\label{fig:jump}
\end{figure}

\item 	In the \emph{super-adiabatic regime} $\eps \ll \min_{y\in[0,1]} 
\lambda_1(y)$, integration by parts shows that 
\begin{equation}
 \delta(y) = A(y) 
 \biggbrak{1+\biggOrder{\frac{\eps}{\min_{y\in[0,1]} \lambda_1(y)}}}\;.
\end{equation} 
Thus $\delta(y)$ tracks $A(y)$, which is close to the sign of 
$\Deltabar(y)$, meaning that with high probability, the jump process is found 
in the currently deepest potential well. 
\end{itemize}

It is also possible to compute explicitly the expectation of the transition 
time $\taujump_+$ from the $-$ state to the $+$ state. We give the simple proof 
of the following result in Appendix~\ref{app:jump}. 

\begin{prop}
\label{prop:twostate} 
For any $y_0\in[0,1]$, one has 
\begin{equation}
 \bigexpecin{-,y_0}{\taujump_+} 
 = \frac{1}{1 - \e^{-R_-(1,0)/\eps}} 
 \int_0^1 \e^{-R_-(y_0+y,y_0)/\eps} \6y\;,
\end{equation} 
where
\begin{equation}
 R_-(y_1,y_0) = \int_{y_0}^{y_1} r_-(\bar y)\6\bar y\;.
\end{equation} 
A similar expression holds for the transition time $\taujump_-$ from the $+$ 
state to the $-$state. 
\end{prop}

The same distinction between regimes as above can be made here:
\begin{itemize}
\item 	If $\eps \gg \max_{y\in[0,1]}r_-(y)$, then the expected jump time does 
not depend on $y_0$ to leading order, and is given by the average 
\begin{equation}
 \bigexpecin{-,y_0}{\taujump_+} = 
 \frac{\eps}{\NB{R_-(1,0)}}
 \biggbrak{1+\biggOrder{\frac{\max_{y\in[0,1]} r_-(y)}\eps}}\;.
\end{equation} 

\item 	If $\eps \ll \min_{y\in[0,1]} r_-(y)$, then the expected jump time is 
much shorter than the oscillation period, and thus given by the instantaneous 
value  
\begin{equation}
 \bigexpecin{-,y_0}{\taujump_+} = 
 \frac{\eps}{r_-(y_0)}
 \biggbrak{1+\biggOrder{\frac{\eps}{\min_{y\in[0,1]} r_-(y)}}}\;.
\end{equation}
\end{itemize}


\subsection{Invariant measure}
\label{ssec:invariant} 

We now return to the fast-slow SDE~\eqref{eq:SDE_fast-slow}. In order to be 
able to apply the potential-theoretic approach of~\cite{Landim_Mariani_Seo19}, 
it is necessary to control the invariant measure of the system. The main result 
of this section is the following theorem, which will be proved in 
Section~\ref{sec:invariant}. 

\begin{theorem}[Invariant measure]
\label{thm:pi} 
For sufficiently small $\sigma$ and $\eps$, 
\NB{the system~\eqref{eq:SDE_fast-slow} admits a unique invariant measure, 
which} has the density 
\begin{equation}
\label{eq:pi_prop} 
 \pi(x,y) = \pi_0(x\vert y) 
 \bigbrak{1 + \alpha_1(y)\phi_1(x\vert y) + \Phi_\perp(x,y)}\;, 
\end{equation} 
where 
\begin{equation}
 \alpha_1(y) = \sinh\biggpar{\frac{\Deltabar(y)}{\sigma^2}}
 - \delta_1(y) \cosh\biggpar{\frac{\Deltabar(y)}{\sigma^2}}\;.
\end{equation} 
Here $\Deltabar(y)$ is given by~\eqref{eq:Deltabar}, and $\delta_1(y)$ is the 
unique periodic solution of the linear second-order equation 
\begin{equation}
\label{eq:ode_delta1} 
  \frac{\varrho^2}{2} \eps\sigma^2\delta_1''
 - \eps q_1(y)\delta_1' 
 - \lambda_1(y) q_2(y)\biggbrak{\delta_1 - 
\tanh\biggpar{\frac{\Deltabar(y)}{\sigma^2}}} + q_3(y) = 0\;,
\end{equation} 
where 
\begin{align}
q_1(y) &= 1 + \bigOrder{\lambda_1(y)\lsig^2}\;, \\
q_2(y) &= 1 + \BigOrder{\frac{\eps}{\sigma^2}\lsig^3}\;,\\
q_3(y) &= \biggOrder{\frac{\eps}{\sigma^2}\lambda_1(y)\lsig^3}
+ \biggOrder{\frac{\eps^3}{\sigma^6} \sqrt{\lambda_1(y)\lsig^3}}\;.
\label{eq:q123} 
\end{align}
Furthermore, the error term $\Phi_\perp(x,y)$ in~\eqref{eq:pi_prop} is 
orthogonal to the span of $\phi_0$ and $\phi_1$, and satisfies 
\begin{equation}
\label{eq:bound_Phiperp_main} 
 \pscal{\pi_0}{\Phi_\perp^2}^{1/2} 
 \lesssim \frac{\eps}{\sigma^2}\cosh\biggpar{\frac{\Deltabar(y)}{\sigma^2}}\;, 
\end{equation} 
where $\pscal{\cdot}{\cdot}$ denotes the standard inner product for 
$L^2(\R,\6x)$.
\end{theorem}

As we will see in Section~\ref{sec:invariant}, the periodic solution 
of~\eqref{eq:ode_delta1} is in fact close to the periodic solution of the 
first-order equation 
\begin{equation}
 \eps\delta_1' 
 = -\lambda_1(y) \frac{q_2(y)}{q_1(y)} \Bigbrak{\delta_1 - A(y)} + 
\frac{q_3(y)}{q_1(y)}\;,
\end{equation} 
which is similar to~\eqref{eq:ODE_delta}. The function $\delta_1(y)$ also has a 
similar interpretation as $\delta(y)$ in~\eqref{eq:ppm}. Indeed, one has (by a 
similar argument as in the proof of Corollary~\ref{cor:hstarAB_pi}) 
\begin{equation}
 p_-(y) := \int_{-\infty}^{x^*_0(y)} \pi(x,y)\6x 
 = \frac12 \bigbrak{1 - \delta_1(y) + \Order{\sigma^2}}\;.
\end{equation} 
By analogy with~\eqref{eq:ppm}, $p_-(y)$ can be interpreted as the \lq\lq 
instantaneous\rq\rq\ probability to be in the left-hand potential well at 
equilibrium.

Given a function $f:[0,1]\to\R$, we introduce the notation 
\begin{equation}
\label{eq:def_avrg} 
 \avrg{f} = \int_0^1 f(y)\6y\;. 
\end{equation}
We will mainly be concerned with the fast-forcing regime 
$\eps\gg\avrg{\lambda_1}$. Then $\delta_1(y)$ is actually nearly constant, in 
the sense that 
\begin{equation}
\label{eq:delta1ff} 
 \delta_1(y) = \bar\delta_1 
 \biggbrak{1 + \biggOrder{\frac{\avrg{\lambda_1}}{\eps}}}\;,
\end{equation} 
where 
\begin{equation}
\label{eq:delta1bar} 
 \bar\delta_1 
 = \frac{1}{\avrg{\lambda_1}} \biggbrak{\avrg{\lambda_1 A} + 
 \biggOrder{\frac{\eps}{\sigma^2}\lsig^3\avrg{\lambda_1}}
+ \biggOrder{\frac{\eps^3}{\sigma^6} \lsig^{3/2}\avrg{\sqrt{\lambda_1}}}}\;.
\end{equation} 
One should note that the main limitation of Theorem~\ref{thm:pi} lies in the 
error term proportional to $\sqrt{\lambda_1(y)}$ in~\eqref{eq:q123}, which 
causes the error term in $\avrg{\sqrt{\lambda_1}}$ in~\eqref{eq:delta1bar}. 
This is due to technical difficulties in controlling $\Phi_\perp$, and will 
limit the applicability of our results to the regime 
\begin{equation}
 \eps \ll \avrg{\lambda_1}^{1/4}\;.
\end{equation} 
In fact, there is already a substantial amount of work involved in getting an 
error term proportional to $(\eps/\sigma^2)^3\avrg{\sqrt{\lambda_1}}$, rather 
than $(\eps/\sigma^2)^2\avrg{\sqrt{\lambda_1}}$. This improvement is due to the 
fact that we are able to prove that 
\begin{equation}
 \Phi_\perp(x,y) = \Phi_\perp^*(x,y) + \Phi_\perp^1(x,y)\;,
\end{equation} 
where $\Phi_\perp^*$ is explicit, and has a contribution of order 
$\lambda_1(y)$ to $q_3(y)$, while $\Phi_\perp^1$ satisfies a bound of the 
form~\eqref{eq:bound_Phiperp_main}, but with a larger power of $\eps$. See 
Corollary~\ref{cor:bound_Phiperp} for details.


\subsection{Main results: expected transition time}
\label{ssec:transition} 

In order to formulate our main result, we introduce two functions 
\begin{align}
 a(y) &= x^*_-(y) + \hat\rho\;, \\
 b(y) &= x^*_+(y) - \hat\rho\;,
 \label{eq:def_ab} 
\end{align}
where $\hat\rho>0$ is a parameter of order $1$ that will be taken sufficiently 
small. We then define two set 
\begin{align}
 \A &= \bigsetsuch{(x,y)\in\R\times[0,1]}{x \leqs a(y)}\;, \\
 \B &= \bigsetsuch{(x,y)\in\R\times[0,1]}{x \geqs b(y)}\; 
\end{align}
see \figref{fig:AB}.

\begin{figure}
\begin{center}
\begin{tikzpicture}[>=stealth',point/.style={circle,inner 
sep=0.035cm,fill=white,draw},x=7cm,y=1.5cm, 
declare function={ya(\x) = -1 + 0.2*sin(\x*360);
yb(\x) = -0.25*cos(\x*360 - 30);
yc(\x) = 1 + 0.2*sin(\x*360 + 45);}]

\path[fill=blue!30,smooth,domain=-0.1:1.1,samples=60,/pgf/fpu,
/pgf/fpu/output format=fixed] plot (\x, {ya(\x) + 0.1}) -- (1.1,-1.4) 
-- (-0.1,-1.4);

\path[fill=blue!30,smooth,domain=-0.1:1.1,samples=60,/pgf/fpu,
/pgf/fpu/output format=fixed] plot (\x, {yc(\x) - 0.1}) -- (1.1,1.4) 
-- (-0.1,1.4);

\draw[->,semithick] (0,-1.4) -> (0,1.8);
\draw[->,semithick] (-0.2,0) -> (1.2,0);

\draw[green!50!black,thick,-,smooth,domain=-0.1:1.1,samples=75,/pgf/fpu,
/pgf/fpu/output format=fixed] plot (\x, {ya(\x)});

\draw[dashed,purple,thick,-,smooth,domain=-0.1:1.1,samples=75,/pgf/fpu,
/pgf/fpu/output format=fixed] plot (\x, {yb(\x)});

\draw[green!50!black,thick,-,smooth,domain=-0.1:1.1,samples=75,/pgf/fpu,
/pgf/fpu/output format=fixed] plot (\x, {yc(\x)});

\node[] at (1.15,0.12) {$y$};
\node[] at (0.035,1.55) {$x$};
\node[blue] at (0.2,-1.1) {$\A$};
\node[blue] at (0.6,1.1) {$\B$};
\node[green!50!black] at (1.18,{ya(1.1)}) {$x^*_-(y)$};
\node[purple] at (0.4,{yb(0.4) + 0.2}) {$x^*_0(y)$};
\node[green!50!black] at (1.18,{yc(1.1)}) {$x^*_+(y)$};

\end{tikzpicture}
\vspace{-5mm}
\end{center}
\caption[]{Definition of the sets $\A$ and $\B$, in relation with the extrema 
$x^*_\pm(y)$ and $x^*_0(y)$ of the map $x \mapsto V_0(x,y)$. 
}
\label{fig:AB}
\end{figure}
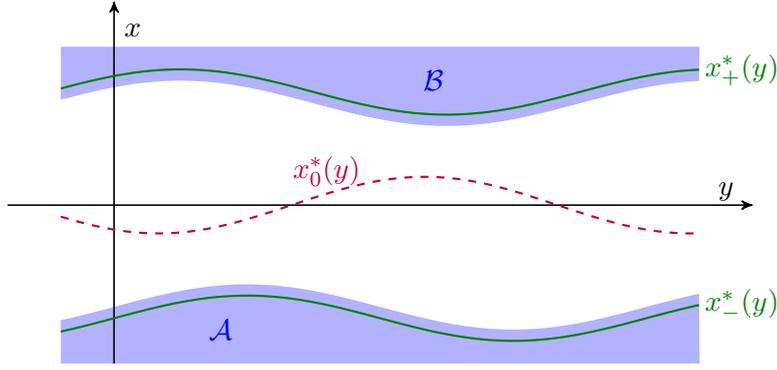

Our first main result gives a general expression for the expected first-hitting 
time of $\B$, when starting with a specific distribution on $\partial\A$, the 
so-called equilibrium measure. 

\begin{theorem}[Main result, general case]
\label{thm:main_general} 
There exists a probability measure $\nu_{\AB}$, supported on $\partial\A$, such 
that 
\begin{equation}
\label{eq:main_general} 
 \int_{\partial\A} \bigexpecin{(x,y)}{\tau_\B} \6\nu_{\AB} 
 = \frac{2\eps [1 - \avrg{\delta_1}]}{\avrg{\lambda_1[1-A\delta_1]}} 
 \bigbrak{1 + R_0(\eps,\sigma)}\;,
\end{equation} 
where $R_0(\eps,\sigma)$ is an error term satisfying
\begin{equation}
\label{eq:def_R0} 
 \bigabs{R_0(\eps,\sigma)} 
 \lesssim \sigma^2 
+ \frac{\eps\lsig\avrg{\sqrt{\lambda_1}}}{\sigma^2[1-\avrg{\delta_1}]}  
+ \frac{\eps\lsig^2\avrg{\lambda_1}}{\sigma^2\avrg{\lambda_1[1-A\delta_1]}} 
+ \frac{\eps^2\lsig\avrg{\sqrt{\lambda_1}}} 
{\sigma^{7/2}\avrg{\lambda_1[1-A\delta_1]}}\;.
\end{equation} 
\end{theorem}

As such, this result has two main limitations. First, it is not immediately 
apparent for which values of $\eps$ and $\sigma$ the remainder 
$R_0(\eps,\sigma)$ is actually small. And second, we do not know the 
equilibrium measure $\nu_{\AB}$. 

We will address both issues in the fast-forcing regime $\eps \gg 
\avrg{\lambda_1}$. In fact, the discussion of the two-state jump process in 
Section~\ref{ssec:jump} suggests that if $\eps < \avrg{\lambda_1}$, the 
expected first-hitting time depends strongly on the starting point, whereas 
it is almost constant if $\eps \gg \avrg{\lambda_1}$. 

Recalling the expressions~\eqref{eq:lambda1} for $\lambda_1(y)$ 
and~\eqref{eq:defA} for $A(y)$, we obtain 
\begin{align}
 \avrg{\lambda_1} &= \bigbrak{\avrg{r_-} + \avrg{r_+}} 
\bigbrak{1+\Order{\sigma^2}}\;, \\
 \avrg{\lambda_1 A} &= \bigbrak{\avrg{r_-} - \avrg{r_+}} 
\bigbrak{1+\Order{\sigma^2}}\;.
\label{eq:average_lambda1A} 
\end{align}
Furthermore, if $\eps \gg \avrg{\lambda_1}$, \eqref{eq:delta1ff} 
and~\eqref{eq:delta1bar} imply that $\delta_1(y)$ is close to 
\begin{equation}
 \frac{\avrg{\lambda_1 A}}{\avrg{\lambda_1}} 
 = \frac{\avrg{r_-} - \avrg{r_+}}{\avrg{r_-} + 
\avrg{r_+}}\bigbrak{1+\Order{\sigma^2}}\;.
\end{equation} 
It follows that the leading term in~\eqref{eq:main_general} is given by 
\begin{align}
 2\eps \frac{\avrg{\lambda_1} - \avrg{\lambda_1A}} 
 {\avrg{\lambda_1}^2 - \avrg{\lambda_1A}^2}
 &= \frac{\eps}{\avrg{r_-}} \bigbrak{1+\Order{\sigma^2}} \\
 &= \frac{2\pi\eps}{\displaystyle
 \int_0^1 \omega_-(y)\omega_0(y) \e^{-2h_-(y)/\sigma^2}\6y}
\bigbrak{1+\Order{\sigma^2}}\;,
\end{align} 
which agrees with~\eqref{eq:main_intro} (recall that we have scaled time by a 
factor $\eps$). 
In order to quantify error terms, we introduce minimal barrier heights
\begin{equation}
\label{eq:def_hmin} 
 h_\pm^{\min} = \min_{0\leqs y\leqs 1} h_\pm(y)\;,
\end{equation} 
and asymmetry factors 
\begin{equation}
\label{eq:def_HH} 
 H = \bigabs{h_-^{\min} - h_+^{\min}}\;, 
 \qquad 
 H_- = \bigbrak{h_-^{\min} - h_+^{\min}}_+\;,
\end{equation} 
where $\brak{\cdot}_+$ denotes the positive part. 

\begin{theorem}[Main result, fast-forcing regime]
\label{thm:main_ff} 
Assume $\eps \gg \avrg{\lambda_1}$. Then for any initial condition 
$(x,y)\in\partial\A$, we have 
\begin{equation}
\label{eq:main_ff} 
 \bigexpecin{(x,y)}{\tau_\B} 
 = \frac{\eps}{\avrg{r_-}} \bigbrak{1 + R_1(\eps,\sigma)}\;,
\end{equation} 
where 
\begin{equation}
\label{eq:R1_thm} 
 \bigabs{R_1(\eps,\sigma)}
 \lesssim \sigma^2 + \biggpar{\frac{\eps\lsig^3}{\sigma^2} + 
\frac{\eps^2\lsig}{\sigma^{7/2}\avrg{\lambda_1}^{1/2}} + 
\frac{\avrg{\lambda_1}^2}{\eps}} \e^{2H/\sigma^2} 
+ \frac{\avrg{\lambda_1}}{\eps} (1 + \e^{2H_-/\sigma^2})\;.
\end{equation} 
\end{theorem}

In the symmetric case $h_-^{\min} = h_+^{\min}$, we have $H = H_- = 0$, and the 
error term takes the simpler form 
\begin{equation}
 \bigabs{R_1(\eps,\sigma)}
 \lesssim \sigma^2 + \frac{\eps\lsig^3}{\sigma^2} + 
\frac{\eps^2\lsig}{\sigma^{7/2}\avrg{\lambda_1}^{1/2}} +
\frac{\avrg{\lambda_1}}{\eps}\;.
\end{equation} 
Disregarding powers of $\sigma$ with respect to exponential terms, we see that 
Theorem~\ref{thm:main_ff} is applicable when  
\begin{equation}
\label{eq:eps_interval} 
 \avrg{\lambda_1} \ll \eps \ll \avrg{\lambda_1}^{1/4}\;.
\end{equation} 
In the asymmetric case $h_-^{\min} \neq h_+^{\min}$, the error term is larger, 
and results in stronger conditions on $\eps$. One can however check (see 
Section~\ref{sec:proof_main}) that there exists a non-empty interval of values 
of $\eps$ for which Theorem~\ref{thm:main_ff} is still meaningful as long as 
\begin{equation}
\label{eq:condition_hmin} 
 \frac12 h_-^{\min} < h_+^{\min} < 2h_-^{\min}\;,
\end{equation} 
that is, as long as the asymmetry between the potential wells is not too large.


\subsection{Discussion}
\label{ssec:discussion} 

Theorem~\ref{thm:main_ff} provides a generalisation of the static 
Eyring--Kramers law~\eqref{eq:EK} to slowly oscillating double-well potentials, 
when the forcing frequency $\eps$ lies in an interval given 
by~\eqref{eq:eps_interval} if the oscillation is symmetric, in the sense that 
$h_-^{\min} = h_+^{\min}$. We now provide some comments on what we expect to 
happen outside this domain of validity. This will also serve as a \lq\lq 
reality check\rq\rq\ of our main results. 

If $\eps \leqs \avrg{\lambda_1}$, Theorem~\ref{thm:main_general} is still 
valid, but perhaps not as useful. The main limitation of the result in that 
case is that $\expecin{(x,y)}{\tau_\B}$ is no longer expected to be 
almost constant, so that the equilibrium measure $\nu_{\AB}$ matters. In fact, 
we do have an explicit expression for $\nu_{\AB}$, which is given 
(cf.~\eqref{eq:nuAB}) by 
\begin{equation}
 \6\nu_{\AB} = \frac{\sigma^2}{2\eps\capacity(\A,\B)} (D\nabla h^*_{\AB} \cdot 
\nn)\pi\6\lambda\;,
\end{equation}
where all notations are defined in Section~\ref{sec:pot}. 
In particular, $\nn$ is the unit normal vector to $\partial\A$, and $\6\lambda$ 
is the Lebesgue measure on $\partial\A$, so that 
\begin{equation}
 \nn \6\lambda = \bigpar{\ex - a'(y)\ey} \6y\;.
\end{equation} 
Using the estimates on the capacity $\capacity(\A,\B)$ given in 
Theorem~\ref{thm:capacity}, the expression~\eqref{eq:def_D} of the diffusion 
matrix $D$, the estimate on the adjoint committor $h^*_{\AB}$ obtained in 
Proposition~\ref{prop:bounds_hAB}, and Theorem~\ref{thm:pi} on the invariant 
measure $\pi$, we obtain that to leading order, 
\begin{equation}
\label{eq:nAB_approx} 
 \6\nu_{\AB} \simeq \frac{\sigma^2}{2\eps} 
 \frac{1 + \alpha_1(y) \phi_1(a(y) \vert y)}
 {\tilde N(y)Z_0(y)\capacity(\cA,\cB)}\6y 
\simeq\frac{\lambda_1(y)[1+A(y)][1-\delta_1(y)]}{\avrg{\lambda_1[1-A\delta_1]}} 
\6y\;.
\end{equation} 
Substituting this in the result~\eqref{eq:main_general} of 
Theorem~\ref{thm:main_general}, and using the fact that $\lambda_1(y)[1+A(y)] = 
2r_-(y)$, we obtain that to leading order,
\begin{equation}
\label{eq:average_Er} 
 \bigavrg{\expec{\tau_\B}r_-(1-\delta_1)}
 \simeq \eps\avrg{1-\delta_1}\;.
\end{equation} 

\begin{itemize}
\item 	In the fast-forcing regime $\eps\gg\avrg{\lambda_1}$, the expectation 
$\expecin{(a(y_0),y_0)}{\tau_\B}$ being nearly constant, we recover 
indeed~\eqref{eq:main_ff}.

\item 	In the superadiabatic regime $\eps\ll\min_y\lambda_1(y)$, the 
discussion 
in Section~\ref{ssec:jump} on the two-state jump process suggests that we have 
\begin{equation}
\label{eq:expec_superadiabatic} 
 \expecin{(a(y_0),y_0)}{\tau_\B}
 \simeq \frac{\eps}{r_-(y_0)}\;,
\end{equation} 
which is indeed consistent with~\eqref{eq:average_Er}. 

\item 	In the intermediate regime $\min_y\lambda_1(y) \leqs \eps \leqs 
\max_y\lambda_1(y)$, the situation is more complicated owing to the phenomenon 
of stochastic resonance (see for instance the discussion 
in~\cite[Section~4.1.2]{Berglund_Gentz_book}). What we expect then is the 
following. If the process starts at a point $(a(y_0),y_0)\in\partial\A$ such 
that $\eps < r_-(y_0)$, the mean hitting time of $\B$ will still 
satisfy~\eqref{eq:expec_superadiabatic}. Otherwise, the transition to $\B$ will 
occur near the smallest $y > y_0$ such that $r_-(y) = \eps$, and thus the 
expectation of $\tau_\B$ is dominated by $y-y_0$. This picture is 
also consistent with large-deviation results obtained in~\cite{Freidlin2}. 
\end{itemize}

The other regime not covered by our results is when 
$\eps\geqs\avrg{\lambda_1}^{1/4}$. As noted above, this is mainly due to 
technical difficulties in controlling the part of the invariant measure $\pi$ 
which is orthogonal to the span of the first two eigenfunctions $\pi_0$ and 
$\pi_1$ of $\sL_x^\dagger$. In fact, it seems quite plausible that the 
expression~\eqref{eq:main_ff} for the mean transition time still holds as long 
as $\eps\ll\sigma^2$ (for larger $\eps$, the slow--fast structure of the 
equation for $\pi$ changes). To establish such a result, however, new ideas are 
needed to achieve a better control of the invariant measure. 


\subsection{Outline of the proof}
\label{ssec:proof} 

As already mentioned, the main ingredient of our proof is the 
potential-theoretic approach to metastability, which was developed 
in~\cite{BEGK,BGK} for reversible diffusions, and extended 
in~\cite{Landim_Mariani_Seo19} to general diffusions. We give a quick overview 
of this approach in Section~\ref{sec:pot}. Its key result relates the expected 
first-hitting time of a set $\B$, when starting in the equilibrium measure 
$\nu_{\AB}$ on the boundary of another set $\A$, with the invariant measure of 
the diffusion and the so-called~\emph{capacity} $\capacity(\A,\B)$. See 
Proposition~\ref{prop:magic} below. 

The main difficulty in our case is to determine the invariant measure $\pi$ of 
the system. While this measure is explicitly known for reversible systems, this 
is no longer the case here. As shown in~\cite{Landim_Mariani_Seo19}, $\pi$ is 
related to the solution of a \NB{Hamilton--Jacobi-type} equation, see 
Lemma~\ref{lem:pot_invariant}. However, obtaining an approximate solution of 
this equation with good enough control of error terms turns out to be 
difficult. Therefore we adopt another approach, which consists in expanding 
$\pi$ on a basis of eigenfunctions of the generator of the system with frozen 
$y$, and analysing the resulting system of ODEs. This is done in 
Section~\ref{sec:invariant}, which contains in particular the proof of 
Theorem~\ref{thm:pi}. 

In Section~\ref{sec:adjoint}, we investigate the adjoint system that enters 
the expression for the mean first-hitting time in Proposition~\ref{prop:magic}. 
In particular, we obtain approximate expressions for the committors 
$\prob{\tau_\A < \tau_\B}$ of the original and adjoint system in 
Proposition~\ref{prop:bounds_hAB}, using a perturbation theory argument around 
the committors of the frozen systems. The necessary estimate for 
Proposition~\ref{prop:magic} is then obtained in 
Corollary~\ref{cor:hstarAB_pi}. 

The other quantity that needs to be determined for the potential-theoretic 
approach to work is the capacity $\capacity(A,B)$. This is comparatively easy 
once the invariant measure is known, since the capacity obeys variational 
principles (the Dirichlet and Thomson principle) that give upper and lower 
bounds once one makes a sufficiently good guess of test functions to feed into 
them. It turns out that the system with frozen $y$ provides such sufficiently 
good guesses, the only difficulty being to account for the fact that these 
guesses are not strictly divergence-free. The main result is 
Theorem~\ref{thm:capacity}, which provides upper and lower bounds on the 
capacity. 

Section~\ref{sec:proof_main} contains the last steps of the proof of 
Theorems~\ref{thm:main_general} and~\ref{thm:main_ff}. While 
Theorem~\ref{thm:main_general} follows directly from the obtained bounds on the 
invariant measure, committors and capacity, Theorem~\ref{thm:main_ff} requires 
a little more work, which consists in simplifying the expressions for the 
dominant term and error terms, and getting rid of the equilibrium measure 
$\nu_{\AB}$. 

In order to increase readability, we have relegated some of the more technical 
proofs to the appendix. Appendix~\ref{app:jump} contains the proof of 
Proposition~\ref{prop:twostate} on the two-state jump process, 
Appendix~\ref{app:pot} contains the proofs of the potential-theoretic results 
in 
Section~\ref{sec:pot}, Appendix~\ref{app:proofs_static} contains the estimates 
on static eigenfunctions required for determining the invariant measure, and 
Appendix~\ref{app:Laplace} gathers a few auxiliary results involving Laplace 
asymptotics. 


\section{Non-reversible potential theory}
\label{sec:pot} 

In this section, we give a short overview of the potential-theoretic results 
contained in~\cite[Section~4]{Landim_Mariani_Seo19}, slightly adapted to our 
situation. All proofs are given in Appendix~\ref{app:pot}. 

The infinitesimal generator of the system~\eqref{eq:SDE_fast-slow} is given by 
\begin{equation}
\label{eq:defL} 
 \sL = \frac{\sigma^2}{2\eps} \bigpar{\partial_{xx} + \varrho^2\eps 
\partial_{yy}} + 
\frac{1}{\eps} b\, \partial_x + \partial_y\;.
\end{equation}
A key idea in~\cite{Landim_Mariani_Seo19} is to decompose $\sL$ into 
a symmetric and an antisymmetric part. This allows to define an adjoint 
stochastic process, and both the direct and adjoint process play a role in the 
expressions for mean first-passage times. 


\subsection{Invariant density}
\label{ssec:pot_invariant} 

\begin{lemma}
\label{lem:pot_invariant} 
The system~\eqref{eq:SDE_fast-slow} has an invariant measure with density 
$\pi(x,y) = Z^{-1}\e^{-2V(x,y)/\sigma^2}$, where $V$ solves the 
\NB{Hamilton--Jacobi-type} equation 
\begin{equation}
\label{eq:HJ} 
 (\partial_x V)^2 + b\,\partial_x V + \eps\varrho^2(\partial_y V)^2 + 
\eps \partial_y V 
 = \frac{\sigma^2}{2} \bigbrak{\partial_{xx}V + \partial_x b + \eps\varrho^2 
\partial_{yy}V}\;.
\end{equation} 
\end{lemma}

\begin{lemma}
\label{lem:Lf} 
The infinitesimal generator~\eqref{eq:defL} can be written as 
\begin{align}
 \sL f &= \frac{\sigma^2}{2\eps} \e^{2V/\sigma^2} \Bigset{\partial_x 
\bigbrak{\e^{-2V/\sigma^2}\partial_x f} + \eps\varrho^2 \partial_y 
\bigbrak{\e^{-2V/\sigma^2}\partial_y f}}
 + c \cdot \nabla f \\
 &=: \frac{\sigma^2}{2\eps} \e^{2V/\sigma^2} \nabla\cdot\bigbrak{D 
\e^{-2V/\sigma^2} \nabla f} + c \cdot \nabla f\; 
\label{eq:L_sym_asym} 
\end{align} 
where 
\begin{equation}
 D = 
 \begin{pmatrix}
  1 & 0 \\ 0 & \eps\varrho^2
 \end{pmatrix}
\end{equation}
is a diffusion matrix, and 
\begin{equation}
\label{eq:defc} 
 c = \frac{1}{\eps} (b + \partial_x V) \,\ex 
 + (1 + \varrho^2\partial_y V) \,\ey  
\end{equation} 
satisfies the vanishing divergence condition 
\begin{equation}
\label{eq:divergence} 
 \nabla \cdot (\e^{-2V/\sigma^2}c) = 0\;.
\end{equation} 
\end{lemma}
%


\subsection{Adjoint process}
\label{ssec:adjoint} 

We decompose $\sL$ into a symmetric and an antisymmetric part by writing $\sL = 
\sym{\sL} + \asym{\sL}$, where 
\begin{equation}
\sym{\sL}f = \frac{\sigma^2}{2\eps} \e^{2V/\sigma^2} \nabla\cdot\bigbrak{D 
\e^{-2V/\sigma^2} \nabla f}\;, 
\qquad 
\asym{\sL}f = c \cdot \nabla f\;. 
\end{equation} 
We write $\T = \R/\Z$ for the circle, and endow $L^2(\R\times\T)$ with the 
inner product 
\begin{equation}
 \pscal{f}{g}_\pi = \int_{\R\times\T} f(x,y) g(x,y) \6\pi\;,
\end{equation} 
where $\6\pi = \pi(x,y)\6x\6y$. 
Then one checks that $\sym{\sL}$ is self-adjoint with respect to this 
inner product, while~\eqref{eq:divergence} and the divergence theorem imply  
\begin{equation}
\label{eq:skew} 
 \int_{\R\times\T} f c\cdot \nabla g \, \6\pi 
 = - \int_{\R\times\T} g c\cdot \nabla f \, \6\pi\;,
\end{equation} 
showing that $\asym{\sL}$ is anti-self-adjoint (skew-symmetric), 
that is $\asym{\sL}^\dagger = - \asym{\sL}$. 
By definition, the adjoint process has the generator 
\begin{equation}
 \sL^* f 
 = \sym{\sL}f - \asym{\sL}f
 = \frac{\sigma^2}{2\eps} \e^{2V/\sigma^2} \nabla\cdot\bigbrak{D 
\e^{-2V/\sigma^2} \nabla f}
- c \cdot \nabla f\;. 
\end{equation} 
The corresponding SDE is given by 
\begin{align}
 \6x_t &= \frac{1}{\eps} b^*(x_t,y_t)\6t + \frac{\sigma}{\sqrt{\eps}} 
\6W^x_t\;, \\
 \6y_t &= -\bigbrak{1 + 2\varrho^2\partial_y V(x,y)}\6t + \sigma \varrho 
\6W^y_t\;,
\label{eq:SDE2_adjoint} 
\end{align}
where 
\begin{equation}
 b^* = -\partial_x V - \eps c_x 
 = -2\partial_x V - b\;.
\end{equation} 
We denote by $\sprobin{x,y}{\cdot}$ the law of the adjoint process starting in 
$(x,y)$, and by $\sexpecin{x,y}{\cdot}$ the corresponding expectations.


\subsection{Committor and capacity}
\label{ssec:capacity} 

Consider two sets $\A = \setsuch{(x,y)}{x\leqs a(y)}$ and $\B = 
\setsuch{(x,y)}{x\geqs b(y)}$, where $a(y) < b(y)$ are smooth periodic 
functions. The committors $h_{\AB}(x,y) = \probin{x,y}{\tau_\A < \tau_\B}$ 
and $h^*_{\AB}(x,y) = \sprobin{x,y}{\tau_\A < \tau_\B}$
satisfy the Dirichlet problems  
\begin{equation}
\begin{cases}
 (\sL h)(x,y) = 0 &\quad (x,y)\in \ABc\;, \\
 h(x,y) = 1 &\quad (x,y)\in \A\;, \\
 h(x,y) = 0 &\quad (x,y)\in \B\;, 
\end{cases}
\qquad 
\begin{cases}
 (\sL^* h^*)(x,y) = 0 &\quad (x,y)\in \ABc\;, \\
 h^*(x,y) = 1 &\quad (x,y)\in \A\;, \\
 h^*(x,y) = 0 &\quad (x,y)\in \B\;. 
\end{cases}
\end{equation} 
\NB{Indeed, this follows from Dynkin's formula, see~\eqref{eq:Dynkin} in 
Appendix~\ref{app:pot}.}
The capacities of the direct and adjoint process are defined via the Dirichlet 
form associated with $\sym{\sL}$, that is 
\begin{align}
 \capacity(\A,\B) &= \frac{\sigma^2}{2\eps} \int_{\ABc} \nabla 
h_{\AB}\cdot(D\nabla h_{\AB}) \6\pi\;, \\
\capacity^*(\A,\B) &= \frac{\sigma^2}{2\eps} \int_{\ABc} \nabla 
h^*_{\AB}\cdot(D\nabla h^*_{\AB}) \6\pi\;.
\end{align} 

\begin{lemma}
\label{lem:capacity} 
We have $\capacity(\A,\B) = \capacity(\B,\A)$, and 
\begin{equation}
\label{eq:cap01} 
 \capacity(\A,\B)
 = \frac{\sigma^2}{2\eps} \int_{\partial \A} (D\nabla h_{\AB} \cdot \nn) 
\pi\6\lambda 
 = \int_{\partial \A} \Bigpar{\frac{\sigma^2}{2\eps} D\nabla h_{\AB} + 
h_{\AB}\,c} 
\cdot \nn 
\, \pi\6\lambda \;,
\end{equation} 
where $\nn$ is the inward-pointing unit normal vector to $\partial \A$, and 
$\6\lambda$ is the arclength on $\partial \A$. An analogous relation, with 
$h_{\AB}$ replaced by $h^*_{\AB}$, holds for $\capacity^*(\A,\B)$. Furthermore, 
\begin{align}
\label{eq:cap02} 
 \capacity(\A,\B) 
 &= \frac{\sigma^2}{2\eps} \int_{\ABc} \bigbrak{\nabla h_{\AB}^* \cdot 
(D\nabla h_{\AB}) - \eps h^*_{\AB}(c\cdot \nabla h_{\AB})} \6\pi \\
 &= \frac{\sigma^2}{2\eps} \int_{\ABc} \bigbrak{\nabla h_{\AB} \cdot 
(D\nabla h^*_{\AB}) + \eps h_{\AB}(c\cdot \nabla h^*_{\AB})} \6\pi
= \capacity^*(\A,\B)\;.
\end{align}
\end{lemma}
%


\subsection{Equilibrium measure and mean hitting time}
\label{ssec:hitting} 

The $\AB$-equilibrium measure $\nu_{\AB}$ is the probability measure supported 
on 
$\partial \A$ defined by 
\begin{equation}
\label{eq:nuAB} 
 \6\nu_{\AB} = \frac{\sigma^2}{2\eps\capacity(\A,\B)} (D\nabla h^*_{\AB} \cdot 
\nn)\pi\6\lambda\;.
\end{equation}
Then we have the following fundamental relation. 

\begin{prop}
\label{prop:magic} 
Let $\tau_\B = \inf\setsuch{t>0}{(x_t,y_t)\in \B}$ denote the first-hitting 
time 
of $\B$. Then  
\begin{equation}
\label{eq:magic} 
 \bigexpecin{\nu_{\AB}}{\tau_\B} 
 := \int_{\partial\A}  \bigexpecin{x}{\tau_\B} \6\nu_{\AB} 
 = \frac{1}{\capacity(\A,\B)} \int_{\B^c} h^*_{\AB} \6\pi\;.
\end{equation} 
\end{prop}
%


\subsection{Variational principles}
\label{ssec:variational} 

For $\ph, \psi$ two vector fields on $\ABc$, we define the bilinear form 
\begin{equation}
\label{eq:def_D} 
 \sD(\ph,\psi) 
 := \frac{2\eps}{\sigma^2} \int_{\ABc} \ph(x,y) \cdot 
\bigpar{D^{-1}\psi(x,y)} \frac{\6x\6y}{\pi(x,y)}\;,
\end{equation} 
and we denote $\sD(\ph,\ph)$ by $\sD(\ph)$. 
For $\gamma\in\R$, we write $\sF^\gamma_{\AB}$ for the closure with respect 
to the norm $\sD(\cdot)$ of the set of flows $\ph$ which are divergence-free, 
i.e. 
\begin{equation}
\label{eq:div_free} 
 \nabla \cdot \ph = 0 
 \qquad \text{in $\ABc$\;,}
\end{equation} 
and such that 
\begin{equation}
\label{eq:flow_gamma} 
 \int_{\partial \A} (\ph\cdot\nn) \6\lambda = -\gamma\;.
\end{equation}
We further denote by $\sH^{\alpha,\beta}_{\AB}$ the set of functions 
$f\in L^2(\6\pi)$ which have constant values $\alpha$ in $\A$, and $\beta$ in 
$\B$. 
For such an $f$, we use the notations 
\begin{equation}
 \Phi_f = \frac{\sigma^2}{2\eps} \pi D\nabla f - \pi f c\;, \qquad 
 \Psi_f = \frac{\sigma^2}{2\eps} \pi D\nabla f\;.
\end{equation} 
Note in particular that 
\begin{equation}
 \sD(\Psi_{h_{\AB}}) = \frac{\sigma^2}{2\eps} \int_{\ABc} \nabla h_{\AB} 
\cdot D\nabla h_{\AB} \6\pi 
= \capacity(\A,\B)\;. 
\end{equation} 
$-\Psi_{h_{\AB}}$ is called the \emph{harmonic flow} from $\A$ to $\B$. 

\begin{lemma}
\label{lem:DPhiPsi} 
For all $f\in\sH^{\alpha,0}_{\AB}$ and $\ph\in\sF^\gamma_{\AB}$, we have 
\begin{equation}
\label{eq:lemma_D} 
 \sD(\Phi_f - \ph, \Psi_{h_{\AB}}) 
 = \alpha\capacity(\A,\B) + \gamma\;.
\end{equation} 
\end{lemma}

Lemma~\ref{lem:DPhiPsi} is all we need to prove the Dirichlet and Thomson 
principles.

\begin{prop}[Dirichlet principle]
\label{prop:Dirichlet} 
We have 
\begin{equation}
 \capacity(\A,\B) = \inf_{f\in\sH^{1,0}_{\AB}} \inf_{\ph\in\sF^0_{\AB}}
 \sD(\Phi_f - \ph)\;,
\end{equation} 
where the infimum is reached for $f=\bar f:=\frac12(h_{\AB} + h^*_{\AB})$ and 
$\ph=\bar\ph:=\Phi_{\bar f} - \Psi_{h_{\AB}}$. Furthermore, the bound
\begin{equation}
\label{eq:Dirichlet_defective} 
 \capacity(\A,\B) \leqs \inf_{f\in\sH^{1,0}_{\AB}} 
 \sD(\Phi_f - \ph) 
 - 2 \int_{\ABc} (\nabla\cdot\ph) h_{\AB} \6x\6y\;,
\end{equation} 
holds for any $\ph$ satisfying~\eqref{eq:flow_gamma} with $\gamma=0$. 
\end{prop}

\begin{prop}[Thomson principle]
\label{prop:Thomson} 
We have 
\begin{equation}
 \capacity(\A,\B) = \sup_{f\in\sH^{0,0}_{\AB}} \sup_{\ph\in\sF^1_{\AB}}
 \frac{1}{\sD(\Phi_f - \ph)}\;,
\end{equation} 
where the supremum is reached for $f=\bar f:=(h_{\AB} - 
h^*_{\AB})/(2\capacity(\A,\B))$ and $\ph=\bar\ph:=\Phi_{\bar f} - 
\Psi_{h_{\AB}}/\capacity(\A,\B)$. Furthermore, the bound 
\begin{equation}
\label{eq:Thomson_defective} 
 \capacity(\A,\B) \geqs \sup_{f\in\sH^{0,0}_{\AB}} 
 \frac{1}{\sD(\Phi_f - \ph)}
 \biggpar{1 + \int_{\ABc} (\nabla\cdot\ph) h_{\AB} \6x\6y}^2\;.
\end{equation} 
holds for any $\ph$ satisfying~\eqref{eq:flow_gamma} with $\gamma=1$.
\end{prop}


\section{The invariant measure}
\label{sec:invariant} 

This section is devoted to the proof of Theorem~\ref{thm:pi}. 
Lemma~\ref{lem:pot_invariant} shows that the invariant density $\pi(x,y)$ can be 
obtained by solving the \NB{Hamilton--Jacobi-type} equation~\eqref{eq:HJ}. 
However, it 
turns out to be difficult to obtain a good control of error terms when trying to 
do so. We thus use another approach instead, which consists in expanding 
$\pi(x,y)$ on the basis of eigenfunctions of $\sL_x^\dagger$, and analysing the 
resulting system of infinitely many coupled ODEs. In order to do so, we will 
need a number of bounds involving these eigenfunction, which we will derive in 
Section~\ref{ssec:invariant_static}. The actual proof of Theorem~\ref{thm:pi} 
will then be given in Section~\ref{ssec:invariant_proof}. 


\subsection{Eigenfunctions of the static system}
\label{ssec:invariant_static} 

The aim of this section is to obtain estimates on the eigenfunction $\phi_1$, 
and on the inner products 
\begin{align}
f_{nm}(y) &= \sigma^2 \pscal{\partial_y\pi_m}{\phi_n} \\
g_{nm}(y) &= \sigma^4 \pscal{\partial_{yy}\pi_m}{\phi_n}\;.
\label{eq:deg_fnm} 
\end{align}
Note that taking derivatives of the orthonormality relations 
$\pscal{\pi_m}{\phi_n} = \delta_{nm}$ yields 
\begin{align}
f_{nm}(y) &= -\sigma^2 \pscal{\pi_m}{\partial_y\phi_n} \\
g_{nm}(y) &= -\sigma^4 \pscal{\pi_m}{\partial_{yy}\phi_n}
-2\sigma^4 \pscal{\partial_y\pi_m}{\partial_y\phi_n}
=: -\ell_{nm}(y) - 2k_{nm}(y)\;.
\label{eq:orth_fnm} 
\end{align}
In particular, since $\phi_0$ is constant, $f_{0m}(y) = 0$ and $g_{0m}(y) = 0$ 
for all $m\in\N$. 
Using standard Laplace asymptotics (cf.~Appendix~\ref{app:Laplace}), it is 
rather easy to obtain estimates on integrals against $\pi_0$ up 
to multiplicative errors of the form $1+\Order{\sigma^2}$. In particular, the 
normalisation of $\pi_0(x\vert y)$ satisfies 
\begin{equation}
\label{eq:Z0_Laplace}
 Z_0(y) = \sqrt{\pi}\sigma 
 \biggbrak{\frac{1}{\omega_-(y)}\e^{2h_-(y)/\sigma^2} + 
\frac{1}{\omega_+(y)}\e^{2h_+(y)/\sigma^2}}
 \bigbrak{1+\Order{\sigma^2}}\;.
\end{equation}  
Similarly, the normalisation of the committor~\eqref{eq:h02} satisfies 
\begin{equation}
\label{eq:N_Laplace} 
 N(y) =  \frac{\sqrt{\pi}\sigma}{\omega_0(y)} \bigbrak{1+\Order{\sigma^2}}\;,
\end{equation} 
and bounds of the same type can be obtained for $f_{1i}$ and $g_{1i}$ for 
$i\in\set{0,1}$. We will, however, need much sharper estimates with 
exponentially small errors of order $\lambda_1(y)$, which requires more work.


\subsubsection{Eigenfunction $\phi_1$}
\label{sssec:phi1} 

We start by providing sharp estimates on the first eigenfunction $\phi_1$ of 
$\sL_x$ and its derivatives. 

Let $\tau_\pm = \tau_{x^*_{\pm}(y)}$ be the first-hitting times of $x^*_\pm(y)$ 
for the static SDE~\eqref{eq:static}, and let $\tau = \tau_-\wedge\tau_+$. The 
Feynman--Kac formula allows us to write 
\begin{align}
\phi_1(x\vert y) 
&= \bigexpecin{x}{\e^{\lambda_1(y) \tau} \phi_1(x_\tau)} \\
&= \phi_-(y) \bigexpecin{x}{\e^{\lambda_1(y) \tau} \indexfct{\tau_- < \tau_+}}
+ \phi_+(y) \bigexpecin{x}{\e^{\lambda_1(y) \tau} \indexfct{\tau_+ < \tau_-}} \\
&= \phi_-(y) \bigbrak{h_0(x\vert y) + h_1(x\vert y)} 
 + \phi_+(y) \bigbrak{1-h_0(x\vert y) + \bar h_1(x\vert y)}\;,
\label{eq:phi1_FK} 
\end{align}
where we use the shorthands $\phi_\pm(y) = \phi_1(x^*_{\pm}(y),y)$, 
while $h_0(x\vert y) = \probin{x}{\tau_- < \tau_+}$ is the committor and 
\begin{equation}
 h_1(x\vert y) = \bigexpecin{x}{(\e^{\lambda_1(y)\tau}-1) 
\indexfct{\tau_-<\tau_+}}\;, \qquad 
 \bar h_1(x\vert y) = \bigexpecin{x}{(\e^{\lambda_1(y)\tau}-1) 
\indexfct{\tau_+<\tau_-}}\;.
\end{equation} 
Recall that $h_0(x\vert y)$ is given by~\eqref{eq:h02} for 
$x\in(x^*_-(y),x^*_+(y))$. Furthermore, $h_0$ is constant equal to $1$ for 
$x<x^*_-(y)$, and constant equal to $0$ for $x>x^*_+(y)$. 

It will be convenient to \emph{define} $\Deltabar(y)$ by the relations 
\begin{equation}
\label{eq:defDeltabar} 
 \pscal{\pi_0}{h_0} = \frac{\e^{-\Deltabar(y)/\sigma^2}}
{\e^{-\Deltabar(y)/\sigma^2}+\e^{\Deltabar(y)/\sigma^2}}\;, 
\qquad 
 \pscal{\pi_0}{1-h_0} = \frac{\e^{\Deltabar(y)/\sigma^2}}
{\e^{-\Deltabar(y)/\sigma^2}+\e^{\Deltabar(y)/\sigma^2}}\;.
\end{equation} 
Indeed, standard Laplace asymptotics (see Lemma~\ref{lem:Laplace}) show 
that this definition is compatible to leading order with~\eqref{eq:Deltabar}. 
We further introduce 
\begin{equation}
 \label{eq:AB}
 A(y) = \tanh\biggpar{\frac{\Deltabar(y)}{\sigma^2}}\;, \qquad 
 B(y) = \frac{1}{\cosh(\Deltabar(y)/\sigma^2)}\;.  
\end{equation} 
Note carefully that $B(y) \in (0,1]$ may be exponentially small, and that we 
have the relations 
\begin{align}
 A(y)^2 + B(y)^2 &= 1\;, &
 \sigma^2 A'(y) &= \Deltabar'(y)B(y)^2\;, \\ 
 &&
 \sigma^2 B'(y) &= -\Deltabar'(y)A(y)B(y)\;. 
\end{align} 
Combining~\eqref{eq:N_Laplace} and~\eqref{eq:Z0_Laplace} with the 
expression~\eqref{eq:lambda1} of $\lambda_1(y)$, we 
obtain the very useful relation 
\begin{equation}
\label{eq:Z0Nlambda1} 
 Z_0(y) N(y) \lambda_1(y) = \frac{2\sigma^2}{B(y)^2} 
\bigbrak{1+\Order{\sigma^2}}\;.
\end{equation}
Finally, to lighten notations, we set 
\begin{equation}
 \ell(\sigma) = \lsig\;,
\end{equation}
and we will sometimes omit the argument $y$. 

The following results establish some properties of $h_0$, $h_1$ and $\phi_1$. 
Their proofs are postponed to Appendix~\ref{app:proof_phi1}.

\begin{prop}[Properties of $h_0$]
\label{prop:h0}
We have 
\begin{subequations}
\begin{align}
\label{eq:bound_dyh0} 
 \bigabs{\partial_y h_0(x\vert y)} \lesssim \frac{1}{\sigma^2}
 h_0(x\vert y) \bigpar{1-h_0(x\vert y)}\;, \\
 \label{eq:bound_dyyh0} 
 \bigabs{\partial_{yy} h_0(x\vert y)} \lesssim \frac{1}{\sigma^4}
 h_0(x\vert y) \bigpar{1-h_0(x\vert y)}\;.
\end{align} 
\end{subequations}
Furthermore, the inner product $\eta(y) = \pscal{\pi_0}{h_0(1-h_0)}$ satisfies 
\begin{equation}
\label{eq:bound_eta} 
0\leqs \eta(y) \lesssim \lambda_1(y) \ell(\sigma) B(y)^2\;.
\end{equation}
\end{prop}

\begin{prop}[Bounds on $h_1$]
\label{prop:h1} 
The remainder $h_1(x\vert \NB{y})$ satisfies the bounds 
\begin{align}
\label{eq:bound_h1} 
 \bigabs{h_1(x\vert y)} &\lesssim \lambda_1(y) \ell(\sigma) h_0(x\vert y) \\
\label{eq:bound_dyh1} 
 \bigabs{\partial_yh_1(x\vert y)} &\lesssim \frac{1}{\sigma^2}\lambda_1(y) 
\ell(\sigma)^2 h_0(x\vert y) \\
\label{eq:bound_dyyh1} 
 \bigabs{\partial_{yy}h_1(x\vert y)} &\lesssim \frac{1}{\sigma^4}\lambda_1(y) 
\ell(\sigma)^3 h_0(x\vert y)\;.
\end{align} 
Similar bounds hold for $\bar h_1(x\vert y)$, with $h_0$ replaced by 
$1-h_0$. 
\end{prop}

\begin{prop}[First eigenfunction]
\label{prop:phi_1} 
The coefficients $\phi_\pm(y)$ of $\phi_1(x\vert y)$ satisfy 
\begin{subequations}
\begin{align}
\label{eq:phipm} 
 \phi_\pm(y) &= \mp \e^{\mp\Deltabar(y)/\sigma^2}
 \bigbrak{1 + \bigOrder{\lambda_1(y) \ell(\sigma)}} \\
\label{eq:dyphipm} 
 \phi_\pm'(y) &= \mp \frac{1}{\sigma^2}\phi_\pm(y)
 \bigbrak{\Deltabar'(y) + \bigOrder{\lambda_1(y) \ell(\sigma)^2}} \\
\label{eq:dyyphipm} 
 \phi_\pm''(y) &= \mp \frac{1}{\sigma^4}\phi_\pm(y)
 \bigbrak{\Deltabar'(y)^2 \mp \sigma^2\Deltabar''(y) + \bigOrder{\lambda_1(y) 
\ell(\sigma)^3}}\;.
\end{align}
\end{subequations}
\end{prop}

Combining the last two propositions with~\eqref{eq:phi1_FK}, we obtain the 
following representations of $\phi_1$ and its derivatives:
\begin{subequations}
\begin{align}
\label{eq:phi1_rep} 
\phi_1 ={}& \phi_- h_0 \bigbrak{1 + \Order{\lambda_1\ell}}
+ \phi_+ (1-h_0) \bigbrak{1 + \Order{\lambda_1\ell}}\;, \\
\label{eq:dyphi1_rep} 
\partial_y \phi_1 ={}& 
\frac{\phi_-}{\sigma^2} h_0 \bigbrak{\Deltabar' + \Order{\lambda_1\ell^2}}
- \frac{\phi_+}{\sigma^2} (1-h_0) \bigbrak{\Deltabar' + \Order{\lambda_1\ell^2}}
+ (\phi_- - \phi_+) \partial_y h_0\;, \\
\partial_{yy} \phi_1 ={}& \frac{\phi_-}{\sigma^4} h_0 \bigbrak{(\Deltabar')^2 + 
\sigma^2 \Deltabar'' + \Order{\lambda_1\ell^3}}
- \frac{\phi_+}{\sigma^4} (1-h_0) \bigbrak{(\Deltabar')^2 
- \sigma^2 \Deltabar'' + \Order{\lambda_1\ell^3}} \\
&{}+ 2(\phi_-' - \phi_+') \partial_y h_0 + (\phi_- - \phi_+) \partial_{yy} 
h_0\;.
\label{eq:dyyphi1_rep} 
\end{align}
\end{subequations}
It is then straightforward to obtain the following expressions for inner 
products involving derivatives of the first two eigenfunctions.

\begin{prop}[Matrix elements involving $\phi_1$]
\label{prop:f10} 
We have 
\begin{align}
 f_{10}(y) &= -B(y) \bigbrak{\Deltabar'(y)  + 
\bigOrder{\lambda_1(y)\ell(\sigma)^2}}\;, \\
 f_{11}(y) &= -A(y)\Deltabar'(y) + \bigOrder{\lambda_1(y)\ell(\sigma)^2}\;, \\
 g_{10}(y) &= B(y) \bigbrak{2 A(y)\Deltabar'(y)^2 - \sigma^2 \Deltabar''(y) 
 + \bigOrder{\lambda_1(y)\ell(\sigma)^3}}\;, \\
 g_{11}(y) &= \bigpar{2A(y)^2-1} \Deltabar'(y)^2 - \sigma^2 A(y)\Deltabar''(y) 
 + \bigOrder{\lambda_1(y)\ell(\sigma)^3}\;. 
\end{align}
\end{prop}

A consequence of these estimates is that we have, for instance,
\begin{equation}
 \sigma^2\partial_y \phi_1
 = \bigpar{A(y)\phi_1 + B(y)}\Deltabar'(y) + R_1(x)\;,
\label{eq:proj_dyphi1} 
\end{equation} 
where $R_1$ is a remainder, dominated by the term in $\partial_y h_0$ 
in~\eqref{eq:dyphi1_rep}. One checks that it satisfies  
\begin{equation}
\label{eq:boundsR} 
 \pscal{\pi_0}{R_1} = 
\bigOrder{\lambda_1\ell^2B}\;, \qquad 
\pscal{\pi_0}{R_1^2} = 
\bigOrder{\lambda_1\ell^2}\;.
\end{equation}
In other words, $\partial_y\phi_1$ lies almost in the space spanned by $\phi_0$ 
and $\phi_1$. 

\begin{remark}
\label{rem:pi0absphi1} 
A useful observation is that the expression~\eqref{eq:phi1_rep} for $\phi_1$ 
implies
\begin{align}
 \pscal{\pi_0}{\abs{\phi_1}}
 &= \bigbrak{\phi_-(y)\pscal{\pi_0}{h_0} + \abs{\phi_+(y)}\pscal{\pi_0}{1-h_0}}
 \bigbrak{1+\Order{\lambda_1\ell}} \\
 &\leqs B(y)\bigbrak{1+\Order{\lambda_1\ell}}\;,
\label{eq:pi0absphi1} 
\end{align}
where we have used the definition~\eqref{eq:defDeltabar} of $\Deltabar$. 
This is often better than the bound $\pscal{\pi_0}{\abs{\phi_1}} \leqs 1$ 
provided by the Cauchy--Schwarz inequality. 
\end{remark}


\subsubsection{Bounds involving other eigenfunctions}
\label{sssec:phin} 

When analysing the system of ODEs giving the invariant density, we will also 
need a number of bounds involving other eigenfunctions than $\phi_0$ and 
$\phi_1$. All proofs of these bounds are postponed to 
Appendix~\ref{app:proof_phin}. We start with some simple $\ell^2$ estimates. 

\begin{prop}
\label{prop:sum_fnm}  
There exists a constant $M_0$, uniform in $\sigma$ and $n\geqs1$, such that
\begin{gather}
\label{eq:sum_fni} 
 \sum_{n\geqs0} f_{ni}(y)^2 \leqs M_0\;, \qquad 
 \sum_{n\geqs0} g_{ni}(y)^2 \leqs M_0 \qquad 
 \qquad \forall i\in\set{0,1}\;,\\
 \sum_{m\geqs0} f_{nm}(y)^2 \leqs M_0\;, \qquad 
 \sum_{m\geqs0} g_{nm}(y)^2 \leqs M_0 \qquad 
 \qquad \forall n\geqs1\;.
 \label{eq:sum_fnm2} 
\end{gather} 
\end{prop}

The following result shows that $h_0$ is almost orthogonal to the span of 
$\pi_0$ and $\pi_1$. 

\begin{prop}
\label{prop:pin_h0}
We have 
\begin{equation}
 \sum_{n\geqs2} \pscal{\pi_n}{h_0}^2 \lesssim \lambda_1(y)\ell(\sigma) B(y)^2\;.
\end{equation} 
\end{prop}

We can also get exponentially small bounds for a number of sums involving 
$f_{nm}$ and $g_{nm}$. 

\begin{prop}
\label{prop:f1mm1} 
The following sums are all of order $\lambda_1(y)\ell(\sigma)^a$ for some 
$a\leqs3$, where $i\in\set{0,1}$:  
\begin{gather}
 \sum_{m\geqs 2} f_{1m}(y)^2\;, \qquad
 \sum_{m\geqs 2} f_{1m}(y)f_{mi}(y)\;, \qquad
 \sum_{m\geqs 2} g_{1m}(y)^2\;, \\
 \sum_{m\geqs 2} f_{1m}(y)g_{mi}(y)\;, \qquad
 \sum_{m\geqs 2} g_{1m}(y)f_{mi}(y)\;, \qquad
 \sum_{m\geqs 2} g_{1m}(y)g_{mi}(y)\;.
\end{gather} 
\end{prop}

Finally, the following result provides exponentially small bounds on similar 
sums, but with all terms divided by $\lambda_n$. These bounds are \emph{not} 
consequences of the previous ones, since the terms of these sums do not have 
the same sign, so that their smallness is due to cancellations between terms. 

\begin{prop}
\label{prop:f1mm1lambdam} 
The following sums are all of order $\lambda_1(y)\ell(\sigma)^a$ for some 
$a\leqs3$, where $i\in\set{0,1}$:  
\begin{gather}
 \sum_{m\geqs 2} \frac{1}{\lambda_m(y)}f_{1m}(y)f_{mi}(y)\;, \qquad
 \sum_{m\geqs 2} \frac{1}{\lambda_m(y)}f_{1m}(y)g_{mi}(y)\;, \\
 \sum_{m\geqs 2} \frac{1}{\lambda_m(y)}g_{1m}(y)f_{mi}(y)\;, \qquad
 \sum_{m\geqs 2} \frac{1}{\lambda_m(y)}g_{1m}(y)g_{mi}(y)\;. 
\end{gather} 
\end{prop}


\subsection{Proof of Theorem~\ref{thm:pi}}
\label{ssec:invariant_proof} 

\NB{
We start by showing that the system of SDEs~\eqref{eq:SDE} admits a unique 
invariant measure.
\begin{prop}
For any $\sigma>0$, the system~\eqref{eq:SDE} admits a unique invariant 
probability measure $\pi$. Furthermore, the expectation of $x^2$ under 
$\pi$ is finite.  
\end{prop}
\begin{proof}
We use a probabilistic argument, following the technique developed by Meyn and 
Tweedie in~\cite{Meyn_Tweedie_1993b}. Consider the Lyapunov function 
\begin{equation}
 U(x,y) = x^2\;.
\end{equation} 
Applying the generator $\sL$ to $U$ yields 
\begin{equation}
 (\sL U)(x,y) = \frac{1}{\eps} \bigbrak{\sigma^2 + 2 x b(x,y)}\;. 
\end{equation} 
By the assumptions on $b$ given at the beginning of Section~\ref{ssec:setup}, 
there exist constants $M_1, M_2 > 0$ such that 
\begin{equation}
 (\sL U)(x,y) \leqs \frac{1}{\eps} \bigbrak{\sigma^2 + M_1 - M_2 U(x,y)}\;. 
\end{equation} 
In particular, $\sL U$ is bounded above, and strictly negative outside a 
compact region in $\T\times\R$. It thus follows 
from~\cite[Theorem~3.3]{Meyn_Tweedie_1993b} that the stochastic process is 
Harris recurrent, meaning that it will visit any set with positive 
measure infinitely often. This in turn implies the existence of a unique 
invariant measure~\cite{Azema_Duflo_Revuz69,Getoor_79}. 
The claim on $x^2 = U(x,y)$ having finite expectation follows 
from~\cite[Theorem~4.2]{Meyn_Tweedie_1993b} with the test function $f$ equal to 
the Lyapunov function $U$. In particular, this shows that the measure $\pi$ can 
be normalised to a probability measure. 
\end{proof}
}

\NB{It thus remains to obtain the precise asymptotics of $\pi$ stated 
in Theorem~\ref{thm:pi}.}
Since for each $y$, the eigenfunctions $\pi_n(\cdot\vert y)$ form a complete 
orthonormal basis of $L^2(\R,\pi_0\6x)$, we can decompose the density $\pi$ of 
the invariant measure as 
\begin{equation}
\label{eq:pi} 
 \pi(x,y) = \sum_{n\geqs 0} \alpha_n(y) \pi_n(x\vert y)
 = \pi_0(x\vert y) \sum_{n\geqs 0} \alpha_n(y) \phi_n(x\vert y)\;.
\end{equation} 
We write the adjoint generator as $\sL^\dagger = \frac1\eps\sL^\dagger_x + 
\sL^\dagger_y$, where $\sL^\dagger_x$ has been defined in~\eqref{eq:Lx}, and 
\begin{equation}
 \sL^\dagger_y \mu = -\partial_y\mu + \rsii \partial_{yy}\mu\;.
\end{equation} 
The stationarity condition $\sL^\dagger \pi = 0$ becomes 
\begin{equation}
\label{eq:Lpieq0} 
 \sum_{n\geqs1} \lambda_n(y)\alpha_n(y)\pi_n(x\vert y) 
 = \eps\sum_{n\geqs0} \sL^\dagger_y \bigpar{\alpha_n(y)\pi_n(x\vert y)}\;,
\end{equation} 
where the right-hand side can be evaluated using 
\begin{align}
 \sL^\dagger_y (\alpha_n\pi_n) 
 &= \Bigpar{-\alpha_n' + \rsii \alpha_n''} \pi_n 
 + \Bigpar{-\alpha_n + \rsi \alpha_n'} \partial_y\pi_n
 + \rsii \alpha_n \partial_{yy} \pi_n\;.
\end{align}
We now project~\eqref{eq:Lpieq0} on each eigenfunction $\phi_n$. 
Since $\pscal{\partial_y\pi_n}{\phi_0} = \partial_y\pscal{\pi_n}{\phi_0} = 0$, 
and similarly for the second derivative, the projection on $\phi_0$ yields 
\begin{equation}
 -\alpha_0'(y) + \rsii \alpha_0''(y) = 0\;.
\end{equation} 
Using periodicity in $y$ and the fact that $\pi$ is normalised, one easily gets 
\begin{equation}
 \alpha_0(y) = 1\;.
\end{equation} 
The projections on the remaining $\phi_n$ result in the following statement, 
whose proof is a simple computation. 

\begin{lemma}
The stationary distribution $\pi$ is given by~\eqref{eq:pi} with $\alpha_0(y) 
= 1$ and $\set{\alpha_n(y)}_{n\in\N}$ given by the first component of the 
unique periodic solution of
\begin{align}
\label{eq:alphan_betan} 
 \rsi \alpha_n' &= 2\alpha_n - 2\beta_n \\
 \sigma^2 \beta_n' &= -\frac{\sigma^2}{\eps}\lambda_n(y)\alpha_n 
 + \sum_{m\geqs1} \bigbrak{c_{nm}(y)\alpha_m + d_{nm}(y)\beta_m}
 + c_{n0}(y)\;,
\end{align}
where
\begin{align}
 c_{n0}(y) &= - f_{n0}(y) + \frac{\varrho^2}{2} g_{n0}(y)\;, \\
 c_{nm}(y) &= f_{nm}(y) + \frac{\varrho^2}{2} g_{nm}(y)\;,
 & m&\geqs1\;,\\
 d_{nm}(y) &= -2f_{nm}(y)\;.
\end{align}
\end{lemma}


\subsubsection{The first-order case}
\label{sssec:invariant_order1} 

It is instructive to consider first the case $\varrho^2 = 0$. Then $\beta_n(y) 
= \alpha_n(y)$, and $\alpha_n(y)$ satisfies the linear inhomogeneous system 
\begin{equation}
\label{eq:alphan} 
 \eps \alpha_n' 
 = -\lambda_n(y) \alpha_n - \frac{\eps}{\sigma^2} f_{n0}(y) 
 - \frac{\eps}{\sigma^2} \sum_{m\geqs1} f_{nm}(y) 
\alpha_m\;.
\end{equation} 
Note that for $n\geqs2$, $\alpha_n(y)$ is a fast variable, which, by the 
general theory of singularly perturbed ordinary differential equations, is 
expected to remain $\eps$-close to a value $\alpha_n^*(y)$ such that the 
right-hand side of the system vanishes. 

The case $n=1$, however, is special since $\lambda_1(y)$ is exponentially 
small. This makes the system hard to study in the form~\eqref{eq:alphan}, 
because $\alpha_1(y)$ can become exponentially large. 
The solution is to observe that, disregarding for a moment the terms $\alpha_n$ 
with $n\geqs2$ we have by Lemma~\ref{lem:Laplace}
\begin{align}
 p_-(y)
 := \bigprob{x(y) < x^*_0(y)} 
 &\simeq \int_{-\infty}^{x^*_0(y)} \pi_0(x\vert y) \bigbrak{1 + \alpha_1(y) 
\phi_1(x\vert y)}\6x \\
 &= \frac{1}{2} B(y) \bigpar{\e^{-\Deltabar/\sigma^2} + \alpha_1(y)}
 \bigbrak{1+\Order{\sigma^2}}\;.
 \label{eq:pminus} 
\end{align}
This suggests setting 
\begin{equation}
\label{eq:def_alpha1} 
 \alpha_1(y) = \frac{A(y) - \delta_1(y)}{B(y)}\;,
\end{equation} 
so that $p_-(y) \simeq \frac12(1-\delta_1(y))$, and therefore $\delta_1(y)$ 
remains of order $1$. Then a computation shows that 
\begin{equation}
 \label{eq:delta1} 
 \eps \delta_1' 
 = \Bigbrak{-\lambda_1(y) + \frac{\eps}{\sigma^2}p_1(y)} 
\bigpar{\delta_1-A(y)} + \frac{\eps}{\sigma^2}w_1(y) 
 + \frac{\eps}{\sigma^2} B(y)\sum_{m\geqs2} f_{1m}(y) 
\alpha_m\;,
\end{equation}
where
\begin{align}
p_1(y) &= -f_{11}(y) - \Deltabar'(y)A(y)
= \Order{\lambda_1(y)\ell(\sigma)^2}\;, \\
w_1(y) &= \Deltabar'(y)B(y)^2 + B(y)f_{10}(y)
= \Order{\lambda_1(y)\ell(\sigma)^2B(y)^2}\;.
\label{eq:p1w1} 
\end{align}
The unique periodic solution of this equation is given by 
\begin{equation}
 \delta_1(y) 
 = \frac{1}{\eps(1-\e^{-\bar\Lambda_1(1,0)/\eps})}
 \int_y^{y+1} \e^{-\bar\Lambda_1(y+1,\bar y)/\eps}
 \biggbrak{\bar\lambda_1(\bar y)A(\bar y) 
 + \frac{\eps}{\sigma^2}w_1(\bar y)
 + \frac{\eps}{\sigma^2}\tilde w_1(\bar y)} \6\bar y\;, 
\end{equation} 
where we have set  $\bar\lambda_1(y) = \lambda_1(y) \NB{{}-{}} 
\frac{\eps}{\sigma^2}p_1(y)$ and 
\begin{align}
\label{eq:def_w1tilde} 
 \tilde w_1(y) &
= B(y) \sum_{m\geqs2} f_{1m}(y) \alpha_m(y)\;, \\
 \bar\Lambda_1(y_2,y_1) &= \int_{y_1}^{y_2} \bar\lambda_1(y)\6y\;.
\end{align} 
In particular, for $\eps\gg\bar\Lambda_1(1,0) = \avrg{\bar\lambda_1}$, 
$\delta_1(y)$ is almost constant, that is, we have 
\begin{align}
 \delta_1(y) &= \bar{\delta}_1
\biggbrak{1+\biggOrder{\frac{\avrg{\bar\lambda_1}}{\eps}}}\;, \\
\bar{\delta}_1 &= 
\frac{1}{\avrg{\bar\lambda_1}}
\int_0^1 \biggbrak{\bar\lambda_1(y)A(y) 
 + \frac{\eps}{\sigma^2}w_1(\bar y)
 + \frac{\eps}{\sigma^2}\tilde w_1(y)} \6y\;.
\end{align} 
To analyse the dynamics of the remaining coefficients $\alpha_n(y)$ with 
$n\geqs2$, we introduce a vector $\alpha_\perp^*(y)$ with components
\begin{equation}
\label{eq:def_alphaperp} 
 \alpha_n^*(y) = -\frac{\eps}{\sigma^2} \frac{1}{\lambda_n(y)}
 \Bigbrak{f_{n0}(y) + \frac{A(y)-\delta_1(y)}{B(y)} f_{n1}(y)}\;,
\end{equation} 
and examine in particular the behaviour of $\alpha_\perp^1(y) = \alpha_\perp(y) 
- \alpha_\perp^*(y)$. 

\begin{prop}
\label{prop:alphastar_alpha1} 
The unique periodic solution of the system~\eqref{eq:alphan} satisfies 
\begin{equation}
\label{eq:alphan_alphastar_alpha1} 
 \alpha_n(y) = \alpha_n^*(y) + \alpha_n^1(y)\;,
\end{equation} 
where 
\begin{align}
\label{eq:sup_alphastar} 
 \sup_{n\geqs2} \lambda_n(y) \bigabs{\alpha_n^*(y)} 
 \lesssim \frac{\eps}{\sigma^2 B(y)}
 &&\forall y\in[0,1]\;, \\
  \sup_{n\geqs2} \lambda_n(y) \bigabs{\alpha_n^1(y)} 
 \lesssim \frac{\eps^2}{\sigma^4 B(y)}
 &&\forall y\in[0,1]\;.
 \label{eq:sup_alpha1} 
\end{align} 
\end{prop}
\begin{proof}
The bound~\eqref{eq:sup_alphastar} is a direct consequence of the 
bound~\eqref{eq:sum_fni} on the sum of $f_{ni}^2$. 
In order to establish~\eqref{eq:sup_alpha1}, we first note that the 
$\alpha_n^1$ satisfy the equation
\begin{equation}
\label{eq:alpha1prime} 
 \eps(\alpha_n^1)' 
 = -\lambda_n(y) \alpha_n^1 
 - \frac{\eps}{\sigma^2} \sum_{m\geqs2} f_{nm}(y) \bigpar{\alpha_m^*(y) + 
\alpha_m^1} - \eps\alpha_n^*(y)'\;.
\end{equation} 
We will show that the set 
\begin{equation}
 H = \Bigsetsuch{(\alpha_\perp^1,y)}{\bigabs{\alpha_m^1} \leqs 
\frac{\eps^2C_0}{\sigma^4B(y)\lambda_m(y)} \; \forall m\geqs2}
\end{equation} 
is invariant under the flow of~\eqref{eq:alpha1prime} for sufficiently large 
$C_0$. Assume $\alpha_\perp^1$ belongs to $\partial H$, and pick $n$ such that 
$\alpha_n^1 = \pm (\eps^2C_0)/(\sigma^4\lambda_nB)$. The Cauchy--Schwarz 
inequality yields 
\begin{equation}
 \biggpar{\sum_{m\geqs2} f_{nm}(\alpha_m^* + \alpha_m^1)}^2 
 \leqs \sum_{m\geqs2} f_{nm}^2 \sum_{m\geqs2} (\alpha_m^* + \alpha_m^1)^2
 \leqs \frac{\eps^2C_1}{B^2\sigma^4} \biggpar{1 + \frac{\eps^2C_0^2}{\sigma^4}}
 \;,
\end{equation}
for a constant $C_1$, where we have used~\eqref{eq:sum_fnm2} to bound the first 
sum, and~\eqref{eq:sup_alphastar} and the definition of $H$ to bound the second 
one. 
The derivative of $\alpha_n^*(y)$ can be bounded using the relations 
\begin{equation}
 \sigma^2 f_{ni}'(y) = g_{ni}(y) + k_{ni}(y)\;, \qquad 
 \sigma^2 \biggpar{\frac{A-\delta_1}{B}}' 
 = \frac{\Deltabar'(1-A\delta_1)-\sigma^2\delta_1'}{B}
 = \bigOrder{B^{-1}}
\end{equation} 
and the Hellmann--Feynman theorem (cf.~\eqref{eq:Hellmann-Feynman}), which 
shows that $\sigma^2\lambda_n'(y)$ has order $1$. The result is that 
\begin{equation}
 \bigabs{(\alpha_n^*)'(y)} \leqs \frac{\eps C_2}{\sigma^4B(y)\lambda_n(y)}
\end{equation} 
for a constant $C_2$. Plugging these bounds into~\eqref{eq:alpha1prime} shows 
that for $C_0$ large enough, the sign of $\eps(\alpha_n^1)'$ is the opposite of 
the sign of $\alpha_n^1$. This shows the invariance of $H$, and therefore the 
bound~\eqref{eq:sup_alpha1}.  
\end{proof}

\begin{cor}
\label{cor:w1tilde} 
The error term $\tilde w_1(y)$ introduced in~\eqref{eq:def_w1tilde} satisfies 
\begin{equation}
\label{eq:w1tilde} 
 \bigabs{\tilde w_1(y)} 
 \lesssim \frac{\eps}{\sigma^2} \lambda_1(y) \ell(\sigma)^3
 + \frac{\eps^2}{\sigma^4} \sqrt{\lambda_1(y) \ell(\sigma)^3}
\end{equation} 
uniformly in $y\in[0,1]$. 
\end{cor}
\begin{proof}
This follows directly from the 
decomposition~\eqref{eq:alphan_alphastar_alpha1}. 
Indeed, the contribution of the $\alpha_m^*$ can be bounded via 
Proposition~\ref{prop:f1mm1lambdam}, and yields the first term on the 
right-hand 
side. The contribution of the $\alpha_m^1$ can be bounded via the 
Cauchy--Schwarz inequality, using Proposition~\ref{prop:f1mm1} 
and~\eqref{eq:sup_alpha1}. 
\end{proof}

One consequence of this result that will be useful when estimating capacities 
is the following. Integrating the ODE~\eqref{eq:delta1} satisfied by 
$\delta_1(y)$ over one period, and using~\eqref{eq:p1w1} 
and~\eqref{eq:w1tilde}, we obtain
\begin{equation}
\label{eq:avrg_ODE} 
 \bigabs{\avrg{\lambda_1[\delta_1-A]}}
 \lesssim \frac{\eps}{\sigma^2} \avrg{\lambda_1}\ell^2 
 + \frac{\eps^2}{\sigma^4} \avrg{\lambda_1}\ell^3 
 + \frac{\eps^3}{\sigma^6} \avrg{\sqrt{\lambda_1}}\ell^{3/2}\;.
\end{equation}
Proposition~\ref{prop:alphastar_alpha1} also allows us to control the remainder 
$\Phi_\perp$ of the invariant measure. For $\sharp \in \set{\;,*,1}$, let us 
write 
\begin{equation}
 \Phi_\perp^{\sharp}(x,y) 
 = \sum_{n\geqs 2} \alpha_n^{\sharp}(y) \phi_n(x\vert y)\;.
\end{equation} 

\begin{cor}
\label{cor:bound_Phiperp} 
Let $\cD = (x^*_-(y),x^*_+(y))$. The $L^2$-bounds 
\begin{subequations}
\begin{align}
\label{eq:Phiperp_L2} 
 \pscal{\pi_0}{(\Phi_\perp^*)^2}^{1/2}
 &\lesssim \frac{\eps}{\sigma^2B(y)}\;, 
 &\pscal{\pi_0}{(\Phi_\perp^1)^2}^{1/2}
 &\lesssim \frac{\eps^2}{\sigma^4B(y)}\;, \\
 \pscal{\pi_0}{(\partial_x\Phi_\perp)^2\indicator{\cD}}^{1/2}
 &\lesssim \frac{\eps}{\sigma^4B(y)}\;, 
 &\pscal{\pi_0}{(\partial_x\Phi_\perp^1)^2\indicator{\cD}}^{1/2}
 &\lesssim \frac{\eps^2}{\sigma^6B(y)}
\label{eq:dxPhiperp_L2} 
\end{align}
\end{subequations}
hold for all $y\in[0,1]$. 
Furthermore, the bounds 
\begin{equation}
\label{eq:Phiperp_sup} 
 \bigabs{\Phi_\perp^*(x,y)} 
 \lesssim 
\frac{\eps\e^{V_0(x,y)/\sigma^2}}{\sigma^{3/2}B(y)^2\sqrt{\lambda_1(y)}}\;, 
\qquad 
 \bigabs{\Phi_\perp^1(x,y)} 
 \lesssim 
\frac{\eps^2\e^{V_0(x,y)/\sigma^2}}{\sigma^{7/2}B(y)^2\sqrt{\lambda_1(y)}}
\end{equation} 
hold for all $x\in\R$ and all $y\in[0,1]$.
\end{cor}
\begin{proof}
The first two $L^2$-bounds follow directly from the fact that 
\begin{equation}
 \pscal{\pi_0}{(\Phi_\perp^*)^2}
 = \norm{\alpha_\perp^*}_{\ell^2}^2
 = \sum_{n\geqs2} (\alpha_m^*)^2\;,
\end{equation} 
while the $L^2$-bound on the derivative is a consequence of 
Lemma~\ref{lem_dxPhi}.  As for the $L^\infty$-bounds~\eqref{eq:Phiperp_sup}, 
they follow from the fact that $\sL_x$ is conjugated to a Schr\"odinger 
operator (cf.~\eqref{eq:Schrodinger} in Appendix~\ref{app:proof_phin}), whose 
eigenfunctions $\psi_n$ are bounded by a constant of order $1$, so that 
\begin{equation}
 \bigabs{\Phi_\perp^*(x,y)}
 = \frac{1}{\sqrt{\pi_0(x\vert y)}}
 \biggabs{\sum_{n\geqs 2} \alpha_n^*(y) \psi_n(x\vert y)}
 \lesssim \frac{1}{\sqrt{\pi_0(x\vert y)}}
 \norm{\alpha_\perp^*}_{\ell^1}\;.
\end{equation} 
The $\ell^1$-norm of $\alpha_\perp^*$ can be bounded using the previous 
proposition. 
\end{proof}

Part of the importance of $\Phi_\perp^*$ lies in the following estimate, which 
shows that functions bounded by $h_0(1-h_0)$ are almost orthogonal to 
$\Phi_\perp^*$, and thus allows to improve a certain number of error bounds 
when 
estimating the capacity. Its proof is close in spirit to the proof of 
Proposition~\ref{prop:sum_fnm}, so we also give it in 
Appendix~\ref{app:proof_phin}.

\begin{prop}
\label{prop:Phistar}
Let $f$ be supported on $\cD=(x^*_-(y),x^*_+(y))$, and satisfy either one of 
the bounds 
\begin{equation}
 \bigabs{f(x)} \leqs M h_0(x\vert y) \bigpar{1 - h_0(x\vert y)}
 \qquad \text{or} \qquad 
  \bigabs{f(x)} \leqs M \e^{2V_0(x,y)/\sigma^2}
\end{equation} 
for all $x\in \cD$, and for some constant $M > 0$. Then 
\begin{equation}
 \bigabs{\pscal{\pi_0}{\Phi_\perp^* f}} \lesssim \frac{\eps}{\sigma^2} 
\lambda_1(y)\ell(\sigma) M\;.
\end{equation} 
\end{prop}


\subsubsection{The second-order case}
\label{sssec:invariant_order2} 

Consider now the case $\varrho^2 > 0$. We again carry out the change of 
variables~\eqref{eq:def_alpha1}, and in addition set 
\begin{equation}
 \beta_1(y) = \frac{A(y) - \delta_1(y)-\hat\gamma_1(y)}{B(y)}\;, 
 \qquad 
 \hat\gamma_1(y) = \frac{\sigma}{\sqrt{\eps}}\gamma_1(y)
 + \frac{\varrho^2}{2} \Deltabar'(y) \bigbrak{1-A(y)\delta_1(y)}\;.
\end{equation} 
The resulting system for $(\delta_1,\gamma_1)$ is given by 
\begin{align}
\label{eq:delta1_gamma1} 
 \sqrt{\eps}\sigma \delta_1'
 &= -\frac{2}{\varrho^2} \gamma_1 \\
 \sqrt{\eps}\sigma \gamma_1'
 &= \Bigbrak{-\lambda_1(y) + \frac{\eps}{\sigma^2}p_1(y)} 
 \bigpar{\delta_1-A(y)} 
 + \frac{\varrho^2}{2}\frac{\sqrt{\eps}}{\sigma} q_1(y)\gamma_1 
 + \frac{\eps}{\sigma^2} \bigbrak{w_1(y) + \tilde w_1(y)}\;,
\end{align}
where 
\begin{align}
p_1(y) ={}& -f_{11}(y) + \frac{\varrho^2}{2}g_{11}(y) 
- \Deltabar'(y)A(y) \\
&{}+ \frac{\varrho^2}{2} \Bigbrak{\Deltabar'(y)^2 + 2 
\Deltabar'(y)A(y)f_{11}(y) + \sigma^2 \Deltabar''(y)A(y)}\;, \\
q_1(y) ={}& 1 - \varrho^2 \Bigbrak{f_{11}(y) + \Deltabar'(y)A(y)}\;, \\
w_1(y) ={}& B(y)^2 \Bigbrak{\Deltabar'(y)(1-\varrho^2 f_{11}(y)) - 
\frac{\varrho^2\sigma^2}{2}\Deltabar''(y)}
+ B(y)\Bigbrak{f_{10}(y) - \frac{\varrho^2}{2}g_{10}(y)}\;,
\end{align}
and the contribution of the other variables is contained in the term 
\begin{equation}
 \tilde w_1(y) 
 = B(y) \sum_{m\geqs2} \Bigbrak{f_{1m}(y) (2\beta_m(y) - \alpha_m(y)) 
 - \frac{\varrho^2}{2}g_{1m}(y) \alpha_m(y)}\;.
\end{equation} 
It follows from Proposition~\ref{prop:f10} that 
\begin{align}
p_1(y) &= \Order{\lambda_1(y)\ell(\sigma)^3}\;, \\
q_1(y) &= 1 + \Order{\lambda_1(y)\ell(\sigma)^2}\;, \\
w_1(y) &= \Order{\lambda_1(y)\ell(\sigma)^3B(y)^2}\;.
\end{align}
It is straightforward to check that the system~\eqref{eq:delta1_gamma1} is 
equivalent to the second-order equation 
\begin{equation}
 \frac{\varrho^2}{2} \eps\sigma^2\delta_1''
 - \eps q_1(y)\delta_1' 
 + \Bigbrak{-\lambda_1(y) + \frac{\eps}{\sigma^2}p_1} 
\bigpar{\delta_1-A(y)} + \frac{\eps}{\sigma^2}\bigbrak{w_1(y) + \tilde w_1(y)}
= 0\;. 
\end{equation} 
By a standard argument of singular perturbation theory (see for 
instance~\cite[Example~2.1.3]{Berglund_Gentz_book}), the solutions of this 
equation are close, up to multiplicative errors $1+\Order{\varrho^4\sigma^4}$, 
to those of the first-order equation~\eqref{eq:delta1}. 

In order to analyse the behaviour of the remaining coefficients $\alpha_n(y)$ 
and $\beta_n(y)$ with $n\geqs2$, we introduce, analogously 
to~\eqref{eq:def_alphaperp}, 
\begin{equation}
 \alpha^*_n(y)
 = \frac{\eps}{\sigma^2} \frac{1}{\lambda_n(y)} 
 \Bigbrak{c_{n0}(y) + c_{n1}(y)\alpha_1(y) + d_{n1}(y)\beta_1(y)}\;.
\end{equation} 

\begin{figure}
\begin{center}
\begin{tikzpicture}[>=stealth',point/.style={circle,inner 
sep=0.035cm,fill=white,draw},x=1.5cm,y=1.5cm, 
declare function={u(\x,\y) = 0.08*(\x - \y + 0.2);
v(\x,\y) = 0.2*(-2*\x + 0.1);}]

\draw[thick,blue,fill=blue!20] (-0.7,-1.5) rectangle (0.7,1.5);
\draw[thin,dashed,blue] (-1.5,-1.5) -- (1.5,1.5);

\draw[->,semithick] (-2,0) -> (2.0,0);
\draw[->,semithick] (0,-2.0) -> (0,2.2);

\foreach \x in {-0.7,-0.6,...,0.7} 
  \draw[->,semithick,purple] ({\x}, 1.5) 
  -- ({\x + u(\x, 1.5)}, {1.5 + v(\x,1.5)});

\foreach \x in {-0.7,-0.6,...,0.7} 
  \draw[->,semithick,purple] ({\x}, -1.5) 
  -- ({\x + u(\x, -1.5)}, {-1.5 + v(\x,-1.5)});

\foreach \y in {-1.5,-1.3,...,1.5} 
  \draw[->,semithick,purple] (0.7, {\y}) 
  -- ({0.7 + u(0.7, {\y})}, {\y + v(0.7, {\y})});

\foreach \y in {-1.5,-1.3,...,1.5} 
  \draw[->,semithick,purple] (-0.7, {\y}) 
  -- ({-0.7 + u(-0.7, {\y})}, {\y + v(-0.7, {\y})});

\node[] at (1.7,0.2) {$\alpha_n^1$};
\node[] at (0.2,1.95) {$\beta_n^1$};
\node[blue] at (-0.5,1.3) {$H$};

\end{tikzpicture}
\vspace{-5mm}
\end{center}
\caption[]{Vector field~\eqref{eq:vector_field_alpha_beta} on the boundary of 
the set $H$, shown for a fixed component $n$ and fixed $y$. The broken line 
shows the approximate location of the points where $(\alpha_n^1)'$ changes sign.
}
\label{fig:alpha_beta}
\end{figure}
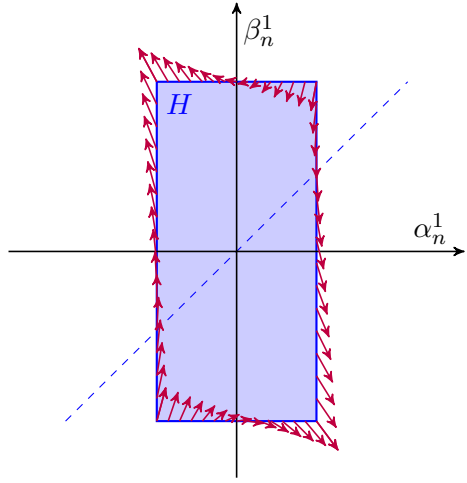

\begin{prop}
The unique periodic solution of the system~\eqref{eq:alphan_betan} satisfies 
\begin{align}
 \alpha_n(y) &= \alpha_n^*(y) + \alpha_n^1(y)\;, \\
 \beta_n(y) &= \alpha_n^*(y) + \beta_n^1(y)\;,
\end{align} 
where 
\begin{subequations}
\begin{align}
\label{eq:sup_alphabetastar} 
 \sup_{n\geqs2} \lambda_n(y) \bigabs{\alpha_n^*(y)} 
 \lesssim \frac{\eps}{\sigma^2 B(y)}
 &&\forall y\in[0,1]\;, \\
\label{eq:sup_alpha1b} 
  \sup_{n\geqs2} \lambda_n(y) \bigabs{\alpha_n^1(y)} 
 \lesssim \frac{\eps^2}{\sigma^4 B(y)}
 &&\forall y\in[0,1]\;, \\
  \sup_{n\geqs2} \lambda_n(y) \bigabs{\beta_n^1(y)} 
 \lesssim \frac{\eps}{\sigma^2 B(y)}
 &&\forall y\in[0,1]\;.
\label{eq:sup_beta1b} 
\end{align} 
\end{subequations}
\end{prop}
\begin{proof}
The bound~\eqref{eq:sup_alphabetastar} is again a direct consequence 
of~\eqref{eq:sum_fni}, noting that
\begin{equation}
  \alpha^*_n(y)
 = -\frac{\eps}{\sigma^2} \frac{1}{\lambda_n} 
 \biggbrak{f_{n0} - \frac{\varrho^2}{2} g_{n0} 
 + \frac{A-\delta_1}{B} \Bigpar{f_{n1} - \frac{\varrho^2}{2}g_{n1}} 
 - \frac{2}{B} f_{n1}\hat\gamma_1}\;.
\end{equation} 
To show~\eqref{eq:sup_alpha1b} and~\eqref{eq:sup_beta1b}, we will use the fact 
that the pairs 
$(\alpha_n^1,\beta_n^1)$ satisfy the system
\begin{align}
\label{eq:vector_field_alpha_beta} 
\eps(\alpha_n^1)'
&= \frac{2\eps}{\varrho^2\sigma^2}
\bigbrak{\alpha_n^1 - \beta_n^1 - \sigma^2(\alpha_n^*)'}
\\
\eps(\beta_n^1)'
&= -\lambda_n\alpha_n^1 
+ \frac{\eps}{\sigma^2} \sum_{m\geqs2} 
\bigbrak{c_{nm}(\alpha_m^*+\alpha_m^1) + d_{nm}(\alpha_m^*+\beta_m^1)}
- \eps(\alpha_n^*)'\;.
\end{align}
We will argue that the unique periodic solution of this equation has to be 
entirely contained in the set
\begin{equation}
 H = \biggsetsuch{(\alpha_\perp^1,\beta_\perp^1,y)}
 {\bigabs{\alpha_m^1} \leqs \frac{\eps^2C_0}{\sigma^4B(y)\lambda_m(y)}, 
 \bigabs{\beta_m^1} \leqs \frac{\eps C_0}{\sigma^2B(y)\lambda_m(y)} 
\; \forall m\geqs2}\;,
\end{equation} 
provided $C_0$ is a sufficiently large constant of order $1$. 
Indeed, similar estimates as in the proof of 
Proposition~\ref{prop:alphastar_alpha1} show that whenever 
$(\alpha_n^1,\beta_n^1,y)$ lies in $H$, one has 
\begin{align}
\eps(\alpha_n^1)'
&= \frac{2\eps}{\varrho^2\sigma^2}
\biggbrak{\alpha_n^1 - \beta_n^1 + \biggOrder{\frac{\eps}{\sigma^2 
B(y)\lambda_n(y)}}}
\\
\eps(\beta_n^1)'
&= -\lambda_n\alpha_n^1 
+ \biggOrder{\frac{\eps^2(1+C_0)}{\sigma^4 B(y)}}\;.
\end{align}
For sufficiently large $C_0$, this vector field has the following properties on 
the boundary $\partial H$ (\figref{fig:alpha_beta}): 
\begin{itemize}
\item 	on the upper boundary of $H$, it points to the left, and 
changes from pointing outward $H$ to pointing inward as $\alpha_n^1$ increases;
\item 	on the right boundary of $H$, it points downward, and 
changes from pointing outward $H$ to pointing inward as $\beta_n^1$ increases;
\item 	the situation is reversed on the lower and left boundaries of $H$. 
\end{itemize}
Combined with the fact that the equation is linear, these properties imply that 
a solution leaving $H$ cannot enter it again. Therefore, the unique periodic 
solution has to lie within $H$.
\end{proof}

As $\alpha_\perp^*$ and $\alpha_\perp^1$ satisfy the same bounds as for 
$\varrho^2 = 0$, it is straightforward to check that 
Corollary~\ref{cor:w1tilde}, Corollary~\ref{cor:bound_Phiperp} and 
Proposition~\ref{prop:Phistar} still hold in the present case. 


\section{Adjoint process and committors}
\label{sec:adjoint} 

Recall from Lemma~\ref{lem:pot_invariant} that the invariant density can be 
written as $\pi(x,y) = Z^{-1} \e^{-2V(x,y)/\sigma^2}$. Since we also have 
\begin{equation}
\label{eq:def_Phi} 
 \pi(x,y) = \pi_0(x\vert y) \Phi(x,y)\;,
 \qquad
  \Phi(x,y) = 1 + \sum_{n\geqs1} \alpha_n(y)\phi_n(x\vert y)\;,
\end{equation} 
solving for $V$ gives the expression 
\begin{equation}
\label{eq:V} 
 V(x,y) = V_0(x,y) - \frac{\sigma^2}{2} \log\Phi(x,y)
 + \frac{\sigma^2}{2} \log\frac{Z_0(y)}{Z}\;.
\end{equation} 
One can get a better idea of the difference between $V$ and $V_0$ by writing 
\begin{equation}
 \Phi(x,y) = \Phi_0(x,y) + \Phi_\perp(x,y)\;, 
\end{equation} 
where
\begin{align}
 \Phi_0(x,y)
 &= 1 + \alpha_1(y) \phi_1(x\vert y) \\
 &= 1 + \frac{A(y)-\delta_1(y)}{B(y)} 
 \Bigbrak{\phi_+(y) + \bigpar{\phi_-(y)-\phi_+(y)}h_0(x\vert y)}
 \bigbrak{1 + \Order{\lambda_1\ell}}\;.
\end{align}
Note that by~\eqref{eq:phipm} we have 
\begin{equation}
\label{eq:B2Phi0} 
 B(y)^2 \Phi_0(x,y) 
 = \bigbrak{1 - A(y)\delta_1(y) + (2h_0(x\vert y)-1)(A(y)-\delta_1(y))}
 \bigbrak{1 + \Order{\lambda_1\ell}}\;.
\end{equation}
The $x$-component of the vector field $c$, defined in~\eqref{eq:defc}, is thus 
given by 
\begin{equation}
\label{eq:cx} 
 c_x = \frac{1}{\eps} \bigpar{\partial_x V - \partial_x V_0}
 = -\frac{\sigma^2}{2\eps} \frac{\partial_x\Phi(x,y)}{\Phi(x,y)}\;,
\end{equation} 
and the adjoint SDE has the form~\eqref{eq:SDE2_adjoint} with 
\begin{align}
 b^*(x,y) &= -\partial_x V_0(x,y) 
+ \sigma^2\frac{\partial_x\Phi(x,y)}{\Phi(x,y)} \\
&= -\partial_x V_0^*(x,y) \;,
\end{align} 
where we have defined the adjoint potential by 
\begin{equation}
\label{eq:V0star} 
 V_0^*(x,y) = V_0(x,y) - \sigma^2 \log\Phi(x,y)\;. 
\end{equation} 
Using the expression~\eqref{eq:phi1_rep} for $\phi_1$, and approximating 
$\Phi(x,y)$ by $\Phi_0(x,y)$, one can deduce from~\eqref{eq:V} 
and~\eqref{eq:V0star} that 
\begin{align}
 V(x^*_0(y),y) - V(x^*_\pm(y),y)
 &\simeq \frac{h_-(y) + h_+(y)}{2}\;,\\
 V_0^*(x^*_0(y),y) - V_0^*(x^*_\pm(y),y)
 &\simeq h_\mp(y)\;.
\end{align} 
In other words, the potential well depths are symmetrised for $V$, and inverted 
for the adjoint potential $V_0^*$ with respect to the initial potential $V_0$ 
(see also~\figref{fig:potential}). 
%

We will need some a priori estimates on the committors $h_{\AB}$ and 
$h^*_{\AB}$. \NB{Recall from~\eqref{eq:def_ab} that the boundaries of $\A$ and 
$\B$ are given by functions $a(y) = x^*_-(y) + \hat\rho$ and $b(y) = x^*_+(y) - 
\hat\rho$ for a small constant $\hat\rho > 0$.}
We expect $h_{\AB}$ to be close to the static committor $\smash{\tilde h_0}$ 
given for $a(y) < x < b(y)$ by 
\begin{equation}
\label{eq:def_htilde0} 
 \tilde h_0(x\vert y) = \frac{1}{\tilde N(y)} 
\int_x^{b(y)}\e^{2V_0(\bar x,y)/\sigma^2}\6\bar x\;, 
 \qquad 
 \tilde N(y) = \int_{a(y)}^{b(y)} \e^{2V_0(x,y)/\sigma^2}\6x\;.
\end{equation} 
Note that $\tilde h_0$ only slightly differs from $h_0$ owing to the different 
boundary conditions. The difference is however exponentially small in 
$\sigma^2$, with an exponent that can be made large by taking $\hat\rho$ 
in~\eqref{eq:def_ab} small. Similarly, $h^*_{\AB}$ should be close to 
\begin{align}
\label{eq:def_htilde0_star} 
 \tilde h_0^*(x\vert y) &= \frac{1}{\tilde N^*(y)} 
\int_x^{b(y)}\e^{2V_0^*(\bar x,y)/\sigma^2} \6\bar x
 &
 \tilde N^*(y) &= \int_{a(y)}^{b(y)} \e^{2V_0^*(x,y)/\sigma^2}\6x \\
 &= \frac{1}{\tilde N^*(y)} 
\int_x^{b(y)}\frac{\e^{2V_0(\bar x,y)/\sigma^2}}{\Phi(\bar x,y)^2} \6\bar x\;,
 &
 &= \int_{a(y)}^{b(y)} \frac{\e^{2V_0(x,y)/\sigma^2}}{\Phi(x,y)^2}\6x\;.
\end{align} 

\begin{prop}
\label{prop:bounds_hAB} 
We have 
\begin{equation}
 h_{\AB}(x,y) = \tilde h_0(x\vert y) + g(x,y)\;, \qquad 
 h^*_{\AB}(x,y) = \tilde h_0^*(x\vert y) + g^*(x,y)\;, 
\end{equation} 
where 
\begin{equation}
\label{eq:bound_g2} 
 \pscal{\pi_0}{g^2}^{1/2} 
 \lesssim \frac{\eps}{\sigma^{3/2}\sqrt{Z_0(y)}}
 \lesssim \frac{\eps\sqrt{\lambda_1(y)}B(y)}{\sigma^2}\;,
\end{equation} 
and similarly for $\pscal{\pi_0}{(g^*)^2}^{1/2}$.
\end{prop}
\begin{proof}
Since $\sL_x\tilde h_0 = 0$, the function $g$ satisfies the equation 
\begin{equation}
 \eps\sL_y g = -\sL_x g - \eps\sL_y \tilde h_0
\end{equation} 
with Dirichlet boundary conditions. Consider first the case $\varrho^2=0$, in 
which $\sL_y = \partial_y$, and define the Lyapunov function 
\begin{equation}
 \sV(g) = \frac12 \pscal{\pi_0}{g^2}\;.
\end{equation} 
Changing $y$ into $-y$, we obtain 
\begin{align}
 \eps\partial_y \sV
 &= \pscal{\pi_0g}{\eps\partial_y g} + \frac12\eps\pscal{\partial_y\pi_0}{g^2} 
\\
 &= \pscal{\pi_0g}{\sL_x g} + \eps\pscal{\pi_0g}{\partial_y\tilde h_0}
 + \frac{\eps}{\sigma^2} \pscal{\pi_0g}{Wg} \\
 &\leqs -c_1 \sV + \eps\sqrt{\sV} \pscal{\pi_0}{(\partial_y\tilde h_0)^2}^{1/2} 
 + \frac{\eps}{\sigma^2} c_2 \sV\;,
\end{align}
where $c_1>0$ is a constant of order $1$ related to the spectral gap of $\sL_x$ 
with Dirichlet boundary conditions on $\tilde\cD=(a(y),b(y))$. Using the 
bounds~\eqref{eq:proof_h0} and~\eqref{eq:bound_h0} on $\partial_y \tilde h_0$ 
obtained in the proof of Proposition~\ref{prop:h0}, we get 
\begin{equation}
 \pscal{\pi_0}{(\partial_y \tilde h_0)^2} 
 \lesssim \frac{1}{\sigma^2Z_0(y)} \int_{a(y)}^{b(y)} 
\frac{\e^{2V_0(x,y)/\sigma^2}}{(\sigma+\abs{x-x^*_0(y)})^2} \6x 
 \lesssim \frac{1}{\sigma^3 Z_0(y)}\;.
\end{equation} 
It follows that 
\begin{equation}
 \eps\partial_y \sV \leqs -\bar{c_1} \sV + \frac{\eps\bar 
c_2}{\sigma^{3/2}\sqrt{Z_0(y)}} \sqrt{\sV}
\end{equation} 
for some constant $\bar c_1, \bar c_2 > 0$. The result then follows by applying 
Gronwall's inequality to $\sW = \sqrt{\sV}$. 

It thus remains to deal with the case $\varrho^2>0$. To this end, we introduce 
\begin{equation}
 k(x,y) = \frac{\sqrt{\eps}}{\sigma} \biggpar{g(x,y) + 
\frac{\varrho^2\sigma^2}{2}\partial_y g(x,y)}\;. 
\end{equation} 
Then the pair $(g,k)$ satisfies the system of hyperbolic type 
\begin{align}
\frac{\varrho^2}{2} \sqrt{\eps}\sigma \partial_y g 
&= k - \frac{\sqrt{\eps}}{\sigma} g \\
\sqrt{\eps}\sigma  \partial_y k
&= -\sL_x g - \eps \partial_y \tilde h_0\;.
\end{align}
This system can be made isotropic via a shearing transformation  
\begin{equation}
 k(x,y) = \frac{\varrho}{\sqrt2} (-\sL_x)^{1/2} \bar k(x,y)\;,
\end{equation} 
where for any $\gamma\in\R$, we set, in terms of eigenvalues $\bar\lambda_n$ 
and 
eigenfunctions $\bar\pi_n$ and $\bar\phi_n$ of $\sL_x$ with Dirichlet boundary 
conditions on $\tilde\cD$,
\begin{equation}
 (-\sL_x)^\gamma f 
 = \sum_{n\geqs1} (-\bar\lambda_n)^\gamma \pscal{\bar\pi_n}{f} \bar\phi_n\;.
\end{equation}
This results in the system 
\begin{align}
 \frac{\sqrt{\eps}\varrho\sigma}{\sqrt2} \partial_y g 
&= (-\sL_x)^{1/2}\bar k - \frac{\sqrt{2\eps}}{\varrho\sigma} g \\
 \frac{\sqrt{\eps}\varrho\sigma}{\sqrt2} \partial_y \bar k
&= (-\sL_x)^{1/2} g - \eps (-\sL_x)^{1/2} \partial_y \tilde h_0
- \frac{\sqrt{\eps}\varrho\sigma}{\sqrt2}(-\sL_x)^{-1/2} \partial_y 
(-\sL_x)^{1/2}\bar k\;.
\end{align}
The result then follows in a similar way as above, by working with the Lyapunov 
functions 
\begin{equation}
 \sV_\pm(g,\bar k) = \frac12\pscal{\pi_0}{(g \pm \bar k)^2}\;, 
\end{equation} 
and showing that for a periodic solution of the system, both $\sV_+$ and 
$\sV_-$ have to remain small. 
\end{proof}

\begin{cor}
\label{cor:hstarAB_pi} 
We have 
\begin{equation}
 \int_{\B^c} h^*_{\AB} \6\pi 
 = \frac{1}{2} \int_0^1 \biggbrak{1-\delta_1(y) 
 + \BigOrder{\frac{\eps}{\sigma^2}\ell(\sigma)\sqrt{\lambda_1(y)}}} 
\6y \,\bigbrak{1+\Order{\sigma^2}}\;.
\end{equation} 
\end{cor}
\begin{proof}
The definition~\eqref{eq:def_Phi} of $\Phi$ implies 
\begin{equation}
 \int_{\B^c} h^*_{\AB} \6\pi 
 = \int_0^1 \pscal{\pi_0}{\Phi h^*_{\AB}} \6y\;.
\end{equation} 
We decompose 
\begin{equation}
 \pscal{\pi_0}{\Phi h^*_{\AB}}
 = \pscal{\pi_0}{\Phi_0 \tilde h_0^*}
 + \pscal{\pi_0}{\Phi_0 g^*}
 + \pscal{\pi_0}{\Phi_\perp \tilde h_0^*} 
 + \pscal{\pi_0}{\Phi_\perp g^*}\;,
\end{equation} 
and estimate the contribution of each term separately. A similar argument as 
in~\eqref{eq:pminus} shows that
\begin{equation}
 \pscal{\pi_0}{\Phi_0 \tilde h_0^*} 
 = \frac12 \bigbrak{1-\delta_1(y)} \bigbrak{1 + \Order{\sigma^2}}\;.
\end{equation} 
The other terms can be bounded via the Cauchy--Schwarz inequality. Namely, we 
obtain 
\begin{align}
\bigabs{\pscal{\pi_0}{\Phi_0 g^*}} &\leqs 
\pscal{\pi_0}{\Phi_0^2}^{1/2} \pscal{\pi_0}{(g^*)^2}^{1/2}
\lesssim\frac{1}{B(y)} \frac{\eps}{\sigma^2} \sqrt{\lambda_1(y)}B(y)\;,\\
\bigabs{\pscal{\pi_0}{\Phi_\perp \tilde h_0^*}} &\leqs 
\pscal{\pi_0}{\Phi_\perp^2}^{1/2} \pscal{\pi_0}{(\tilde h_0^*)^2}^{1/2}
\lesssim\frac{\eps}{\sigma^2B(y)} \sqrt{\lambda_1(y)\ell(\sigma)B(y)^2}\;,\\ 
\bigabs{\pscal{\pi_0}{\Phi_\perp g^*}} &\leqs 
\pscal{\pi_0}{\Phi_\perp^2}^{1/2} \pscal{\pi_0}{(g^*)^2}^{1/2}
\lesssim\frac{\eps}{\sigma^2B(y)} \frac{\eps}{\sigma^2} \sqrt{\lambda_1(y)}B(y) 
\;,
\end{align}
where we have used Proposition~\ref{prop:pin_h0} (which applies to $\tilde h_0$ 
as well), Corollary~\ref{cor:bound_Phiperp} and 
Proposition~\ref{prop:bounds_hAB}. 
\end{proof}


\section{Estimating the capacity}
\label{sec:capacity} 

Recall from~\eqref{eq:def_avrg} the notation 
\begin{equation}
 \avrg{f} = \int_0^1 f(y)\6y\;,
\end{equation}
and define 
\begin{equation}
 C_0 = \frac{1}{4\eps} \avrg{\lambda_1 \brak{1 - A\delta_1}}\;.
\end{equation} 
The purpose of this section is to establish the following estimates on the 
capacity $\capacity(\A,\B)$. 

\begin{theorem}[Estimate of the capacity]
\label{thm:capacity} 
There exist constants $M_\pm$ such that 
\begin{align}
\frac{\capacity(\A,\B)}{C_0} 
&\leqs 1 + M_+ \biggbrak{\sigma^2 + 
\frac{\eps\ell(\sigma)\sigma^{-2}\avrg{\lambda_1} + 
\eps^2\sigma^{-3}\avrg{\sqrt{\lambda_1}}}
{\avrg{\lambda_1[1-A\delta_1]}}}\;, \\
\frac{\capacity(\A,\B)}{C_0} 
&\geqs 1 - M_- \biggbrak{\sigma^2 + 
\frac{\eps\ell(\sigma)^2\sigma^{-2}\avrg{\lambda_1} + 
\eps^2\sigma^{-{7/2}}\avrg{\sqrt{\lambda_1}}}
{\avrg{\lambda_1[1-A\delta_1]}}}\;.
\end{align}
\end{theorem}

The proof of this result is naturally divided into two parts. We will prove the 
upper bound in Section~\ref{ssec:capacity_upper}, and the lower bound in 
Section~\ref{ssec:capacity_lower}. 


\subsection{Upper bound on the capacity}
\label{ssec:capacity_upper} 

The upper bound on the capacity will follow from the defective-flow Dirichlet 
principle~\eqref{eq:Dirichlet_defective}. The expressions for the minimisers 
given in Proposition~\ref{prop:Dirichlet} suggest taking as test functions 
\begin{align}
 f &= \frac{1}{2} (\tilde h_0 + \tilde h_0^*)\;, \\
 \ph &= \Phi_f - \Psi_{\tilde h_0}
 = \Psi_{(\tilde h_0^* - \tilde h_0)/2} - \pi f c\;.
\end{align}
The defective-flow Dirichlet principle then reads  
\begin{equation}
 \capacity(\A,\B) \leqs \sD(\Psi_{\tilde h_0})
 - 2 \int_{\ABc} (\nabla\cdot\ph) h_{\AB} \6x\6y\;.
\end{equation} 

\begin{prop}
There exists a constant $M_+$ such that 
\begin{equation}
 \sD(\Psi_{\tilde h_0}) \leqs 
 \frac{1}{4\eps} 
\biggbrak{ \avrg{\lambda_1\brak{1-A\delta_1}} 
+ \frac{\eps\ell(\sigma) M_+}{\sigma^2} \avrg{\lambda_1} 
+ \frac{\eps^2 \sqrt{\ell(\sigma)}M_+}{\sigma^4} \avrg{\sqrt{\lambda_1}}} 
\bigbrak{1+\Order{\sigma^2}}\;.
\end{equation} 
\end{prop}
\begin{proof}
We introduce a probability measure $\mu$ on $\tilde\cD = (a(y),b(y))$ with 
density (see~\eqref{eq:def_htilde0})
\begin{equation}
 \mu(x\vert y) = \frac{1}{\tilde N(y)} \e^{2V_0(x,y)/\sigma^2}
 = -\partial_x \tilde h_0(x)\;.
\label{eq:def_mu} 
\end{equation} 
Note that 
\begin{equation}
 \pscal{\mu}{h_0} 
 = -\frac{N(y)}{\tilde N(y)} \int_{a(y)}^{b(y)} h_0 \partial_x h_0 \6x 
 = \frac{1}{2} \bigbrak{1 + \Order{\e^{-\kappa/\sigma^2}}}\;,
\label{eq:muh0} 
\end{equation} 
where $\kappa$ can be made large by taking $\hat\rho$ in~\eqref{eq:def_ab} 
large. 
Applying the definition~\eqref{eq:def_D} of $\sD$ to $\Psi_{\tilde h_0}$, we 
obtain 
\begin{equation}
 \sD(\Psi_{\tilde h_0}) 
 = \frac{\sigma^2}{2\eps} 
 \int_0^1 \bigbrak{\pscal{\pi}{(\partial_x \tilde h_0)^2} 
 + \eps\varrho^2 \pscal{\pi}{(\partial_y \tilde h_0)^2}}\6y\;.
\end{equation} 
In order to estimate the first inner product, we use the decomposition $\pi = 
\pi_0(\Phi_0 + \Phi_\perp^* + \Phi_\perp^1)$, expand, and consider the resulting 
terms separately. For the first term,  using~\eqref{eq:phi1_rep} to compute 
$\pscal{\mu}{\phi_1}$, we find 
\begin{align}
\pscal{\pi_0}{\Phi_0(\partial_x\tilde h_0)^2}
&= \frac{1}{\tilde N(y)Z_0(y)} \bigbrak{1 + \alpha_1(y)\pscal{\mu}{\phi_1}} \\
&= \frac{1}{\tilde N(y)Z_0(y)} \frac{1-A(y)\delta_1(y)}{B(y)^2}
\bigbrak{1 + \Order{\e^{-\kappa/\sigma^2}}} \\
&= \frac{1}{2\sigma^2} \lambda_1(y)\bigbrak{1-A(y)\delta_1(y)}
\bigbrak{1 + \Order{\sigma^2}}\;.
\end{align}
The second term can be directly bounded via Proposition~\ref{prop:Phistar} by 
\begin{equation}
 \bigabs{\pscal{\pi_0}{\Phi_\perp^*(\partial_x \tilde h_0)^2}} 
 \lesssim \frac{\eps}{\sigma^4} \lambda_1(y) \ell\;.
\end{equation} 
As for the third term, it satisfies 
\begin{align}
 \bigabs{\pscal{\pi_0}{\Phi_\perp^1(\partial_x \tilde h_0)^2}} 
 &\leqs \pscal{\pi_0}{(\Phi_\perp^1)^2}^{1/2} \pscal{\pi_0}{(\partial_x \tilde 
h_0)^4}^{1/2} \\
 &\lesssim \frac{\eps^2}{\sigma^4 B(y)} \frac{1}{Z_0(y)^{1/2}\tilde N(y)^{3/2}}
 \lesssim \frac{\eps^2}{\sigma^6} \sqrt{\lambda_1(y)}\;.
\end{align} 
It remains to estimate the contribution of $\pscal{\pi}{(\partial_y 
\tilde h_0)^2}$. We split it into two parts, which satisfy
\begin{equation}
 \bigabs{\pscal{\pi_0}{\Phi_0(\partial_y\tilde h_0)^2}}
 \lesssim \frac{1}{B(y)^2} \pscal{\pi_0}{(\partial_y\tilde h_0)^2}
 \lesssim \frac{\eta(y)}{\sigma^4 B(y)^2}
 \lesssim \frac{\lambda_1(y)\ell}{\sigma^2}\;,
\end{equation} 
(where we used the sharper estimate~\eqref{eq:IJ_sharper} to bound 
$\partial_y\tilde h_0$), and 
\begin{equation}
 \bigabs{\pscal{\pi_0}{\Phi_\perp(\partial_y\tilde h_0)^2}}
 \leqs \pscal{\pi_0}{\Phi_\perp^2}^{1/2} 
 \pscal{\pi_0}{(\partial_y\tilde h_0)^4}^{1/2} 
 \lesssim \frac{\eps}{\sigma^5} \sqrt{\lambda_1(y)\ell}\;.
\end{equation} 
Collecting all terms gives the claimed result. 
\end{proof}

To complete the proof of the upper bound on the capacity, it remains to control 
the error due to the fact that $\ph$ is not exactly divergence-free. Note that 
in view of~\eqref{eq:divergence}, we have 
\begin{equation}
 \nabla \cdot (\pi f c) = \pi (\nabla f \cdot c)\;, 
\end{equation} 
This yields 
\begin{equation}
\label{eq:nabla_ph} 
 -2 \nabla \cdot \ph 
 = \biggbrak{\frac{\sigma^2}{2\eps} \nabla\cdot (\pi D\nabla \tilde h_0) 
 + \pi \nabla\tilde h_0 \cdot c}
 - \biggbrak{\frac{\sigma^2}{2\eps} \nabla\cdot (\pi D\nabla \tilde h_0^*) 
 - \pi \nabla\tilde h_0^* \cdot c}\;.
\end{equation}
The contributions of the two brackets to the error term can be estimated 
separately. They are small because $\tilde h_0$ and $\tilde h_0^*$ are both 
approximately harmonic with respect to $\sL$ and $\sL^*$. 

\begin{prop}
We have the bound 
\begin{equation}
 \biggabs{-2\int_{\ABc} (\nabla\cdot\ph) h_{\AB} \6x\6y} 
 \lesssim \frac{\ell(\sigma)}{\sigma^2} \avrg{\lambda_1}
 + \frac{\eps}{\sigma^3} \avrg{\sqrt{\lambda_1}}\;.
\label{eq:defective2} 
\end{equation} 
\end{prop}
\begin{proof}
We will consider the contribution of the first bracket in~\eqref{eq:nabla_ph}.
The expression~\eqref{eq:cx} for $c_x$ shows that the derivatives with respect 
to $x$ cancel exactly, while the expression~\eqref{eq:defc} for $c_y$ shows 
that the remaining part is equal to 
\begin{equation}
 \frac12\varrho^2\sigma^2 \partial_y(\pi\partial_y\tilde h_0)
 + \pi \partial_y \tilde h_0 c_y
 = \pi \Bigbrak{\partial_y \tilde h_0 + \frac{\varrho^2\sigma^2}{2} 
\partial_{yy}\tilde h_0}\;.
\end{equation} 
The first error term is thus given by 
\begin{equation}
 \int_0^1 \pscal{\pi_0}{\Phi\Bigbrak{\partial_y \tilde h_0 + 
\frac{\varrho^2\sigma^2}{2} 
\partial_{yy}\tilde h_0}h_{\AB}}\6y\;.
\end{equation} 
Bounding $h_{\AB}$ by $1$, and using
\begin{align}
\pscal{\pi_0}{\abs{\Phi_0} \abs{\partial_y\tilde h_0}} 
&\lesssim \frac{\eta(y)}{\sigma^2 B(y)^2}
\lesssim \frac{1}{\sigma^2} \lambda_1(y)\ell\;, \\
\pscal{\pi_0}{\abs{\Phi_\perp} \abs{\partial_y\tilde h_0}}
&\leqs \pscal{\pi_0}{(\Phi_\perp)^2}^{1/2} 
\pscal{\pi_0}{(\partial_y\tilde h_0)^2}^{1/2} 
\lesssim \frac{\eps}{\sigma^3} \sqrt{\lambda_1(y)}\;,
\end{align}
we find that the contribution of $\partial_y \tilde h_0$ satisfies the claimed 
bound. The contribution of $\partial_{yy}\tilde h_0$ is bounded similarly. The 
proves the result for the first bracket in~\eqref{eq:nabla_ph}, and the proof 
for the second bracket is similar. 
\end{proof}


\subsection{Lower bound on the capacity}
\label{ssec:capacity_lower} 

To obtain a lower bound on the capacity, we will apply the defective Thomson 
principle~\eqref{eq:Thomson_defective}. Since~\eqref{eq:cx} implies that the 
drift terms $b$ and $b^*$ are close to each other for $x$ near $x^*_0(y)$, 
Proposition~\ref{prop:Thomson} suggests taking $f=0$ and a test flow $\ph$ 
approximately proportional to $-\smash{\Psi_{\tilde h_0}}$, where $\tilde h_0$ 
is the static committor~\eqref{eq:def_htilde0}. In fact, 
by~\eqref{eq:Z0Nlambda1} the $\ex$-component of $\Psi_{\tilde h_0}$ is close to 
$-\lambda_1(y) B(y)^2 \Phi(x,y)$. We thus choose as test flow 
\begin{equation}
\label{eq:test_flow} 
 \ph(x,y) 
 = \frac{1}{4\eps C} \lambda_1(y) B(y)^2 \Phi(x,y) \ex\;,
\end{equation} 
where the constant $C$ is chosen in such a way that the unit flux 
condition
\begin{equation}
 \int_{\partial \A} (\ph\cdot\nn) \6\lambda = -1
\end{equation}
is met. This amounts to requiring
\begin{align}
 4\eps C 
 ={}& \int_0^1 \lambda_1(y) B(y)^2 \Phi(a(y),y) \6y \\ 
 ={}& \int_0^1 \lambda_1(y) B(y)^2 \bigbrak{\Phi_0(a(y),y) + 
\Phi_\perp(a(y),y)} 
\6y \\ 
 ={}& \int_0^1 \lambda_1(y) \biggbrak{1 - A(y)\delta_1(y)  
+ \biggOrder{\frac{\eps\ell}{\sigma^{2}}}
+ \biggOrder{\frac{\eps^3\ell^{3/2}}{\sigma^{6}\sqrt{\lambda_1(y)}}}} \6y 
\\
&{}+ \biggOrder{\frac{\eps}{\sigma^{3/2}} \int_0^1 \sqrt{\lambda_1(y)} 
\e^{-h_-(y)/\sigma^2}\6y}\\
 ={}& \avrg{\lambda_1 \brak{1 - A\delta_1}}  
+ \biggOrder{\frac{\eps\ell}{\sigma^{2}}\avrg{\lambda_1}}
+ \biggOrder{\frac{\eps^3\ell^{3/2}}{\sigma^{6}}\avrg{\sqrt{\lambda_1}}}\;,
\end{align}
where we have used the expression~\eqref{eq:B2Phi0} for $B^2\Phi_0$, the 
expression~\eqref{eq:avrg_ODE} for $\avrg{\lambda_1 [A - \delta_1]}$, 
Corollary~\ref{cor:bound_Phiperp} to estimate the contribution of $\Phi_\perp$, 
as well as \eqref{eq:lambda1}. 

\begin{prop}
There exists a constant $M_-$ such that 
\begin{equation}
 \sD(-\ph) \leqs 
 \frac{1}{4\eps C^2} \biggbrak{\avrg{\lambda_1\brak{1-A\delta_1}} + 
\frac{\eps\ell(\sigma)^2 M_-}{\sigma^2} \avrg{\lambda_1} 
+ \frac{\eps^2 M_-}{\sigma^{7/2}} \avrg{\sqrt{\lambda_1}}} 
\bigbrak{1+\Order{\sigma^2}}\;.
\end{equation} 
\end{prop}
\begin{proof}
Substituting the expression~\eqref{eq:test_flow} of the test flow in the 
definition~\eqref{eq:def_D} of $\sD$ and using~\eqref{eq:Z0Nlambda1}, we 
obtain 
\begin{align}
\sD(-\ph) 
&= \frac{1}{8\eps\sigma^2C^2} \int_0^1 \lambda_1(y)^2 B(y)^4 
\int_{a(y)}^{b(y)}  \frac{\Phi(x,y)^2}{\pi(x,y)} \6x\6y \\
&= \frac{1}{4\eps C} 
\int_0^1 \lambda_1(y) B(y)^2 
\pscal{\mu}{\Phi} \6y \bigbrak{1+\Order{\sigma^2}}\;, 
\end{align}
where $\mu$ is the probability density introduced in~\eqref{eq:def_mu}.  
The leading contribution comes from the term (cf.~\eqref{eq:B2Phi0})
\begin{equation}
 \pscal{\mu}{\Phi_0}
 = \frac{1}{B(y)^2} \Bigbrak{1 - A(y)\delta_1(y) + (A(y) - \delta_1(y)) 
\pscal{\mu}{2h_0-1}} \bigbrak{1+\Order{\lambda_1\ell}}\;.
\end{equation}
Therefore, when integrating against $\lambda_1(y)B(y)^2$, \eqref{eq:muh0} 
and~\eqref{eq:avrg_ODE} imply that the contribution of the term in $(A(y) - 
\delta_1(y))$ satisfies the claimed bound. Furthermore, it follows from 
Proposition~\ref{prop:Phistar} that 
\begin{equation}
\label{eq:bound_muPhiperp} 
 \bigabs{\pscal{\mu}{\Phi_\perp^*}} 
 = \frac{Z_0(y)}{\tilde N(y)} 
 \bigabs{\pscal{\pi_0}{\e^{4V_0/\sigma^2}\Phi_\perp^*}}
 \lesssim \frac{\eps}{\sigma^2} \frac{Z_0(y)\lambda_1(y)\ell}{\tilde N(y)}
 \lesssim \frac{\eps\ell}{\sigma^2 B(y)^2}\;.
\end{equation}
Combining this with the bound 
\begin{equation}
 \bigabs{\pscal{\mu}{\Phi_\perp^1}} \lesssim 
\frac{\eps^2}{\sigma^{7/2}B(y)^2\sqrt{\lambda_1(y)}}\;,
\end{equation} 
which follows from~\eqref{eq:Phiperp_sup}, yields the result. 
\end{proof}

Since the test flow $\ph$ is not exactly divergence-free, to complete the 
proof of the lower bound it remains to control the error term 
in~\eqref{eq:Thomson_defective}. 

\begin{prop}
The error term satisfies 
\begin{equation}
 \biggabs{\int_{\ABc} (\nabla\cdot\ph) h_{\AB} \6x\6y} 
 \lesssim \frac{1}{4\eps C} 
 \biggbrak{\frac{\eps\ell(\sigma)^2}{\sigma^2} \avrg{\lambda_1}
 + \frac{\eps^2}{\sigma^5} \avrg{\sqrt{\lambda_1}}}\;.
\label{eq:defective1} 
\end{equation}
\end{prop}
\begin{proof}
The definition~\eqref{eq:test_flow} of the test flow implies
\begin{equation}
 \int_{\ABc} (\nabla\cdot\ph) h_{\AB} \6x\6y
 = \frac{1}{4\eps C} \int_0^1 \lambda_1(y) B(y)^2 
\pscal{\partial_x\Phi}{h_{\AB}} \6y\;.
\end{equation} 
We decompose the inner product as
\begin{equation}
\label{eq:sum_pscal} 
\pscal{\partial_x\Phi}{h_{\AB}}
= \pscal{\partial_x\Phi_0}{\tilde h_0}
+ \pscal{\partial_x\Phi_0}{g}
+ \pscal{\partial_x\Phi_\perp^*}{\tilde h_0}
+ \pscal{\partial_x\Phi_\perp^1}{\tilde h_0}
+ \pscal{\partial_x\Phi_\perp}{g}
\end{equation} 
and estimate the resulting terms separately. For the first term, we note that 
\begin{equation}
 \pscal{\partial_x\Phi_0}{\tilde h_0}
 = 2\frac{A(y)-\delta_1(y)}{B(y)^2} \int_{a(y)}^{b(y)} \tilde h_0 \partial_x 
h_0\6x \bigbrak{1+\Order{\lambda_1\ell}}\;.
\end{equation} 
As in~\eqref{eq:muh0} above, the integral is exponentially close to $\frac12$, 
which results in a negligible term when integrating against 
$\lambda_1(y)B(y)^2$, owing to~\eqref{eq:avrg_ODE}. 

Regarding the second term, we observe that 
\begin{equation}
 \bigabs{\pscal{\partial_x\Phi_0}{g}} 
 \lesssim \frac{\abs{A(y)-\delta_1(y)}}{B(y)^2} \bigabs{\pscal{\partial_x 
h_0}{g}}\;, 
\end{equation}
where 
\begin{equation}
 \pscal{\partial_x h_0}{g}^2 
 \leqs \pscal{\pi_0^{-1}}{(\partial_x h_0)^2} 
 \pscal{\pi_0}{g^2}
 \leqs \frac{Z_0(y)}{N(y)} \frac{\eps^2}{\sigma^3 Z_0(y)}
 \lesssim \frac{\eps^2}{\sigma^4}
\end{equation} 
by~\eqref{eq:bound_g2}. The contribution of this term is thus of the order of 
$(\eps/\sigma^2)\avrg{\lambda_1}$. 

The third term in~\eqref{eq:sum_pscal} satisfies 
(cf.~\eqref{eq:bound_muPhiperp})
\begin{equation}
 \pscal{\partial_x\Phi_\perp^*}{\tilde h_0}^2
 = \pscal{\Phi_\perp^*}{\partial_x\tilde h_0}^2
 = \pscal{\mu}{\Phi_\perp^*}^2
\lesssim \frac{\eps^2\ell^2}{\sigma^4B(y)^4}\;,
\end{equation} 
which results in  a contribution of order 
$(\eps\ell/\sigma^2)\avrg{\lambda_1}$. 
The fourth term in~\eqref{eq:sum_pscal} can be 
bounded using~\eqref{eq:Phiperp_L2} by 
\begin{equation}
 \pscal{\partial_x\Phi_\perp^1}{\tilde h_0}^2
 = \pscal{\Phi_\perp^1}{\partial_x\tilde h_0}^2
 \leqs \pscal{\pi_0}{(\Phi_\perp^1)^2} \pscal{\pi_0^{-1}}{(\partial_x\tilde 
h_0)^2} 
\lesssim \frac{\eps^4}{\sigma^8B(y)^2} \frac{Z_0(y)}{N(y)}\;
\end{equation}
Using~\eqref{eq:Z0Nlambda1} and integrating against $\lambda_1 B^2$, we obtain 
indeed a quantity bounded by the second term on the right-hand side 
of~\eqref{eq:defective1}, though with a slightly better power of $\sigma$. For 
the last term in~\eqref{eq:sum_pscal}, we use the quick-and-dirty bound 
\begin{equation}
 \pscal{\partial_x\Phi_\perp}{g}^2
 \leqs Z_0(y)^2 \pscal{\pi_0}{(\partial_x\Phi_\perp)^2} \pscal{\pi_0}{g^2}
 \lesssim Z_0(y)^2 \frac{\eps^2}{\sigma^8B(y)^2} \frac{\eps^2}{\sigma^3 Z_0(y)}
 = \frac{\eps^4Z_0(y)}{\sigma^{11}B(y)^2}
\end{equation} 
implied by~\eqref{eq:dxPhiperp_L2}, which accounts for the second summand 
in~\eqref{eq:defective1}. 
\end{proof}


\section{Proof of the main result}
\label{sec:proof_main} 

Theorem~\ref{thm:main_general} follows immediately from 
Proposition~\ref{prop:magic}, Corollary~\ref{cor:hstarAB_pi} and the 
estimate of the capacity given in Theorem~\ref{thm:capacity}. 

We thus proceed with the proof of Theorem~\ref{thm:main_ff}. To this end, we 
start by evaluating more precisely the expression~\eqref{eq:main_general} of 
the expected transition time integrated with respect to the equilibrium 
measure. 

Using~\eqref{eq:delta1ff} and~\eqref{eq:delta1bar}, we obtain 
\begin{equation}
 \delta_1(y) = \frac{1}{\avrg{\lambda_1}} 
\biggbrak{\avrg{\lambda_1A}\biggpar{1 + 
\biggOrder{\frac{\avrg{\lambda_1}}{\eps}}} + R_\delta}\;,
\end{equation} 
where 
\begin{equation}
 R_\delta = 
 \biggOrder{\frac{\eps\ell^3}{\sigma^2} \avrg{\lambda_1}}
 + \biggOrder{\frac{\eps^3\ell^{3/2}}{\sigma^6}\avrg{\sqrt{\lambda_1}}}\;.
\end{equation} 
Using the expressions~\eqref{eq:average_lambda1A} for $\avrg{\lambda_1}$ and 
$\avrg{\lambda_1A}$, we obtain 
\begin{equation}
\label{eq:avrg_1delta} 
 1 - \avrg{\delta_1} 
 = \frac{2\avrg{r_+}}{\avrg{\lambda_1}}
 \biggbrak{1 + 
\biggOrder{\frac{\avrg{\lambda_1}\avrg{\lambda_1A}}{\eps\avrg{r_+}}} + 
\biggOrder{\frac{R_\delta}{\avrg{r_+}}}} 
 \bigbrak{1 + \Order{\sigma^2}}\;.
\end{equation} 
In a similar way, we get
\begin{equation}
\label{eq:avrg_l1Adelta} 
 \avrg{\lambda_1[1-A\delta_1]}
 = \frac{4\avrg{r_-}\avrg{r_+}}{\avrg{\lambda_1}} 
 \biggbrak{1 + 
\biggOrder{\frac{\avrg{\lambda_1}\avrg{\lambda_1A}^2}{\eps\avrg{r_-}\avrg{r_+}}}
 + \bigOrder{\avrg{\lambda_1A}R_\delta}}
 \bigbrak{1 + \Order{\sigma^2}}\;.
\end{equation} 
Substituting in~\eqref{eq:main_general} yields
\begin{equation}
 \int_{\partial\A} \bigexpecin{x}{\tau_\B} \6\nu_{\AB}
 = \frac{\eps}{\avrg{r_-}} 
 \bigbrak{1 + R_1(\eps,\sigma)}\;,
\end{equation}
where
\begin{equation}
\bigabs{R_1(\eps,\sigma)} \lesssim   
 \frac{\avrg{\lambda_1}\avrg{\lambda_1A}}{\eps\avrg{r_+}}
+\frac{\avrg{\lambda_1}\avrg{\lambda_1A}^2}{\eps\avrg{r_-}\avrg{r_+}}
+ \frac{R_\delta}{\avrg{r_+}}
+ \bigabs{R_0}\;,
\end{equation} 
and $R_0$ is defined in~\eqref{eq:def_R0}. Discarding error terms already 
accounted for in the previous estimate, and using~\eqref{eq:avrg_1delta} 
and~\eqref{eq:avrg_l1Adelta}, we obtain 
\begin{equation}
 \bigabs{R_0} 
 \lesssim \sigma^2 
 + \frac{\eps\ell}{\sigma^2}
 \biggbrak{\frac{\avrg{\lambda_1}\avrg{\sqrt{\lambda_1}}}{\avrg{r_+}}
 + \frac{\ell\avrg{\lambda_1}^2}{\avrg{r_-}\avrg{r_+}}}
 + \frac{\eps^2\ell}{\sigma^{7/2}}
 \frac{\avrg{\lambda_1}\avrg{\sqrt{\lambda_1}}}{\avrg{r_-}\avrg{r_+}}\;.
\end{equation} 
In order to simplify the expression of $R_1$, we start by noting than Jensen's 
inequality implies 
\begin{equation}
 \avrg{\sqrt{\lambda_1}}^2 \leqs \avrg{\lambda_1}\;.
\end{equation} 
Then we define 
\begin{equation}
 \sR_- = \frac{\avrg{r_-}}{\avrg{r_+}}\;, \qquad 
 \sR = \frac{\avrg{r_-}}{\avrg{r_+}} + \frac{\avrg{r_+}}{\avrg{r_-}}\;. 
\end{equation} 
A short computation shows that 
\begin{equation}
 \bigabs{R_1(\eps,\sigma)}
 \lesssim \sigma^2 + \biggpar{\frac{\eps\ell^3}{\sigma^2} + 
\frac{\eps^2\ell}{\sigma^{7/2}\avrg{\lambda_1}^{1/2}} + 
\frac{\avrg{\lambda_1}^2}{\eps}} \sR 
+ \frac{\avrg{\lambda_1}}{\eps} (1 + \sR_-)\;.
\end{equation} 
This expression is indeed equivalent to~\eqref{eq:R1_thm}, since we have 
\begin{equation}
 \sR \lesssim \e^{2H/\sigma^2}\;, 
 \qquad 
 \sR_- \lesssim \e^{2H_-/\sigma^2}
\end{equation} 
for the constants $H$ and $H_-$ introduced in~\eqref{eq:def_HH}. 

For $R_1$ to be small, $\eps$ has to satisfy the condition
\begin{equation}
 \avrg{\lambda_1}^2 \sR \vee \avrg{\lambda_1} \sR_- 
 \ll \eps \ll 
 \frac{1}{\sR} \wedge \frac{\avrg{\lambda_1}^{1/4}}{\sR^{1/2}}\;.
\end{equation} 
Since $\avrg{\lambda_1}$ has order $\e^{-2(h_-^{\min}\vee 
h_+^{\min})/\sigma^2}$, where $h_\pm^{\min}$ have been defined 
in~\eqref{eq:def_hmin}, by treating separately the cases $h_+^{\min} > 
h_-^{\min}$ and $h_+^{\min} < h_-^{\min}$, one readily obtains that this 
condition can be satisfied for a non-empty interval of values of $\eps$ if and 
only if Condition~\eqref{eq:condition_hmin} is met. 

It remains to show that we can replace the expectation when starting in the 
equilibrium measure $\nu_{\AB}$ by the expectation when starting in a single 
point on $\partial\A$. This will follow if we can show that 
$\expecin{z}{\tau_\B}$ depends little on the starting point $z\in\partial\A$. 
We will do this by adapting an argument used in the proof 
of~\cite[Proposition~3.6]{BG12a}. 

We first fix a point $z\in\partial\A$, and show that $\expecin{\bar 
z}{\tau_\B}$ is close to $\expecin{z}{\tau_\B}$ for all $\bar z$ in a ball of 
small radius of order $1$ centred in $z$. 
Let $\Omega$ be an event of probability close to $1$, on which 
$\smash{\tau_{\B^+}^z \leqs \tau_\B^{\bar z} \leqs \tau_{\B^-}^z}$, where 
the sets $\B^-\subset \B \subset \B^+$ have boundaries close to each other. 
Using the Cauchy--Schwarz inequality, we obtain 
\begin{equation}
 \expecin{z}{\tau_{\B^+}} - \sqrt{\expecin{z}{\tau_{\B^+}^2}} 
\sqrt{\fP(\Omega^c)}
 \leqs \expecin{\bar z}{\tau_\B} \leqs
 \expecin{z}{\tau_{\B^-}} + \sqrt{\expecin{\bar z}{\tau_{\B}^2}} 
\sqrt{\fP(\Omega^c)}\;.
\end{equation} 
Using a standard large-deviation estimate on $\probin{\bar z}{\tau_\B > T}$ for 
a fixed $T>0$, one easily obtains a bound of the form 
\begin{equation}
\label{eq:proofmain10} 
 \expecin{\bar z}{\tau_{\B}^2} \leqs T_1(\eta) \e^{\eta/\sigma^2} 
\expecin{z}{\tau_\B}^2 
\end{equation} 
with $T_1(\eta) < \infty$ for all $\eta>0$ (see for instance~\cite[Sections 5.2 
to 5.4]{BG12a}, which applies to a much harder infinite-dimensional setting). 
It thus suffices to show that $\fP(\Omega^c) \leqs \e^{-\kappa/\sigma^2}$ for 
some $\kappa>0$ and to apply~\eqref{eq:proofmain10} with $\eta < \kappa/2$ to 
show that $\expecin{\bar z}{\tau_\B}$ and $\expecin{z}{\tau_\B}$ are 
exponentially close to each other. We do this by choosing 
\begin{equation}
 \Omega = \biggset{\frac{\norm{\bar z_t - z_t}}{\norm{\bar z - z}} \leqs 
c\e^{-mt} \; \forall t\geqs 0}\;.
\end{equation} 
Indeed, in~\cite{Martinelli_Olivieri_Scoppola_89} it is shown that this event 
has a probability exponentially close to $1$ for appropriate values of $c, 
m>0$ (see also~\cite{Tsatsoulis_Weber_18} for a more streamlined version of the 
proof of~\cite{Martinelli_Olivieri_Scoppola_89} in a more general setting). 

It remains to show that $\expecin{z}{\tau_\B}$ changes little when $z$ moves 
along the boundary $\partial\A$. To do this, we fix, say, $\bar z = (y,a(y))$ 
with $0 < y < 1$ and $z = (1, a(1))$ on $\partial\A$. Let 
\begin{equation}
 \tau_1(\bar z) = \inf\setsuch{t>1}{y_t = 1}
\end{equation} 
be the first time at which the sample path starting in $\bar z$ hits the line 
$\set{y=1}$. Using for instance~\cite[Proposition~6.3]{BG14}, one easily obtains 
that with probability exponentially close to $1$, $\tau_1(\bar z)$ is bounded by 
a constant of order $1$, and $\norm{\bar z_{\tau_1} - z}$ is smaller than an 
arbitrary constant of order $1$. From that we deduce as above that 
\begin{equation}
 \bigexpecin{\bar z}{\tau_\B}
 = \bigexpecin{z}{\tau_\B} 
 \biggbrak{1 + \biggOrder{\frac{\avrg{r_-}}{\eps}}}\;,
\end{equation} 
which concludes the proof of Theorem~\ref{thm:main_ff}. 


\appendix


\section{The two-state jump process}
\label{app:jump} 

\begin{figure}
\begin{center}
\begin{tikzpicture}[->,>=stealth',shorten >=2pt,shorten 
<=2pt,auto,node
distance=4.0cm, thick,main node/.style={circle,scale=0.7,minimum size=1.1cm,
fill=blue!20,draw,font=\sffamily\Large}]

  \node[main node] (1) {$-$};
  \node[main node] (2) [right of=1] {$+$};

  \path[every node/.style={font=\sffamily\small}]
    (1) edge [above] node {$r_-(y)/\eps$} (2)
    ;
\end{tikzpicture}
\vspace{-5mm}
\end{center}
\caption[]{Absorbing variant of the two-state markovian jump process.
}
\label{fig:jump2}
\end{figure}
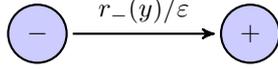

\begin{proof}[{\sc Proof of Proposition~\ref{prop:twostate}}] Consider a 
modified two-state process in which the $+$ state has been made absorbing 
(\figref{fig:jump2}). Its first-hitting time $\tau_+$, starting from the $-$ 
state at a fixed time $y_0$, agrees with the corresponding first-hitting time 
of the original process. The occupation probability $p_-(y)$ of the $-$ state 
satisfies
\begin{equation}
 \eps p_-'(y) = - r_-(y)p_-(y)
\end{equation}
with initial value $p_-(y_0) = 1$, and is thus given by  
\begin{equation}
 p_-(y) 
 = \bigprobin{y_0}{\tau_+ > y} 
 = \e^{-R_-(y,y_0)/\eps}\;.
\end{equation} 
The expectation of $\tau_+$ is thus given by 
\begin{align}
\bigexpecin{y_0}{\tau_+} 
&= \int_0^\infty \bigprobin{y_0}{\tau_+ > y_0 + y} \6y \\
&= \int_{y_0}^\infty \e^{-R_-(y,y_0)/\eps} \6y\;.
\end{align}
Noting that by periodicity, 
\begin{equation}
 R_-(y_0+n+\bar y, y_0) 
 = n R_-(1,0) + R_-(y_0+\bar y,y_0)\;,
\end{equation} 
we obtain 
\begin{equation}
 \bigexpecin{y_0}{\tau_+} 
 = \sum_{n=0}^\infty \e^{-n R_-(1,0)/\eps} \int_0^1 \e^{-R_-(y_0+\bar 
y,y_0)/\eps} \6\bar y\;.
\end{equation} 
Summing the geometric series yields the claimed result. 
\end{proof}


\section{Proofs of the potential-theoretic results}
\label{app:pot} 

In this section, we provide quick proofs of the potential-theoretic results 
stated in Section~\ref{sec:pot}. Except for a small addition in the case of 
test flows which are not divergence-free, all theses proofs are contained 
in~\cite{Landim_Mariani_Seo19}. We include them here for convenience, as we use 
slightly different notations and scalings. 
\NB{The main ingredient for these results are well-known relations between SDEs 
and PDEs. In particular, Dynkin's formula (or Ito's formula for stopping times) 
states that if $z_t$ satisfies a $d$-dimensional SDE with infinitesimal 
generator $\sL$, then 
\begin{equation}
\label{eq:Dynkin} 
 \bigexpecin{z}{\ph(z_\tau)} 
 = \ph(z) + \biggexpecin{z}{\int_0^\tau (\sL\ph)(z_s) \6s}
\end{equation} 
holds for any stopping time $\tau$, and sufficiently regular functions $\ph$, 
see for instance~\cite[Theorem~7.4.1]{Oeksendal}.}

\NB{
\begin{remark}
Another useful link between SDEs and PDEs is the Feynman--Kac formula, which 
states that 
\begin{equation}
 u(t,z) = \biggexpecin{z}{\exp\biggset{-\int_0^t q(z_s)\6s} \ph(z_t)}
\end{equation} 
satisfies the initial value problem 
\begin{equation}
 \begin{cases}
 \dfrac{\partial u}{\partial t}(t,z) 
 = (\sL u)(t,z) - q(z) u(t,z)\;, \\[10pt]
 u(0,z) = \ph(z) 
 \label{eq:PDE_FK} 
\end{cases}
\end{equation}
under appropriate conditions on $\ph$ and $q$, 
see e.g.~\cite[Theorem~8.2.1]{Oeksendal}. Combining this with Dynkin's formula, 
and letting $t$ go to infinity, we obtain that if $\tau$ denotes the first-exit 
time from a sufficiently regular domain $\cD$, then 
\begin{equation}
\label{eq:DFK} 
 u(z) = \biggexpecin{z}{\exp\biggset{-\int_0^\tau q(z_s)\6s} \ph(z_\tau) 
 - \int_0^\tau \exp\biggset{-\int_0^s q(z_v)\6v} \theta(z_s)\6s}
\end{equation} 
satisfies the boundary value problem 
\begin{equation}
 \begin{cases}
 (\sL u)(z) = q(z)u(z) + \theta(z) &\quad z\in \cD\;, \\
 u(z) = \ph(z) &\quad z\in \partial\cD\;,
 \label{eq:PDE_DFK} 
\end{cases}
\end{equation}
provided the functions $q$ and $\ph$ are sufficiently smooth. 
\end{remark}
}


\subsection{Invariant density}
\label{ssec:proof_invariant} 

\begin{proof}[{\sc Proof of Lemma~\ref{lem:pot_invariant}}]
The adjoint in $L^2(\R\times(\R/\Z))$ of $\sL$ is given by 
\begin{equation}
 \sL^\dagger \mu 
 = \frac{\sigma^2}{2\eps} \bigpar{\partial_{xx}\mu + \eps\varrho^2 
\partial_{yy}\mu} - \frac{1}{\eps}\partial_x \bigbrak{b\mu} - 
\partial_y\mu\;.
\end{equation}  
The condition $\sL^\dagger \e^{-2V/\sigma^2} = 0$ is equivalent 
to~\eqref{eq:HJ}. 
\end{proof}

\begin{proof}[{\sc Proof of Lemma~\ref{lem:Lf}}]
Relation~\eqref{eq:L_sym_asym} follows from a short computation. Using the 
explicit form \eqref{eq:defc} of $c$, one obtains 
\begin{align}
 \frac{\sigma^2}{2}\e^{2V/\sigma^2} \nabla \cdot (\e^{-2V/\sigma^2}c) 
 ={}& -\nabla V \cdot c + \frac{\sigma^2}{2} \nabla\cdot c\\
 ={}& -\frac{1}{\eps} (\partial_x V)^2 - \varrho^2(\partial_y V)^2 
 -\frac{1}{\eps}\partial_xV b - \partial_y V \\
 &{}+ \frac{\sigma^2}{2\eps} \bigpar{\partial_{xx}V + 
\eps\varrho^2\partial_{yy}V + \partial_x b}\;, 
\end{align} 
which vanishes by~\eqref{eq:HJ}. 
\end{proof}


\subsection{Capacity}
\label{ssec:proof_capacity} 

\begin{proof}[{\sc Proof of Lemma~\ref{lem:capacity}}]
The fact that $\capacity(\A,\B) = \capacity(\B,\A)$ follows from the relation 
\begin{equation}
 h_{\AB}(x,y) = 1 - h_{\B\A}(x,y)\;.
\end{equation}
To prove the first equality 
in~\eqref{eq:cap01}, we use integration by parts, that is, 
\begin{align}
\label{eq:cap_proof01} 
 \int_{\ABc} \nabla \cdot (h_{\AB}\e^{-2V/\sigma^2}D\nabla 
h_{\AB}) \frac{\6x\6y}{Z}
={}& \int_{\ABc} \nabla h_{\AB}\cdot (D\nabla 
h_{\AB}) \6\pi \\
&{}+ \int_{\ABc} h_{\AB} \nabla\cdot(\e^{-2V/\sigma^2}D\nabla 
h_{\AB}) \frac{\6x\6y}{Z}\;.
\end{align} 
By the divergence theorem and the boundary conditions for $h_{\AB}$, 
\begin{equation}
 \int_{\ABc} \nabla \cdot (h_{\AB}\e^{-2V/\sigma^2}D\nabla 
h_{\AB}) \frac{\6x\6y}{Z} = \int_{\partial \A} (D\nabla h_{\AB} \cdot \nn) 
\pi\6\lambda\;.
\end{equation} 
Since $\sL h_{\AB}$ vanishes on $\ABc$, the second term on the right-hand 
side of~\eqref{eq:cap_proof01} is equal to $2\eps/\sigma^2$ times
\begin{equation}
 \int_{\ABc} h_{\AB} (\sym{\sL} h_{\AB}) \6\pi 
 = - \int_{\ABc} h_{\AB} (\asym{\sL} h_{\AB}) \6\pi 
 = - \int_{\ABc} h_{\AB} (c\cdot\nabla h_{\AB}) \6\pi 
\end{equation} 
The same skew-symmetry argument 
as in~\eqref{eq:skew} implies that the last integral vanishes. Since the first 
term on the right-hand side of~\eqref{eq:cap_proof01} is proportional to the 
capacity, the first equality in~\eqref{eq:cap01} follows. To prove the second 
equality, we use the fact that owing to the vanishing divergence 
condition~\eqref{eq:divergence}, we have 
\begin{equation}
 0 = \int_{\A^c} \e^{2V/\sigma^2} \nabla\cdot(\e^{-2V/\sigma^2}c) \6\pi 
 = \int_{\A^c} \nabla\cdot(\e^{-2V/\sigma^2}c) \frac{\6x\6y}{Z}
 = \int_{\partial \A} (c\cdot\nn) \pi\6\lambda\;.
\end{equation} 
Since $h_{\AB}=1$ on $\partial \A$, the integral of $h_{\AB}(c\cdot\nn) 
\pi\6\lambda$ indeed vanishes. 
To prove the first equality in~\eqref{eq:cap02}, we use a similar computation 
as in~\eqref{eq:cap_proof01} to obtain 
\begin{equation}
 \int_{\ABc} \nabla h^*_{\AB} \cdot (D\nabla h_{\AB}) \6\pi 
 = \int_{\partial \A} (D\nabla h_{\AB})\cdot \nn \, \pi\6\lambda 
 + \eps\int_{\ABc} h^*_{\AB} (c\cdot\nabla h_{\AB}) \6\pi\;.
\end{equation} 
The first term on the right-hand side is proportional to the capacity, yielding 
the claimed result. The second equality in~\eqref{eq:cap02} then follows 
from~\eqref{eq:skew}, while the last equality is obtained by exchanging the 
roles of $h_{\AB}$ and $h^*_{\AB}$ in the above computation. 
\end{proof}


\subsection{Equilibrium measure and mean hitting time}
\label{ssec:proof_hitting} 

\begin{proof}[{\sc Proof of Proposition~\ref{prop:magic}}]
\NB{It follows from Dynkin's formula~\eqref{eq:Dynkin} that the}
function $w_\B(x) = \expecin{x}{\tau_\B}$ satisfies the Poisson problem
\begin{equation}
 \begin{cases}
 (\sL w_\B)(x,y) = -1 &\quad (x,y)\in \B^c\;, \\
 w_\B(x,y) = 0 &\quad (x,y)\in \B\;. 
 \label{eq:Poisson} 
\end{cases}
\end{equation}
By the divergence theorem, we have 
\begin{align}
\frac{\sigma^2}{2\eps} \int_{\partial \A} w_\B (D\nabla 
h^*_{\AB}\cdot\nn)\pi\6\lambda
&= \frac{\sigma^2}{2\eps} \int_{\ABc} 
\nabla\cdot\bigpar{w_\B\e^{-2V/\sigma^2} D\nabla h^*_{\AB}} \frac{\6x\6y}{Z} \\
&= \int_{\ABc} \Bigbrak{\frac{\sigma^2}{2\eps}\nabla w_\B \cdot D\nabla 
h^*_{\AB} + w_\B \sym{\sL}h^*_{\AB}} \6\pi
 \\
&= \int_{\B^c} \Bigbrak{\frac{\sigma^2}{2\eps}\nabla w_\B \cdot D\nabla 
h^*_{\AB} + w_\B (c\cdot \nabla h^*_{\AB})} \6\pi\;,
\label{eq:magic_proof01} 
\end{align}
where we have used the facts that $(\sym{\sL}-\asym{\sL})h^*_{\AB}$ vanishes on 
$\ABc$, while $\nabla h^*_{\AB}=0$ on $\A$. Furthermore, since $h^*_{\AB}$ 
vanishes on $\B$, we have 
\begin{align}
0 &= \frac{\sigma^2}{2\eps}\int_{\B^c} 
\nabla\cdot\bigpar{h^*_{\AB}\e^{-2V/\sigma^2} 
D\nabla w_\B} \frac{\6x\6y}{Z} \\
&= \frac{\sigma^2}{2\eps}\int_{\B^c} \bigbrak{\nabla h^*_{\AB}\cdot D\nabla 
w_\B} 
\6\pi 
+ \int_{\B^c}  h^*_{\AB} \sym{\sL}w_\B \6\pi\;.
\end{align}
Since $w_\B$ solves the Poisson problem~\eqref{eq:Poisson}, we have 
$\sym{\sL}w_\B = -1 - c\cdot\nabla w_\B$, so that substitution 
in~\eqref{eq:magic_proof01} yields 
\begin{equation}
 \frac{\sigma^2}{2\eps} \int_{\partial \A} w_\B (D\nabla 
h^*_{\AB}\cdot\nn)\NB{\pi \6\lambda} 
 = \int_{\B^c} \bigbrak{h^*_{\AB} + h^*_{\AB}(c\cdot\nabla w_\B) + 
w_\B(c\cdot\nabla h^*_{\AB})} \6\pi\;.
\end{equation} 
By the skew-symmetry property~\eqref{eq:skew} and the boundary conditions, the 
contribution of the last two summands in the integral on the right-hand side 
vanishes. 
\end{proof}


\subsection{Variational principles}
\label{ssec:proof_variational} 

\begin{proof}[{\sc Proof of Lemma~\ref{lem:DPhiPsi}}]
We start by noting that 
\begin{equation}
\label{eq:proof_lemmaD01} 
 \sD(\Phi_f, \Psi_{h_{\AB}}) 
 = \int_{\ABc} \Bigpar{\frac{\sigma^2}{2\eps} D\nabla f - fc}\cdot 
\nabla h_{\AB} \6\pi\;.
\end{equation} 
Integrating by parts with respect to $\nabla f$, we obtain 
\begin{align}
\frac{\sigma^2}{2\eps} \int_{\ABc} D\nabla f \cdot \nabla h_{\AB} \6\pi
&= \frac{\sigma^2}{2\eps} \int_{\partial \A} \alpha (D\nabla h_{\AB}\cdot\nn) 
\6\pi - \int_{\ABc} f (\sym{\sL}h_{\AB}) \6\pi \\
&= \alpha\capacity(\A,\B) + \int_{\ABc} f (c\cdot\nabla h_{\AB}) \6\pi\;.
\end{align}
The second term on the right-hand side cancels the $c$-dependent term 
in~\eqref{eq:proof_lemmaD01}. Furthermore, we have 
\begin{equation}
  \sD(\ph, \Psi_{h_{\AB}}) 
 = \int_{\ABc} \ph\cdot\nabla h_{\AB} \6x\6y 
 = \int_{\partial \A} (\ph\cdot\nn)\6\lambda 
 - \int_{\ABc} (\nabla\cdot\ph) h_{\AB} \6x\6y\;.
\end{equation} 
The first term on the right-hand side is equal to $-\gamma$ 
by~\eqref{eq:flow_gamma}, while the second one vanishes since $\ph$ is 
divergence-free. 
\end{proof}

\begin{remark}
If $\ph$ is only approximately divergence-free, the above proof yields 
\begin{equation}
 \label{eq:lemma_D_approx} 
 \sD(\Phi_f - \ph, \Psi_{h_{\AB}}) 
 = \alpha\capacity(\A,\B) + \gamma
 + \int_{\ABc} (\nabla\cdot\ph) h_{\AB} \6x\6y\;.
\end{equation}
This can be used to obtain bounds from flows that are not exactly 
divergence-free.
\end{remark}

\begin{proof}[{\sc Proof of Proposition~\ref{prop:Dirichlet}}]
Pick $f\in\sH^{1,0}_{\AB}$ and $\ph\in\sF^0_{\AB}$. By~\eqref{eq:lemma_D} with 
$\alpha=1$ and $\gamma=0$ and the Cauchy--Schwarz inequality, we have 
\begin{equation}
 \capacity(\A,\B)^2 = \sD(\Phi_f - \ph,\Psi_{h_{\AB}})^2
 \leqs \sD(\Phi_f - \ph)\sD(\Psi_{h_{\AB}}) 
 = \sD(\Phi_f - \ph)\capacity(\A,\B)\;,
\end{equation} 
showing that $\capacity(\A,\B) \leqs \sD(\Phi_f - \ph)$. Furthermore, 
$\sD(\Phi_{\bar f} - \bar\ph) = \sD(\Psi_{h_{\AB}}) = \capacity(\A,\B)$. Since 
clearly $\bar f \in \smash{\sH^{1,0}_{\AB}}$, it remains to show that $\bar\ph 
\in \smash{\sF^0_{\AB}}$. Noting that 
\begin{equation}
 \bar\ph = \frac{\sigma^2}{4\eps} \pi D \bigbrak{\nabla h^*_{\AB} - \nabla 
h_{\AB}} - \frac12 \pi c \bigbrak{h_{\AB} + h^*_{\AB}}\;, 
\end{equation} 
we obtain 
\begin{align}
\nabla\cdot\bar\ph
&= \frac{\sigma^2}{4\eps Z} \nabla \cdot \bigbrak{\e^{-2V/\sigma^2} D (\nabla 
h^*_{\AB} - \nabla h_{\AB})}
 - \frac{1}{2Z} \nabla \cdot \bigbrak{\e^{-2V/\sigma^2} c (h_{\AB} + 
h^*_{\AB})} 
\\
&= \frac12 \pi \bigbrak{\sym{\sL} h^*_{\AB} - \sym{\sL} h_{\AB}}
- \frac12 \pi c\cdot(\nabla h_{\AB} + \nabla h^*_{\AB}) \\
&= \frac12 \pi \bigbrak{\sL^* h^*_{\AB} - \sL h_{\AB}} = 0\;,
\end{align}
where we have used the fact that $\nabla \cdot (\e^{-2V/\sigma^2}c) = 0$ in the 
second line. Furthermore, 
\begin{align}
\int_{\partial \A} \bar\ph \cdot \nn \6\lambda
&= \frac12 \int_{\partial \A} \Bigbrak{\frac{\sigma^2}{2\eps} D\nabla h^*_{\AB} 
- 
c h^*_{\AB}} \cdot \nn \6\lambda
- \frac12 \int_{\partial \A} \Bigbrak{\frac{\sigma^2}{2\eps} D\nabla h_{\AB} + 
c h_{\AB}} \cdot \nn \6\lambda \\
&= \frac12\capacity^*(\A,\B) - \frac12\capacity(\A,\B)\;,
\end{align}
which vanishes by Lemma~\ref{lem:capacity}. The 
bound~\eqref{eq:Dirichlet_defective} is obtained by 
using~\eqref{eq:lemma_D_approx} instead of ~\eqref{eq:lemma_D}. 
\end{proof}

\begin{proof}[{\sc Proof of Proposition~\ref{prop:Thomson}}]
Pick $f\in\sH^{0,0}_{\AB}$ and $\ph\in\sF^1_{\AB}$. By~\eqref{eq:lemma_D} with 
$\alpha=0$ and $\gamma=1$ and the Cauchy--Schwarz inequality, we have 
\begin{equation}
 1 = \sD(\Phi_f - \ph,\Psi_{h_{\AB}})^2
 \leqs \sD(\Phi_f - \ph)\sD(\Psi_{h_{\AB}}) 
 = \sD(\Phi_f - \ph)\capacity(\A,\B)\;,
\end{equation} 
showing that $\capacity(\A,\B) \geqs 1/\sD(\Phi_f - \ph)$. By bilinearity of 
$\sD$, we have $\sD(\Phi_{\bar f} - \bar\ph) = 1/\capacity(\A,\B)$. Since $\bar 
f\in \smash{\sH^{0,0}_{\AB}}$, it remains to show that 
$\bar\ph\in\smash{\sF^1_{\AB}}$. This time, we have 
\begin{equation}
\bar\ph = -\frac{\pi}{2\capacity(\A,\B)} \Bigpar{\frac{\sigma^2}{2\eps} D 
\bigbrak{\nabla h_{\AB} + \nabla h^*_{\AB}} - c \bigbrak{h_{\AB} - 
h^*_{\AB}}}\;.
\end{equation} 
Using the fact that $\nabla \cdot (\e^{-2V/\sigma^2}c) = 0$, we obtain 
\begin{equation}
 \nabla \cdot \bar\ph
 = -\frac{\pi}{2\capacity(\A,\B)} \Bigpar{\sym{\sL} \bigbrak{h_{\AB} + 
h^*_{\AB}} + 
c\nabla\cdot\bigbrak{h_{\AB} - h^*_{\AB}}} = 0\;,
\end{equation} 
and 
\begin{align}
\capacity(\A,\B) \int_{\partial \A} \bar\ph \cdot \nn \6\lambda
={}& - \frac12 \int_{\partial \A} \Bigbrak{\frac{\sigma^2}{2\eps} D\nabla 
h_{\AB} 
+ c h_{\AB}} \cdot \nn \6\lambda \\
&{}- 
\frac12 \int_{\partial \A} \Bigbrak{\frac{\sigma^2}{2\eps} D\nabla h^*_{\AB} - 
c h^*_{\AB}} \cdot \nn \6\lambda
 \\
={}& -\frac12\capacity(\A,\B) - \frac12\capacity^*(\A,\B) 
= -\capacity(\A,\B)\;,
\end{align}
showing that $\bar\ph$ has flux $-1$ as required. The 
bound~\eqref{eq:Thomson_defective} is again obtained by 
using~\eqref{eq:lemma_D_approx} instead of ~\eqref{eq:lemma_D}. 
\end{proof}



\section{Estimates on static eigenfunctions}
\label{app:proofs_static} 


\subsection{Bounds on $h_0$, $h_1$ and $\phi_1$}
\label{app:proof_phi1} 

In this section, to lighten notations, we will often drop the dependence of the 
functions on $y$. 

\begin{proof}[{\sc Proof of Proposition~\ref{prop:h0}}]
First note that $\partial_y h_0(x)$ vanishes for $x\not\in(x^*_-,x^*_+)$. We 
thus assume henceforth that $x\in(x^*_-,x^*_+)$. Taking the derivative with 
respect to $y$ of the expression~\eqref{eq:h02} for the committor, we obtain 
\begin{equation}
\label{eq:proof_h0} 
 \partial_y h_0(x) 
 = \frac{2}{\sigma^2} \frac{1}{N} 
 \bigbrak{(1-h_0(x))I(x) - h_0(x) (J-I(x))}\;,
\end{equation} 
where 
\begin{align}
 I(x) &= \int_x^{x^*_+} \partial_y V_0(\bar x) \e^{2V_0(\bar x)/\sigma^2} 
\6\bar x + \frac{\sigma^2}{2} \dtot{x^*_+}{y} \e^{-2h_+/\sigma^2} \\
 J &= \frac{\sigma^2}{2} N'(y) 
 = I(x^*_-) - \frac{\sigma^2}{2} \dtot{x^*_-}{y} \e^{-2h_-/\sigma^2}\;.
\end{align}
By standard Laplace asymptotics (see Appendix~\ref{app:Laplace}), we obtain 
\begin{equation}
\label{eq:bound_h0} 
 h_0(x) \asymp
 \begin{cases}
  1 & \text{if $x\leqs x^*_0\;,$} \\
  \dfrac{\sigma\e^{2V_0(x)/\sigma^2}}{\sigma+\abs{x-x^*_0}}
  & \text{if $x > x^*_0\;,$}
 \end{cases}
 \qquad 
 1 - h_0(x) \asymp
 \begin{cases}
  \dfrac{\sigma\e^{2V_0(x)/\sigma^2}}{\sigma+\abs{x-x^*_0}} 
  & \text{if $x\leqs x^*_0\;,$} \\ 
  1 & \text{if $x > x^*_0\;.$}
 \end{cases} 
\end{equation} 
Similarly, using the fact that $x\mapsto V_0(x)$ is increasing on 
$(x^*_-,x^*_0)$ and decreasing on $(x^*_0,x^*_-)$, we get 
\begin{equation}
 \abs{I(x)} \lesssim 
 \begin{cases}
  \sigma^3 & \text{if $x\leqs x^*_0\;,$} \\
  \sigma^2\e^{2V_0(x)/\sigma^2}
  & \text{if $x > x^*_0\;,$}
 \end{cases}
 \qquad 
 \abs{J-I(x)} \lesssim 
 \begin{cases}
  \sigma^2\e^{2V_0(x)/\sigma^2} 
  & \text{if $x\leqs x^*_0\;,$} \\ 
  \sigma^3 & \text{if $x > x^*_0\;.$}
 \end{cases} 
\end{equation} 
Substituting in~\eqref{eq:proof_h0} yields 
\begin{equation}
 \bigabs{\partial_y h_0(x)} \lesssim \frac{1}{\sigma} \e^{2V_0(x)/\sigma^2}\;,
\end{equation} 
which implies~\eqref{eq:bound_dyh0}. The bound~\eqref{eq:bound_dyyh0} follows 
in an analogous way from the fact that 
\begin{equation}
 \partial_{yy} h_0(x)
 = \frac{2}{\sigma^2} \frac{1}{N} 
 \bigbrak{-2J\partial_y h_0(x) + (1-h_0(x))\partial_y I(x) - h_0(x) 
\partial_y(J-I(x))}\;.
\end{equation}
The bound~\eqref{eq:bound_eta} on $\eta(y)$ is a consequence of the fact 
that~\eqref{eq:bound_h0} yields 
\begin{equation}
 \eta(y) \lesssim \frac{1}{Z_0} \int_{x^*_-}^{x^*_+} 
 \e^{-2V_0(x)/\sigma^2}
 \frac{\sigma\e^{2V_0(x)/\sigma^2}}{\sigma+\abs{x-x^*_0}}\6x 
 \lesssim \frac{\sigma\log(\sigma^{-1})}{Z_0}
\end{equation} 
combined with~\eqref{eq:N_Laplace} and~\eqref{eq:Z0Nlambda1}.
\end{proof}

\begin{remark}
Actually, Lemma~\ref{lem:Laplace3} provides the sharper bound
\begin{equation}
\label{eq:IJ_sharper} 
 \abs{I(x)} \lesssim 
 \begin{cases}
  \sigma^3 & \text{for $x\leqs x^*_0\;,$} \\
  \sigma^2(\sigma+\abs{x-x^*_0})\e^{2V_0(x)/\sigma^2}
  & \text{for $x > x^*_0\;,$}
 \end{cases}
\end{equation}
and similarly for $J-I(x)$.
\end{remark}

In order to prove Proposition~\ref{prop:h1}, we will use the fact that $h_1$ 
and its derivatives satisfy certain Poisson boundary value problems. For this 
purpose, we will repeatedly use the following lemma.

\begin{lemma}
\label{lem:Poisson} 
Let $\cD=(x^*_-,x^*_+)$, and let $\ph$ satisfy the Poisson problem 
\begin{equation}
\label{eq:Poisson_problem} 
 \begin{cases}
  \bigpar{(\sL_x + \lambda_1) \ph}(x) = \psi(x) \quad & x\in \cD\;, \\
  \ph(x) = 0 & x\in\partial \cD\;.
 \end{cases}
\end{equation} 
Assume that there exists a constant $c$ such that $\abs{\psi(x)} \leqs c 
h_0(x)$ for all $x\in \cD$. Then 
\begin{align}
 \abs{\ph(x)} &\lesssim c\ell(\sigma) h_0(x) \\
 \abs{\partial_x\ph(x)} &\lesssim \frac{1}{\sigma^2} c\ell(\sigma) h_0(x)
\end{align}
holds for all $x\in \cD$. Analogous bounds hold with $h_0(x)$ replaced by 
$1-h_0(x)$ throughout. 
\end{lemma}
\begin{proof}
Consider first the simpler Poisson problem $\sL_x\ph = \psi$, with zero 
boundary conditions as in~\eqref{eq:Poisson_problem}. Using the second 
expression for $\sL_x$ in~\eqref{eq:Lx}, it is easy to check that its solution 
is given by 
\begin{align}
 (\sL_x^{-1}\psi)(x)
 = \frac{2}{\sigma^2} \int_{x^*_-}^x \e^{2V_0(x_1)/\sigma^2} 
 \biggl[ & \int_{x^*_-}^{x_1} 
\e^{-2V_0(x_2)/\sigma^2}(1-h_0(x_2))\psi(x_2)\6x_2 \\
 &{}- \int_{x_1}^{x^*_+} \e^{-2V_0(x_2)/\sigma^2}h_0(x_2)\psi(x_2)\6x_2 
\biggr]\6x_1\;.
\label{eq:Lx_inverse} 
\end{align}
Using the assumption on $\psi$ and the bounds~\eqref{eq:bound_h0} on $h_0$, one 
obtains that for $x \leqs x^*_0$, 
\begin{equation}
 \bigabs{(\sL_x^{-1}\psi)(x)} \lesssim c \ell(\sigma) 
 \lesssim c \ell(\sigma) h_0(x)\;.
\end{equation} 
A similar conclusion is obtained for $x \geqs x^*_0$ using the 
equivalent expression 
\begin{align}
 (\sL_x^{-1}\psi)(x)
 = \frac{2}{\sigma^2} \int_x^{x^*_+} \e^{2V_0(x_1)/\sigma^2} 
 \biggl[ & \int_{x_1}^{x^*_+} 
\e^{-2V_0(x_2)/\sigma^2}h_0(x_2)\psi(x_2)\6x_2 \\
 &{}- \int_{x^*_-}^{x_1} \e^{-2V_0(x_2)/\sigma^2}(1-h_0(x_2))\psi(x_2)\6x_2 
\biggr]\6x_1\;.
\end{align}
Corresponding bounds on $\partial_x(\sL_x^{-1}\psi)$ are obtained in a similar 
way, using the derivative with respect to $x$ of~\eqref{eq:Lx_inverse}. It thus 
remains to extend the bounds to $(\sL_x + \lambda_1)^{-1}\psi$. This follows 
readily from the Neumann-type series 
\begin{equation}
 (\sL_x + \lambda_1)^{-1}\psi 
 = \sum_{k\geqs0} (-\lambda_1)^k \bigpar{(\sL_x)^{-1}}^k \psi\;,
\end{equation} 
bounding each term by repeatedly applying the bounds on $\sL_x^{-1}$ and 
summing the resulting geometric series. 
\end{proof}

\begin{proof}[{\sc Proof of Proposition~\ref{prop:h1}}]
Taking the difference of the equations $(\sL_x+\lambda_1)(h_0+h_1)=0$ and $\sL_x 
h_0 = 0$, we find that $h_1$ satisfies the Poisson 
problem~\eqref{eq:Poisson_problem} with 
\begin{equation}
 \psi(x) = -\lambda_1 h_0(x)\;.
\end{equation} 
Lemma~\ref{lem:Poisson} thus immediately yields 
\begin{equation}
\label{eq:proof_dxh1}
 \bigabs{h_1(x)} \lesssim \lambda_1 \ell(\sigma) h_0(x)\;, \qquad 
 \bigabs{\partial_xh_1(x)} 
 \lesssim \frac{1}{\sigma^2}\lambda_1 \ell(\sigma)h_0(x)\;. 
\end{equation} 
Taking the derivative with respect to $y$ of the equation for $h_1$, we obtain 
that $\partial_y h_1$ satisfies~\eqref{eq:Poisson_problem} with
\begin{equation}
 \psi(x) = -\lambda_1'(h_0(x)+h_1(x)) - \lambda_1 \partial_yh_0(x) + 
\partial_{xy}V_0(x)\partial_x h_1(x)\;.
\end{equation} 
The bounds on $h_0$, $\partial_y h_0$ and~\eqref{eq:proof_dxh1} imply that 
$\psi(x)$ has order $\ell\lambda_1h_0(x)/\sigma^2$, so that 
Lemma~\ref{lem:Poisson} yields 
\begin{equation}
\label{eq:proof_dxyh1}
 \bigabs{\partial_yh_1(x)} \lesssim \frac{1}{\sigma^2}
 \lambda_1 \ell(\sigma)h_0(x)\;, \qquad 
 \bigabs{\partial_{xy}h_1(x)} 
 \lesssim \frac{1}{\sigma^4}\lambda_1 \ell(\sigma)h_0(x)\;. 
\end{equation} 
The bound on $\partial_{yy}h_1$ is obtained in an analogous way, by taking one 
more derivative with respect to $y$. 
\end{proof}

\begin{proof}[{\sc Proof of Proposition~\ref{prop:phi_1}}]
We introduce the variables 
\begin{equation}
\label{eq:defuv} 
 u(y) = -\frac{\phi_+(y)}{\phi_-(y)}\;, \qquad 
 v(y) = -\phi_+(y)\phi_-(y)\;.
\end{equation} 
The orthogonality condition
\begin{equation}
 0 = \pscal{\pi_0}{\phi_1} 
 = \phi_-(y)\pscal{\pi_0}{h_0+h_1} + \phi_+(y)\pscal{\pi_0}{1-h_0+\bar h_1}
\end{equation} 
yields 
\begin{equation}
\label{eq:u_value} 
 u(y) 
 =\frac{\pscal{\pi_0}{h_0+h_1}}{\pscal{\pi_0}{1-h_0+\bar h_1}}
 = \e^{-2\Deltabar(y)/\sigma^2} \bigbrak{1 + \Order{\lambda_1\ell}}\;,
\end{equation} 
where we have used~\eqref{eq:defDeltabar} and Proposition~\ref{prop:h1} to 
obtain the last equality. The function $v(y)$ is then determined via the 
normalisation condition
\begin{align}
1 &= \pscal{\pi_1}{\phi_1} = \pscal{\pi_0}{\phi_1^2} \\
&= \frac{v(y)}{u(y)} X(y) + u(y)v(y) Y(y) - 2v(y) Z(y)\;,
\label{eq:proof_uv1} 
\end{align}
where 
\begin{align}
X(y) &:= \pscal{\pi_0}{[h_0+h_1]^2}
= \pscal{\pi_0}{h_0}\bigbrak{1+\Order{\lambda_1\ell}} + \Order{\eta}\;, \\
Y(y) &:= \pscal{\pi_0}{[1-h_0+\bar h_1]^2}
= \pscal{\pi_0}{1-h_0}\bigbrak{1+\Order{\lambda_1\ell}} + \Order{\eta}\;, \\
Z(y) &:= \pscal{\pi_0}{[h_0+h_1][1-h_0+\bar h_1]}
= \Order{\eta}\;,
\end{align}
owing to Propositions~\ref{prop:h0} and~\ref{prop:h1}. Substituting 
in~\eqref{eq:proof_uv1}, using the expressions~\eqref{eq:defDeltabar} of 
$\pscal{\pi_0}{h_0}$ and $\pscal{\pi_0}{1-h_0}$ and solving for $v(y)$ yields 
\begin{equation}
 v(y) = 1+\Order{\lambda_1(y)\ell}
\end{equation} 
thanks in particular to the bound~\eqref{eq:bound_eta} on $\eta(y)$. Expressing 
$\phi_\pm(y)$ in terms of $u(y)$ and $v(y)$ yields~\eqref{eq:phipm}. 

The other relations then follow essentially by taking derivatives with respect 
to $y$ of the above expressions. Differentiating~\eqref{eq:defuv},  
we obtain 
\begin{equation}
\label{eq:proof_phiprime} 
 \phi_+'(y) = -\frac12 \biggpar{\frac{v'(y)}{\phi_-(y)} + \phi_-(y)u'(y)}\;, 
\qquad 
 \phi_-'(y) = -\frac{v'(y)-\phi_-(y)^2u'(y)}{2\phi_+(y)}\;.
\end{equation}
Differentiating~\eqref{eq:u_value} yields 
\begin{equation}
 u'(y) = \partial_y 
 \frac{\pscal{\pi_0}{h_0} + \pscal{\pi_0}{h_1}}
 {\pscal{\pi_0}{1-h_0} + \pscal{\pi_0}{\bar h_1}}
 = -2\frac{\Deltabar'(y)}{\sigma^2} \e^{-2\Deltabar(y)/\sigma^2}
 \bigbrak{1+\Order{\lambda_1(y)\ell}}\;, 
\end{equation} 
while the derivative of~\eqref{eq:proof_uv1} gives 
\begin{equation}
 v'(y) 
 = \frac{u^{-2}u'X-u'Y-u^{-1}X'-uY'+2Z'}
 {u^{-1}X+uY'-2Z}v(y)
 = \biggOrder{\frac{\lambda_1(y)\ell}{\sigma^2}}\;.
\end{equation} 
Substituting in~\eqref{eq:proof_phiprime} yields~\eqref{eq:dyphipm}. 
In the same spirit, one obtains 
\begin{equation}
 u''(y) = \frac{4}{\sigma^4}
 \biggbrak{\Deltabar'(y)^2 - \frac12\sigma^2\Delta''(y) + 
\Order{\lambda_1(y)\ell^3}} u(y)\;, \qquad
 v''(y) = \biggOrder{\frac{\lambda_1(y)\ell^3}{\sigma^2}}\;,
\end{equation} 
and plugging this into the derivative of~\eqref{eq:proof_phiprime} 
yields~\eqref{eq:dyyphipm}. 
\end{proof}

\begin{proof}[{\sc Proof of Proposition~\ref{prop:f10}}]
The expression for $f_{10} = -\sigma^2 \pscal{\pi_0}{\partial_y\phi_1}$ follows 
from the expression~\eqref{eq:dyphi1_rep} for $\partial_y\phi_1$, the 
definition~\eqref{eq:defDeltabar} of $\Deltabar(y)$, and the fact that 
$\sigma^2\pscal{\pi_0}{\abs{\partial_y h_0}}$ has order $\eta(y)$ by 
Proposition~\ref{prop:h0}. A similar argument applies to $f_{11} = -\sigma^2 
\pscal{\pi_0}{\phi_1\partial_y\phi_1}$.   

The expression for $g_{10}$ is obtained by evaluating separately the two 
summands on the right-hand side of~\eqref{eq:orth_fnm}. Proceeding as for 
$f_{10}$, we obtain
\begin{equation}
 \sigma^4 \pscal{\pi_0}{\partial_{yy}\phi_1}
 = B \bigbrak{\sigma^2 \Deltabar'' + \bigOrder{\lambda_1 \ell^3}}\;.
\end{equation} 
In order to determine $\pscal{\partial_y\pi_0}{\partial_y\phi_1}$, we note that 
on one hand, 
\begin{equation}
 \partial_y \pscal{\pi_0}{h_0}
 = \pscal{\partial_y\pi_0}{h_0} + \pscal{\pi_0}{\partial_y h_0}
 = \pscal{\partial_y\pi_0}{h_0} + \biggOrder{\frac{\lambda_1B^2}{\sigma^2}}\;,
\end{equation} 
while on the other hand, \eqref{eq:defDeltabar} implies 
\begin{equation}
 \partial_y \pscal{\pi_0}{h_0}
 = - \frac{2\Deltabar'}{\sigma^2(\e^{-\Deltabar/\sigma^2} + 
\e^{\Deltabar/\sigma^2})^2}\;.
\end{equation} 
This yields 
\begin{equation}
 \sigma^4 \pscal{\partial_y\pi_0}{\partial_y\phi_1}
 = B \bigbrak{- 2 A(\Deltabar')^2 + \bigOrder{\lambda_1 \ell^2}}\;,
\end{equation} 
and implies the stated expression for $g_{10}$. The computation of $g_{11}$ is 
similar.
\end{proof}

For further reference, we list here a few more expressions of particular inner 
products, which can be derived in the same way as in the above proof:
\begin{subequations}
\begin{align}
\label{eq:pi0dyphi12} 
  \sigma^4 \pscal{\pi_0}{(\partial_y\phi_1)^2}
 &= (\Deltabar')^2 + \Order{\lambda_1 \ell^2}\;, \\
\label{eq:dypi1dyphi1} 
 \sigma^4 \pscal{\partial_y\pi_1}{\partial_y\phi_1}
 &= - A^2 (\Deltabar')^2 + \bigOrder{\lambda_1 \ell^2}\;, \\
\label{eq:pi0dyyphi12} 
 \sigma^8 \pscal{\pi_0}{(\partial_{yy}\phi_1)^2}
 &= (\Deltabar')^4 + 2 \sigma^2 A (\Deltabar')^2 \Deltabar'' 
 + \sigma^4 (\Deltabar'')^2 + \bigOrder{\lambda_1 \ell^3}\;.
\end{align}
\end{subequations}


\subsection{Bounds on other eigenfunctions}
\label{app:proof_phin} 

To obtain estimates involving other eigenfunctions than $\phi_1$, it will 
sometimes be useful to take advantage of the fact that $\sL_x$ is 
conjugated to the Schr\"odinger operator
\begin{equation}
\label{eq:Schrodinger} 
 \tilde \sL_x = \e^{-V_0/\sigma^2}\sL_x\e^{V_0/\sigma^2}
 = \frac{\sigma^2}{2} \partial_{xx} - \frac{1}{2\sigma^2}U_0\;,
\end{equation} 
where $U_0$ is the three-well potential 
\begin{equation}
 U_0(x,y) = \bigl(\partial_x V_0(x,y)\bigr)^2 - \sigma^2\partial_{xx} 
V_0(x,y)\;.
\end{equation} 
In particular, $\tilde\sL_x$ has the same eigenvalues $-\lambda_n$ as $\sL_x$, 
and its eigenfunctions $\psi_n$ satisfy 
\begin{equation}
\label{eq:psin} 
 \psi_n(x) = \frac{1}{\sqrt{Z_0}} \e^{-V_0/\sigma^2} \phi_n(x) 
 = \sqrt{Z_0} \e^{V_0/\sigma^2} \pi_n(x)\;.
\end{equation} 
In particular, we have the relations 
\begin{equation}
\label{eq:dypsin} 
 \partial_y \psi_n 
 = \frac{1}{\sqrt{\pi_0}} \biggbrak{\partial_y\pi_n - \frac{1}{\sigma^2}W\pi_n}
 = \sqrt{\pi_0} \biggbrak{\partial_y\phi_n + \frac{1}{\sigma^2}W\phi_n}
\end{equation} 
between derivatives of eigenfunctions, where we have used 
\begin{equation}
\label{eq:defW} 
 \partial_y \pi_0 = \frac{2}{\sigma^2}  W \pi_0\;,
\qquad 
W = \pscal{\pi_0}{\partial_y V_0} - \partial_y V_0\;.
\end{equation} 
Note that by the Feynman--Hellmann 
theorem, we have 
\begin{equation}
\label{eq:Hellmann-Feynman} 
 \lambda_n'(y) 
 = - \pscal{\psi_n}{\partial_y\tilde\sL_x\psi_n} 
 = \frac{1}{2\sigma^2} \pscal{\psi_n}{\partial_y U_0\psi_n}\;,
\end{equation} 
while first-order perturbation theory shows that if $\lambda_n \neq \lambda_m$, 
then
\begin{equation}
\label{eq:ev_perturbation} 
 2\sigma^2 \pscal{\psi_n}{\partial_y\psi_m}
 = \frac{1}{\lambda_m - \lambda_n} \pscal{\psi_n}{\partial_y U_0\psi_m}\;.
\end{equation} 
This entails in particular the following useful estimate.

\begin{lemma}
\label{lem:sum_dypin} 
For any function $f\in L^2(\pi_0[1+W^2+\sigma^2\partial_yW]^2\6x)$, 
\begin{align}
 \sigma^4 \sum_{n\geqs 1} \pscal{\partial_y\pi_n}{f}^2
 &\lesssim \pscal{\pi_0}{[1+W^2]f^2}\;, \\
 \sigma^8 \sum_{n\geqs 1} \pscal{\partial_{yy}\pi_n}{f}^2
 &\lesssim \pscal{\pi_0}{[1+W^2+\sigma^2\partial_yW]^2f^2}\;.
\end{align} 
\end{lemma}
\begin{proof}
By~\eqref{eq:dypsin}, we have 
\begin{equation}
 \sigma^4\pscal{\partial_y\pi_n}{f}^2
 \leqs 2 \pscal{\sqrt{\pi_0}}{Wf\psi_n}^2 
 + 2\sigma^4\pscal{\sqrt{\pi_0}}{\partial_y\psi_nf}^2\;. 
\end{equation}
Summing over $n$ yields two terms, the first one being equal to 
\begin{equation}
 2\pscal{\sqrt{\pi_0}}{W^2f^2\sqrt{\pi_0}}
 = 2\pscal{\pi_0}{W^2f^2}\;.
\end{equation} 
As for the second sum, the Cauchy--Schwarz inequality 
and~\eqref{eq:ev_perturbation} show that it is bounded by 
\begin{equation}
 2\sigma^4\sum_{n\geqs1} \pscal{\partial_y\psi_n}{\partial_y\psi_n} 
\pscal{\pi_0}{f^2}
 \leqs \frac12\pscal{\pi_0}{f^2}
 \sum_{n\neq m} \frac{\pscal{\psi_n}{\partial_y U_0\psi_m}^2}{(\lambda_m - 
\lambda_n)^2}\;.
\end{equation} 
The last sum is bounded, because $\lambda_n$ grows like $n^2$, while 
$\pscal{\psi_n}{\partial_y U_0\psi_m}$ is bounded uniformly in $n$ and $m$. 
This proves the first inequality, and the second one is proved in a similar 
way, taking one more derivative with respect to $y$. 
\end{proof}

\begin{remark}
In may happen that two eigenvalue $\lambda_n(y)$ and $\lambda_m(y)$ cross for 
particular value of $y$. In that case, the bound~\eqref{eq:ev_perturbation} 
becomes useless, but an equivalent result can be obtained by locally modifying 
the basis of eigenfunctions. For simplicity, we will not give details of this 
procedure here, but refer the reader to~\cite{Joye95,JMGJ99}.
\end{remark}

In the same spirit, the following lemma allows to estimate derivatives of 
functions expanded in the eigenbasis.

\begin{lemma}
\label{lem_dxPhi} 
Recall that $\cD=(x^*_-,x^*_+)$ and let 
\begin{equation}
 \Phi(x) = \sum_{n\geqs2} \alpha_n\phi_n(x)\;.
\end{equation} 
Then 
\begin{equation}
 \pscal{\pi_0}{(\partial_x\Phi)^2\indicator{\cD}}
 \lesssim \frac{1}{\sigma^4} \sum_{n\geqs2} \alpha_n^2
 + \frac{1}{\sigma^2} \sum_{n\geqs2} \lambda_n\alpha_n^2\;.
\end{equation} 
\end{lemma}
\begin{proof}
By~\eqref{eq:psin}, we have 
\begin{equation}
 \partial_x \Phi(x) 
 = \frac{1}{\sqrt{\pi_0(x)}} \sum_{n\geqs2} \alpha_n 
\biggbrak{\frac{1}{\sigma^2}\partial_x V_0(x)\psi_n(x) + \partial_x\psi_n(x)}
 =: \frac{1}{\sqrt{\pi_0(x)}} \bigbrak{\Psi_1(x) + \Psi_2(x)}\;. 
\end{equation} 
This implies 
\begin{equation}
 \pscal{\pi_0}{(\partial_x\Phi)^2\indicator{\cD}}
 \leqs 2\pscal{\Psi_1\indicator{\cD}}{\Psi_1\indicator{\cD}} + 
2\pscal{\Psi_2\indicator{\cD}}{\Psi_2\indicator{\cD}}\;.
\end{equation} 
The first term on the right-hand side satisfies the claimed bound since 
$\partial_x V_0$ is bounded on $\cD$. As for the second term, the fact that 
$\psi_m$ is an eigenfunction of $\tilde\sL_x$ implies 
\begin{equation}
 \pscal{\partial_x\psi_n}{\partial_x\psi_m}
 = - \pscal{\psi_n}{\partial_{xx}\psi_m}
 = \frac{2}{\sigma^2} \lambda_n \delta_{nm} - 
\frac{1}{\sigma^4}\pscal{\psi_n}{U_0\psi_m}\;.
\end{equation} 
This yields 
\begin{equation}
 \pscal{\Psi_2\indicator{\cD}}{\Psi_2\indicator{\cD}}
 = \frac{2}{\sigma^2} \sum_{n\geqs2} \lambda_n\alpha_n^2 
 - \frac{1}{\sigma^4} \pscal{\Psi}{\indicator{\cD}U_0\Psi}\;, \qquad 
 \Psi(x) = \sum_{n\geqs2}\alpha_n\psi_n(x)\;,
\end{equation} 
which also satisfies the claimed bound. 
\end{proof}

\begin{proof}[{\sc Proof of Proposition~\ref{prop:sum_fnm}}]
To prove the first two bounds, we note that owing to the completeness of the 
set of eigenfunctions, one has 
\begin{equation}
\sum_{n\geqs0} f_{n0}^2 
= \sigma^4 \sum_{n\geqs0} \pscal{\partial_y\pi_0}{\phi_n} \pscal{\pi_n}{W} 
= 4\pscal{\pi_0}{W^2}\;,
\end{equation} 
which has order $1$. In a similar way, we obtain
\begin{equation}
\sum_{n\geqs0} f_{n1}^2 
= 4\pscal{\pi_0}{\phi_1 W^2} + 2\sigma^2 \pscal{\pi_0}{\partial_y\phi_1 W}\;.
\end{equation} 
Using the Cauchy--Schwarz inequality, \eqref{eq:proj_dyphi1} and 
Remark~\ref{rem:pi0absphi1}, one obtains that both terms have again order $1$. 
The proof of the bounds involving $g_{ni}$ are similar. 

The last two bounds then follow directly from Lemma~\ref{lem:sum_dypin} with $n$ 
and $m$ interchanged, taking $f = \phi_n$, since $W$ is bounded uniformly on 
compact sets, while for large $\abs{x}$, the decay of $\pi_0(x)$ dominates any 
polynomially growing term.
\end{proof}

\begin{proof}[{\sc Proof of Proposition~\ref{prop:pin_h0}}]
Using again the completeness of the set of eigenfunctions, we have 
\begin{equation}
 \sum_{n\geqs0} \pscal{\pi_n}{h_0}^2
 = \sum_{n\geqs0} \pscal{\pi_0}{h_0\phi_n}\pscal{\pi_n}{h_0}
 = \pscal{\pi_0}{h_0^2} 
 = \pscal{\pi_0}{h_0} - \eta(y)\;.
\end{equation} 
At the same time, we also have 
\begin{align}
 \sum_{n=0}^1 \pscal{\pi_n}{h_0}^2
 &= \pscal{\pi_0}{h_0}^2 + \pscal{\pi_0}{\phi_1h_0}^2 \\[-8pt]
 &= \pscal{\pi_0}{h_0}^2 + \Bigbrak{\phi_-(y) \pscal{\pi_0}{h_0^2}
 + \phi_+(y)\pscal{\pi_0}{h_0(1-h_0)}}^2 
 \bigbrak{1+\Order{\lambda_1\ell}}\\
 &= \pscal{\pi_0}{h_0}^2 \bigbrak{1 + \e^{2\Deltabar(y)/\sigma^2}} 
 + \bigOrder{\eta(y)}\;.
\end{align} 
The result follows by subtracting the two sums, and using the 
definition~\eqref{eq:defDeltabar} of $\Deltabar(y)$. 
\end{proof}

\begin{proof}[{\sc Proof of Proposition~\ref{prop:f1mm1}}]
We will spell out the proofs in the case $i=1$, since the case $i=0$ is 
similar, though slightly easier. The first sum can be estimated by noting that 
\begin{equation}
 \sum_{m\geqs0} f_{1m}^2 
 = \sigma^4 \sum_{m\geqs0} \pscal{\pi_0}{\partial_y\phi_1\phi_m} 
\pscal{\pi_m}{\partial_y\phi_1}
 = \sigma^4 \pscal{\pi_0}{(\partial_y\phi_1)^2}
 = (\Deltabar')^2 + \Order{\lambda_1 \ell^2}\;,
\end{equation} 
where we have used~\eqref{eq:pi0dyphi12} in the last step.
Since Proposition~\ref{prop:f10} also yields 
\begin{equation}
 f_{10}^2 + f_{11}^2 = (\Deltabar')^2 + \Order{\lambda_1 \ell^2}\;,
\end{equation} 
we conclude that the first sum indeed has order $\lambda_1 \ell^2$. 
In the same spirit,
\begin{equation}
 \sum_{m\geqs0} f_{1m}f_{m1} 
 = -\sigma^4 \sum_{m\geqs0} \pscal{\partial_y\pi_1}{\phi_m} 
\pscal{\pi_m}{\partial_y\phi_1}
 = -\sigma^4 \pscal{\partial_y\pi_1}{\partial_y\phi_1}
 = A^2(\Deltabar')^2 + \Order{\lambda_1 \ell^2}
\end{equation} 
by~\eqref{eq:dypi1dyphi1}, while $f_{10}f_{01} + f_{11}^2 = A^2(\Deltabar')^2 + 
\Order{\lambda_1 \ell^2}$, showing the result for the second sum. 

Regarding the third sum, we use the decomposition $g_{1m} = -\ell_{1m} - 2 
k_{1m}$ given in~\eqref{eq:orth_fnm} and estimate separately the sums of 
squares of $\ell_{1m}$ and $k_{1m}$. Noting that 
\begin{equation}
 \sum_{m\geqs0} \ell_{1m}^2 
 = \sigma^8 \sum_{m\geqs0} \pscal{\pi_0}{\partial_{yy}\phi_1\phi_m} 
\pscal{\pi_m}{\partial_{yy}\phi_1}
 = \sigma^8 \pscal{\pi_0}{(\partial_{yy}\phi_1)^2}
\end{equation} 
and using~\eqref{eq:pi0dyyphi12}, we find that this sum is indeed equal to 
$\ell_{10}^2 + \ell_{11}^2 + \Order{\lambda_1\ell^3}$. As for the sum of 
$k_{1m}^2$, we note that~\eqref{eq:proj_dyphi1} implies 
\begin{align}
 k_{1m} &= 
 \sigma^2 \Deltabar' \pscal{\partial_y\pi_m}{A\phi_1+B}
 + \sigma^4 \pscal{\partial_y\pi_m}{R_1} \\
 &= \Deltabar' A f_{1m} + \sigma^2 \pscal{\partial_y\pi_m}{R_1}\;.
 \label{eq:k1m} 
\end{align}
We have already bounded $\sum_{m\geqs2} f_{1m}^2$, and the sum involving the 
error term $R_1$ can be bounded using Lemma~\ref{lem:sum_dypin} 
and~\eqref{eq:boundsR}. The proof is similar for the other sums.  
\end{proof}

In order to prove Proposition~\ref{prop:f1mm1lambdam}, we introduce two linear 
operators $\Pi_\perp$ and $\sL_\perp^{-1}$ defined by 
\begin{align}
 \bigpar{\Pi_\perp f}(x) 
 &= \sum_{m\geqs 2} \pscal{\pi_m}{f} \phi_m(x)\;, \\
 \bigpar{\sL_\perp^{-1} f}(x) 
 &= -\sum_{m\geqs 2} \frac{1}{\lambda_m}\pscal{\pi_m}{f} \phi_m(x)\;. 
\end{align}
The operator $\Pi_\perp$ is the projection on the complement of the span of 
$\phi_0$ and $\phi_1$, while $\sL_\perp^{-1}$ is the Green function of $\sL_x$ 
restricted to this complement. Note that $\sL_\perp^{-1} = 
\sL_\perp^{-1}\Pi_\perp = \Pi_\perp\sL_\perp^{-1}$. 

\begin{lemma}
\label{lem:Lperp}
Let $\sG_0$ be the Green function with Dirichlet boundary conditions, given 
by~\eqref{eq:Lx_inverse} for $x\in \cD=(x^*_-,x^*_+)$, and by  
\begin{equation}
 \bigpar{\sG_0f}(x) = 
 \begin{dcases}
  -\frac{2}{\sigma^2} \int_x^{x^*_-} \e^{2V_0(x_1)/\sigma^2}
  \int_{-\infty}^{x_1} \e^{-2V_0(x_2)/\sigma^2} f(x_2) \6x_2 \6x_1
  &\text{if $x<x^*_-$\;,} \\
  -\frac{2}{\sigma^2} \int_{x^*_+}^x \e^{2V_0(x_1)/\sigma^2}
  \int_{x_1}^{\infty} \e^{-2V_0(x_2)/\sigma^2} f(x_2) \6x_2 \6x_1
  &\text{if $x>x^*_+$\;.}  
 \end{dcases}
\end{equation} 
Then we have the representation 
\begin{equation}
\label{eq:Lperp_inverse} 
 \bigpar{\sL_\perp^{-1} f}(x) 
 = f_- h_0(x) + f_+ (1-h_0(x)) + \bigpar{\sG_0\Pi_\perp f}(x)\;,
\end{equation}
where the boundary values $f_\pm$ are given by 
\begin{align}
 f_- &= -\pscal{\pi_0}{\sG_0 \Pi_\perp f} 
 - \e^{\Deltabar/\sigma^2} \pscal{\pi_1}{\sG_0 \Pi_\perp f} 
\bigbrak{1+\Order{\lambda_1\ell}}\;, \\
 f_+ &= -\pscal{\pi_0}{\sG_0 \Pi_\perp f} 
 + \e^{-\Deltabar/\sigma^2} \pscal{\pi_1}{\sG_0 \Pi_\perp f} 
\bigbrak{1+\Order{\lambda_1\ell}}\;.
\end{align}
\end{lemma}
\begin{proof}
We view $\sL_\perp^{-1} f$ as the solution, on each of the intervals 
$(-\infty,x^*_-)$, $\cD$ and $(x^*_+,\infty)$, of a Dirichlet--Poisson problem 
similar to~\eqref{eq:Poisson_problem}, but with boundary values $f_\pm$. 
The expression~\eqref{eq:Lperp_inverse} is checked in the same way as in 
Lemma~\ref{lem:Poisson}, recalling that $h_0$ is constant outside $\cD$. The 
boundary values $f_\pm$ follow from the conditions 
$\pscal{\pi_0}{\sL_\perp^{-1} f} = \pscal{\pi_1}{\sL_\perp^{-1} f} = 0$, which 
are equivalent to the linear system 
\begin{equation}
 \begin{pmatrix}
  \pscal{\pi_0}{h_0} & \pscal{\pi_0}{1-h_0} \\ 
  \pscal{\pi_1}{h_0} & \pscal{\pi_1}{1-h_0}
 \end{pmatrix}
 \begin{pmatrix}
  f_- \\ f_+ 
 \end{pmatrix}
 = - 
 \begin{pmatrix}
  \pscal{\pi_0}{\sG_0 \Pi_\perp f} \\ \pscal{\pi_1}{\sG_0 \Pi_\perp f}
 \end{pmatrix}\;.
\end{equation} 
Solving this system, using~\eqref{eq:defDeltabar} and the fact that 
\begin{equation}
 \pscal{\pi_1}{h_0} = -\pscal{\pi_1}{1-h_0}
 = \frac{1+\Order{\lambda_1\ell}}{\e^{\Deltabar/\sigma^2} + 
\e^{-\Deltabar/\sigma^2}}
\end{equation} 
as a consequence of Propositions~\ref{prop:h0}, \ref{prop:h1} 
and~\eqref{eq:phi1_rep} yields the result. 
\end{proof}

\begin{proof}[{\sc Proof of Proposition~\ref{prop:f1mm1lambdam}}]
The first sum can be written 
\begin{equation}
 S_1 := \sum_{m\geqs2} \frac{1}{\lambda_m} f_{1m}f_{m1}
 = \sigma^4 \pscal{\partial_y\pi_1}{\sL_\perp^{-1}\partial_y\phi_1}\;.
\end{equation} 
Applying Lemma~\ref{lem:Lperp} and using the 
representation~\eqref{eq:proj_dyphi1} of $\partial_y\phi_1$, we obtain 
\begin{align}
 S_1 
 &= \sigma^4 (f_- - f_+) \pscal{\partial_y\pi_1}{h_0}
 + \sigma^4 \pscal{\partial_y\pi_1}{\sG_0\Pi_\perp R_1} \\
 &= -\sigma^4 (\e^{\Deltabar/\sigma^2} + \e^{-\Deltabar/\sigma^2})
 \pscal{\partial_y\pi_1}{h_0} \pscal{\pi_1}{\sG_0\Pi_\perp R_1} 
 + \sigma^4 \pscal{\partial_y\pi_1}{\sG_0\Pi_\perp R_1}\;.
\end{align}
By the expressions~\eqref{eq:phi1_rep} of $\phi_1$, we have
\begin{equation}
 f_{11} = \sigma^2\pscal{\partial_y\pi_1}{\phi_1}
 = \sigma^2 (\e^{\Deltabar/\sigma^2} + \e^{-\Deltabar/\sigma^2}) 
 \pscal{\partial_y\pi_1}{h_0}[1 + \Order{\lambda_1\ell}]
 + \Order{\lambda_1\ell}\;,
\end{equation} 
which yields
\begin{equation}
 \sigma^2 (\e^{\Deltabar/\sigma^2} + \e^{-\Deltabar/\sigma^2}) 
 \pscal{\partial_y\pi_1}{h_0}
 = f_{11} + \Order{\lambda_1\ell}
 = -A\Delta' + \Order{\lambda_1\ell}\;.
\end{equation} 
Using the fact that $\partial_y\pi_1 = \partial_y\pi_0\phi_1 + 
\pi_0\partial_y\phi_1$ and the expression~\eqref{eq:defW} for 
$\partial_y\pi_0$, we arrive at 
\begin{align}
 S_1 &= \sigma^2 
 \bigpscal{2\pi_1 \bigpar{W+\Deltabar'A + \Order{\lambda_1\ell^2}} + 
\pi_0\bigpar{\Deltabar'B + R_1}}{\sG_0\Pi_\perp R_1} \\
 &\lesssim \pscal{\pi_0}{\abs{\sG_0\Pi_\perp R_1}} 
 + \pscal{\abs{\pi_1}}{\abs{\sG_0\Pi_\perp R_1}}\;.
\end{align} 
It remains to estimate $\sG_0\Pi_\perp R_1$. The remainder $R_1$ is a sum of 
several terms, but the leading contribution comes from 
$(\phi_--\phi_+)\sigma^2\partial_y h_0$. We have 
\begin{equation}
 \Pi_\perp(\sigma^2\partial_y h_0) 
 = \sigma^2\partial_y h_0 + c_0 + c_1\phi_1\;,
\end{equation} 
where $c_0 = -\pscal{\pi_0}{\sigma^2\partial_y h_0} = \Order{\lambda_1\ell 
B^2}$ and $c_1 = -\pscal{\pi_0}{\sigma^2\partial_y h_0} = \Order{\lambda_1\ell 
B}$. By Lemma~\ref{lem:Poisson}, we obtain 
\begin{equation}
 \bigabs{\sG_0\Pi_\perp (\sigma^2\partial_y h_0)}
 \lesssim h_0(1-h_0) + c_0 \ell + c_1\ell\abs{\phi_1}\;.
\end{equation} 
Thanks to Remark~\ref{rem:pi0absphi1}, we conclude that 
\begin{align}
 \pscal{\pi_0}{\abs{\sG_0\Pi_\perp (\sigma^2\partial_y h_0)}} 
 &\lesssim \lambda_1 \ell^2 B^2 \;, \\
 \pscal{\abs{\pi_1}}{\abs{\sG_0\Pi_\perp (\sigma^2\partial_y h_0)}} 
 &\lesssim \lambda_1 \ell^2 B \;. 
\end{align} 
After estimating the other terms of $R_1$, we arrive at the bound 
$S_1 \lesssim \lambda_1\ell^2$. 

The second sum can be written 
\begin{equation}
 \sum_{m\geqs2} \frac{1}{\lambda_m} f_{1m}g_{m1}
 = \sigma^6 \pscal{\partial_{yy}\pi_1}{\sL_\perp^{-1}\partial_y\phi_1}\;,
\end{equation} 
and can be estimated in a similar way, expressing $\partial_{yy}\pi_1$ in 
terms of $\partial_y\phi_1$ and $\partial_{yy}\phi_1$, where the 
latter can be written in terms of $\phi_1$ and a 
remainder using~\eqref{eq:dyyphi1_rep}. 

The third sum can be written, using~\eqref{eq:k1m}, as
\begin{equation}
 -\sum_{m\geqs2} \frac{1}{\lambda_m} \bigbrak{\ell_{1m}+2k_{1m}}f_{m1}
 = \sigma^6 \pscal{\partial_y\pi_1}{\sL_\perp^{-1}\partial_{yy}\phi_1} 
 + 2\sigma^2 S_1 
 - 2\sigma^6\sum_{m\geqs2} \frac{f_{m1}}{\lambda_m} 
  \pscal{\partial_y\pi_m}{R_1}\;,
\end{equation} 
where the last sum can be estimated via the Cauchy--Schwarz inequality. The 
case of the last sum is similar. 
\end{proof}

\begin{proof}[{\sc Proof of Proposition~\ref{prop:Phistar}}]
The expression~\eqref{eq:def_alphaperp} of $\alpha_\perp^*$ can be rewritten as 
\begin{equation}
 \alpha_n^*(y) = -\eps\frac{1}{\lambda_n} \pscal{\partial_y\pi_0 + 
\alpha_1\partial_y\pi_1}{\phi_n}\;.
\end{equation} 
Therefore, we have 
\begin{equation}
 \pscal{\pi_0}{\Phi_\perp^* f} 
 = \sum_{n\geqs2} \alpha_n^*(y) \pscal{\pi_n}{f} 
 = -\eps \pscal{\partial_y\pi_0 + 
\alpha_1\partial_y\pi_1}{\sL_\perp^{-1}f}\;.
\end{equation} 
This quantity can be estimated in a similar way as $S_1$ in the previous proof, 
by noting that the only thing that really matters is the fact that $f$ can be 
bounded by a constant times $h_0(1-h_0)$. 
\end{proof}




\section{Laplace asymptotics}
\label{app:Laplace} 

In this appendix, we gather a few standard results on Laplace asymptotics, 
which can be obtained from those in~\cite{Olver_book}.

\begin{lemma}
\label{lem:Laplace} 
Let $f\in\sC^2(\R,\R)$ satisfy the following conditions: 
\begin{itemize}
\item 	$\abs{f(x)}$ has at most polynomial growth for large $x$;
\item 	$f$ is bounded away from $0$ in neighbourhoods $I_\pm$ of $x^*_-(y)$ 
and $x^*_+(y)$, whose size does not depend in $\sigma$; 
\item 	$f'/f$ and $f''/f$ are bounded uniformly in $\sigma$ in $I_\pm$.
\end{itemize}
Then 
\begin{equation}
 \pscal{\pi_0}{f} 
 = \frac{f(x^*_-(y)) \e^{-\Deltabar(y)/\sigma^2} + f(x^*_+(y)) 
\e^{\Deltabar(y)/\sigma^2}}{\e^{-\Deltabar(y)/\sigma^2} 
+ \e^{\Deltabar(y)/\sigma^2 }} \bigbrak{1 + \Order{\sigma^2}}\;.
\end{equation} 
\end{lemma}
\begin{proof}
Using the change of variables $x = x^*_\pm(y) + \sigma z/(\sqrt 2 
\omega_-(y))$, 
one obtains 
\begin{align}
 &\int_{I_\pm} \e^{-2V_0(x,y)/\sigma^2} f(x)\6x \\
 &= \frac{\sigma}{\sqrt2\omega_-(y)} 
 \int_{\tilde I_\pm} \biggbrak{f(x^*_\pm) 
 + \frac{\sigma}{\sqrt2\omega_-(y)} z f'(x^*_-(y)) 
 + \frac{\sigma^2}{4\omega_-(y)^2} z^2 f''(x^*_\pm(y)+\theta)} \e^{-z^2/2}\6z\\
 &= \frac{\sigma\sqrt{\pi}}{\omega_-(y)} f(x^*_\pm(y))[1 + \Order{\sigma^2}]\;,
 \label{eq:Laplace_proof} 
\end{align} 
where $\tilde I_\pm = \sqrt2\omega_\pm(y)(x-x^*_\pm(y))/\sigma$. 
Furthermore, the integral over $\R\setminus(I_-\cup I_+)$ is negligible with 
respect to the sum of these two integrals. The result then follows from 
applying~\eqref{eq:Laplace_proof} first to $f=1$ to estimate $Z_0(y)$, and then 
to general $f$ satisfying the stated assumptions. 
\end{proof}

We will also need estimates involving the integral of $\e^{2V_0/\sigma^2}$ 
against a function vanishing polynomially at $x = x^*_0$. To ease notation, we 
will assume that $x^*_0 = 0$, and write  
\begin{equation}
 V_0(x) = -\frac12\omega_0^2 x^2 + W(x) 
\end{equation} 
where $W(x) = \Order{x^3}$. Consider the integrals 
\begin{align}
 I_n &= \int_{-\delta}^\delta x^n \e^{2V_0(x)/\sigma^2} \6x\;, \\
 J_n(x) &= \e^{-2V_0(x)/\sigma^2} \int_x^\delta x_1^n \e^{2V_0(x_1)/\sigma^2} 
\6x_1\;,
\end{align}
where $n\in\N_0$ and $\delta$ has order $1$. 

\begin{lemma}
\label{lem:Laplace2} 
We have the asymptotics 
\begin{equation}
\label{eq:In}
I_n = 
\begin{cases}
\Gamma\biggpar{\dfrac{n+1}{2}} \dfrac{\sigma^{n+1}}{\omega_0^{n+1}}
\bigbrak{1+\Order{\sigma^2}}
&\text{if $n$ is even\;,} \\[2ex]
\Order{\sigma^{n+2}}
&\text{if $n$ is odd\;.} 
\end{cases}
\end{equation} 
\end{lemma}
\begin{proof}
The case of even $n$  follows from a direct application 
of~\cite[Theorem~8.1]{Olver_book}, where the fact that the 
error has order $\sigma^2$ is due to the leading term of $W$ being odd. 
When $n$ is odd, we use integration by parts to obtain 
\begin{equation}
\label{eq:ibp_proof} 
 I_n = -\frac{\sigma^2}{2\omega_0^2} \e^{-\omega_0^2 x^2/\sigma^2} 
x^{n-1}\e^{2W(x)/\sigma^2} \biggr\vert_{-\delta}^\delta 
+ \frac{\sigma^2}{2\omega_0^2} \int_{-\delta}^\delta \e^{-\omega_0^2 
x^2/\sigma^2} 
\dtot{}{x} \Bigbrak{x^{n-1}\e^{2W(x)/\sigma^2}}\6x\;.
\end{equation} 
If $n=1$, the integral has order $\sigma^3$ by~\eqref{eq:In} with $n=2$, while 
the boundary terms are negligible. For $n\geqs 3$, we 
obtain $I_n = \Order{\sigma^2I_{n-1}} + \Order{\sigma^{n+2}}$, so that the 
result follows by induction. 
\end{proof}

In particular, we have 
\begin{equation}
 I_0 = \frac{\sqrt{\pi}}{\omega_0} \sigma \bigbrak{1+\Order{\sigma^2}}\;, 
 \qquad 
 I_2 = \frac{\sqrt{\pi}}{2\omega_0^3} \sigma^3
 \bigbrak{1+\Order{\sigma^2}}\;.
\end{equation} 

\begin{lemma}
\label{lem:Laplace3} 
There is a constant $M>0$, independent of $\sigma$ and $\delta$, such that 
\begin{equation}
\label{eq:Jn} 
\bigabs{J_n(x)} 
\leqs M\sigma^2(\sigma + \abs{x})^{n-1}
\end{equation} 
holds for any $n\in\N$ and any $x\in[-\delta,\delta]$.
In particular, 
\begin{align}
 \label{eq:J1} 
 J_1(x) &= \frac{\sigma^2}{2\omega_0^2} 
 + \bigOrder{\sigma^2(\abs{x}+\sigma)}\;, \\
 J_3(x) &= \frac{\sigma^4}{2\omega_0^4} 
 + \bigOrder{\sigma^2(x^2+\sigma^3)}\;.
 \label{eq:J3} 
\end{align}
\end{lemma}
\begin{proof}
For $x=0$ and for $\abs{x}$ of order $1$,~\eqref{eq:Jn} follows 
from~\cite[Theorem~8.1]{Olver_book}. For intermediate $x$, we can use the fact 
that 
\begin{align}
 \frac{\sigma^2}{2} J_n'(x) 
 &= -\partial_x V_0 J_n(x) - \frac{\sigma^2}{2} x^n \\
 &= \bigbrak{\omega_0^2 x - \partial_x W(x)}J_n(x) - \frac{\sigma^2}{2} x^n\;,
\end{align} 
whose right-hand side vanishes for $J_n(x) = J^\star_n(x) = \Order{\sigma^2 
x^{n-1}}$. Since $J_n(x)-J_n^*(x)$ is a decreasing function of $x$ for $x<0$, 
and~\eqref{eq:Jn} is satisfied for negative $x$ of order $1$, it holds for all 
$x\in[0,\delta]$. A similar argument applies for $x>0$ by changing $x$ into 
$-x$. To prove~\eqref{eq:J1}, we use a similar integration-by-parts argument as 
in~\eqref{eq:ibp_proof} to obtain  
\begin{equation}
\e^{2V_0(x)/\sigma^2}J_1(x) 
= \frac{\sigma^2}{2\omega_0^2} 
\bigbrak{\e^{2V_0(x)/\sigma^2} - \e^{2V_0(\delta)/\sigma^2}}
+ \frac{1}{\omega_0^2} \int_x^\delta \partial_x W 
\e^{2V_0(x_1)/\sigma^2}\6x_1\;.
\end{equation}
The integral on the right-hand side can be bounded using~\eqref{eq:Jn} with 
$n=2$, while the term $\e^{2V_0(\delta)/\sigma^2}$ is negligible. Finally, 
\eqref{eq:J3} follows from the integration-by-parts 
relation 
\begin{equation}
\e^{2V_0(x)/\sigma^2}J_1(x) 
= -\frac{x^2}{2} 
\bigbrak{\e^{2V_0(x)/\sigma^2} - \e^{2V_0(\delta)/\sigma^2}}
- \frac{1}{\sigma^2} \int_x^\delta x_1^2 \partial_x V_0 
\e^{2V_0(x_1)/\sigma^2}\6x_1\;,
\end{equation}
expressing the integral on the right-hand side in terms of $J_3(x)$. 
\end{proof}


\bibliographystyle{plain}
{\small \bibliography{KPref}}

\newpage
{\small \tableofcontents}

\vfill

\bigskip\bigskip\noindent
{\small
Institut Denis Poisson (IDP) \\ 
Universit\'e d'Orl\'eans, Universit\'e de Tours, CNRS -- UMR 7013 \\
B\^atiment de Math\'ematiques, B.P. 6759\\
45067~Orl\'eans Cedex 2, France \\
{\it E-mail address: }
{\tt nils.berglund@univ-orleans.fr}

\end{document}